\documentclass[12pt]{amsart}

\usepackage[a4paper,left=3.5cm,right=3.5cm]{geometry}
\usepackage{amsmath}
\usepackage{latexsym}
\usepackage{amsfonts}
\usepackage{amssymb}
\usepackage{graphicx}
\usepackage{mathdots}
\usepackage{multicol}
\usepackage{pdfsync}
\usepackage{hyperref}
\usepackage{tikz}
\usetikzlibrary[arrows,snakes]
\tikzstyle{bull}=[circle,draw=black,fill=black!80]
\tikzstyle{holl}=[circle,draw=black]

\def\dueq{\stackrel{\mathrm d}\sim}

\newtheorem{theorem}{Theorem}[section]
\newtheorem{lemma}[theorem]{Lemma}
\newtheorem{proposition}[theorem]{Proposition}
\newtheorem{corollary}[theorem]{Corollary}
\newtheorem{definition}[theorem]{Definition}

\theoremstyle{definition}

\newtheorem{example}[theorem]{Example}

\newtheorem{remark}[theorem]{Remark}

\newcommand{\la}{\langle}
\newcommand{\ra}{\rangle}
\newcommand{\nR}{R\!\!\!\!\! \diagup}

\let\Omega=\varOmega
\let\Gamma=\varGamma
\let\Lambda=\varLambda
\let\Sigma=\varSigma
\let\Phi=\varPhi
\let\Psi=\varPsi
\let\Delta=\varDelta
\let\Pi=\varPi
\let\Theta=\varTheta
\let\phi=\varphi
\let\psi=\varepsilon

\begin{document}

\title{Duality for distributive and implicative semi-lattices}

\author{Guram Bezhanishvili and Ramon Jansana}

\date{}


\maketitle

{\bf Note 2023:} 
\color{blue} 
This paper dates back to 2008 when it appeared on the Barcelona Logic Group webpage 
http://www.ub.edu/grlnc/docs/BeJa08-m.pdf. 
Subsequently it was published in three installments:
\begin{itemize}
\item G.~Bezhanishvili and R.~Jansana, \emph{Generalized Priestley quasi-orders}. Order {\bf 28} (2011), no. 2, 201--220.
\item G.~Bezhanishvili and R.~Jansana, \emph{Priestley style duality for distributive meet-semilat\-tices}. Studia Logica {\bf 98} (2011), no. 1-2, 83--122.
\item G.~Bezhanishvili and R.~Jansana, \emph{Esakia style duality for implicative semilattices}. Appl. Categ. Structures {\bf 21} (2013), no. 2, 181--208.
\end{itemize}  
Recently the old Barcelona Logic Group webpage ceased to exist, and we were encouraged to upload the paper to arXiv. We have made no corrections to the original version, but the following is worth mentioning. 

Recently it came to our attention that the construction of distributive envelopes goes back to the work of Cornish and Hickman (1978) in the more general setting of weakly distributive meet-semilattices: 
\begin{itemize}
\item W. H.~Cornish and R. C.~Hickman, \emph{Weakly distributive semilattices. Acta Math. Acad. Sci. Hungar}. {\bf 32} (1978), no. 1-2, 5--16.
\end{itemize}
The following recent paper connects our duality to Hofmann-Mislove-Stralka duality for semilattices, as well as provides a more detailed study of the non-bounded case:
\begin{itemize}
\item G. Bezhanishvili, L. Carai, P. J. Morandi, \emph{Connecting generalized Priestley duality to Hofmann-Mislove-Stralka duality}, https://arxiv.org/pdf/2207.13938.pdf, 2022.
\end{itemize} 
\color{black}

\begin{abstract}
We develop a new duality for distributive and implicative meet
semi-lattices. For distributive meet semi-lattices our duality
generalizes Priestley's duality for distributive lattices and
provides an improvement of Celani's duality. Our generalized
Priestley spaces are similar to the ones constructed by Hansoul.
Thus, one can view our duality for distributive meet semi-lattices
as a completion of Hansoul's work. For implicative meet
semi-lattices our duality generalizes Esakia's duality for Heyting
algebras and provides an improvement of Vrancken-Mawet's and
Celani's dualities. In the finite case it also yield's K\"ohler's
duality. Thus, one can view our duality for implicative meet
semi-lattices as a completion of K\"ohler's work. As a
consequence, we also obtain a new duality for Heyting algebras,
which is an alternative to the Esakia duality.
\end{abstract}

\tableofcontents

\section{Introduction}

Topological representation of distributive semi-lattices goes back
to Stone's pioneering  work \cite{Sto37}. For distributive join
semi-lattices with bottom it was worked out in detail in Gr\"atzer
\cite[Section II.5, Theorem 8]{Gra98}. A full duality between meet
semi-lattices with top (which are dual to join semi-lattices with
bottom) and certain ordered topological spaces was developed by
Celani \cite{Cel03b}. The main novelty of \cite{Cel03b} was the
characterization of meet semi-lattice homomorphisms preserving top
by means of certain binary relations. But the ordered topological
spaces of \cite{Cel03b} are rather difficult to work with, which is
the main drawback of the paper. On the other hand, Hansoul
\cite{Han03} developed rather nice order topological duals of
bounded join semi-lattices, but had no dual analogue of bounded join
semi-lattice homomorphisms. It is our intention to develop duality
for semi-lattices that improves both \cite{Cel03b} and \cite{Han03}.
Like in \cite{Cel03b}, we work with meet semi-lattices which are
dual to join semi-lattices. We generalize the notion of a Priestley
space to that of a generalized Priestley space and develop duality
between the category of bounded distributive meet semi-lattices and
meet semi-lattice homomorphisms and the category of generalized
Priestley spaces and generalized Priestley morphisms. In the
particular case of bounded distributive lattices, our duality yields
the well-known Priestley duality \cite{Pri70,Pri72}.

The first duality for finite implicative meet semi-lattices was
given by K\"ohler \cite{Koh81}. It was extended to the infinite case
by Vrancken-Mawet \cite{VM86} and Celani
\cite{Cel03a}.\footnote{Celani's paper contains a gap, which we
correct at the end of Section 11.3.} The Vrancken-Mawet and Celani
dualities are in terms of spectral-like ordered topological spaces.
Both topologies are not Hausdorff, thus rather difficult to work
with. We develop a new duality for bounded implicative meet
semi-lattices that improves both Vrancken-Mawet's and Celani's
dualities. It is obtained as a particular case of our duality for
bounded distributive meet semi-lattices. We generalize the notion of
an Esakia space to that of a generalized Esakia space and develop
duality between the category of bounded implicative meet
semi-lattices and implicative meet semi-lattice homomorphisms and
the category of generalized Esakia spaces and generalized Esakia
morphisms. In the particular case of Heyting algebras, our duality
yields the well-known Esakia duality \cite{Esa74,Esa85}. Moreover,
it provides a new duality for Heyting algebras, which is an
alternative to the Esakia duality.

The paper is organized as follows. In Section 2 we provide all the
needed information to make the paper self-contained. In particular,
we recall basic facts about meet semi-lattices, distributive meet
semi-lattices, and implicative meet semi-lattices, as well as the
basics of Priestley's duality for bounded distributive lattices and
Esakia's duality for Heyting algebras. In Section 3 we consider
filters and ideals, as well as prime filters and prime ideals of a
distributive meet semi-lattice and develop their theory. In Section
4 we introduce some of the main ingredients of our duality such as
distributive envelops, Frink ideals, and optimal filters and give a
detailed account of their main properties. The introduction of
optimal filters is one of the main novelties of the paper. There are
more optimal filters than prime filters of a bounded distributive
meet semi-lattice $L$, and it is optimal filters and \emph{not}
prime filters that serve as points of the dual space of $L$, which
allows us to prove that the Priestley-like topology of the dual of
$L$ is compact, thus providing improvement of the previous work
\cite{Gra98,Koh81,VM86,Cel03b,Cel03a}, in which the dual of $L$ was
constructed by means of prime filters of $L$. In Section 5 we introduce a new class of homomorphisms between distributive meet semi-lattices, we call sup-homomorphisms, and provide an abstract characterization of distributive envelopes by means of sup-homomorphisms. In Section 6 we
introduce generalized Priestley spaces, prove their main properties,
and provide a representation theorem for bounded distributive meet
semi-lattices by means of generalized Priestley spaces. In Section 7
we introduce generalized Esakia spaces and provide a representation
theorem for bounded implicative meet semi-lattices by means of
generalized Esakia spaces. In Section 8 we introduce generalized
Priestley morphisms and show that the category of bounded
distributive meet semi-lattices and meet semi-lattice homomorphisms
is dually equivalent to the category of generalized Priestley spaces
and generalized Priestley morphisms. We also introduce generalized
Esakia morphisms and show that the category of bounded implicative
meet semi-lattices and implicative meet semi-lattice homomorphisms
is dually equivalent to the category of generalized Esakia spaces
and generalized Esakia morphisms. In Section 9 we show that the subclasses of generalized Priestley morphisms which dually correspond to sup-homomorphisms can be characterized by means of special functions between generalized Priestley spaces, we call strong Priestley morphisms, and show that the same also holds in the category of generalized Esakia spaces. In Section 10 we prove that in the
particular case of bounded distributive lattices our duality yields
the well-known Priestley duality, and that in the particular case of
Heyting algebras it yields the Esakia duality. We also give an
application to modal logic by showing that descriptive frames, which
are duals of modal algebras, can be thought of as generalized
Priestley morphisms of Stone spaces into themselves. Moreover, we
introduce partial Esakia functions between Esakia spaces and show
that generalized Esakia morphisms between Esakia spaces can be
characterized in terms of partial Esakia functions. Furthermore, we
show that Esakia morphisms can be characterized by means of special
partial Esakia functions, we call partial Heyting functions, thus
yielding a new duality for Heyting algebras, which is an alternative
to the Esakia duality. We conclude the section by showing that
partial functions are not sufficient for characterizing neither meet
semi-lattice homomorphisms between bounded distributive lattices nor
implicative meet semi-lattice homomorphisms between bounded implicative
meet semi-lattices. In Section 11 we show how our duality works by giving
dual descriptions of Frink ideals, ideals, and filters, as well as
1-1 and onto homomorphisms.
In Section 12 we show how to adjust our technique to handle the
non-bounded case. Finally, in Section 13 we give a detailed
comparison of our work with the relevant work by Gr\"atzer
\cite{Gra98}, K\"ohler \cite{Koh81}, Vrancken-Mawet \cite{VM86},
Celani \cite{Cel03b,Cel03a}, and Hansoul \cite{Han03}, and correct a
mistake in \cite{Cel03a}.

\section{Preliminaries}

In this section we recall basic facts about meet semi-lattices,
distributive meet semi-lattices, and implicative meet semi-lattices.
We also recall the basics of Priestley's duality for bounded
distributive lattices and Esakia's duality for Heyting algebras.

We start by recalling that a \textit{meet semi-lattice} is a
commutative idempotent semigroup $\langle S,\cdot\rangle$. For a
given meet semi-lattice $\langle S,\cdot\rangle$, we denote $\cdot$
by $\wedge$, and define a partial order $\leq$ on $S$ by $a\leq b$
iff $a=a\wedge b$. Then $a\wedge b$ becomes the greatest lower bound
of $\{a,b\}$ and $\langle S,\wedge\rangle$ can be characterized as a
partially ordered set $\langle S,\leq\rangle$ such that every
nonempty finite subset of $S$ has a greatest lower bound. Below we
will be interested in meet semi-lattices with the greatest element
$\top$, i.e., in commutative idempotent monoids $\langle
M,\wedge,\top\rangle$. Let $\mathsf{M}$ denote the category of meet
semi-lattices with $\top$ and meet semi-lattice homomorphisms
preserving $\top$; that is $\mathsf{M}$ is the category of
commutative idempotent monoids and monoid homomorphisms.

It is obvious that meet semi-lattices serve as a natural
generalization of lattices: if $\la L, \wedge, \vee \ra$ is a
lattice, then $\la L, \wedge\ra$ is a meet semi-lattice.
Distributive meet semi-lattices serve as a natural generalization of
distributive lattices. We recall that a meet semi-lattice $L$ is
\textit{distributive} if for each $a, b_1, b_2 \in L$ with $b_1
\wedge b_2 \leq a$, there exist $c_1, c_2 \in L$ such that $b_1 \leq
c_1$, $b_2 \leq c_2$, and $a = c_1 \wedge c_2$. As follows from
\cite[Section II.5, Lemma 1]{Gra98}, a lattice $\la L, \wedge, \vee
\ra$ is distributive iff the meet semi-lattice $\la L, \wedge\ra$ is
distributive. Let $\mathsf{DM}$ denote the category of distributive
meet semi-lattices with $\top$ and meet semi-lattice homomorphisms
preserving $\top$. Obviously $\mathsf{DM}\subset\mathsf{M}$.
However, unlike lattices, there is no meet semi-lattice identity
expressing the difference, since the variety $\mathsf{M}$ is
generated by the class $\mathsf{DM}$. In fact, we have that the two
element meet semi-lattice $\mathbf{2}$ generates $\mathsf{M}$
\cite{HK71}. We list some of the interesting properties of the
variety $\mathsf{M}$. Since they are not directly related to our
main purpose, we don't provide any proofs. Some of the proofs can be
found in \cite{HK71,Gra98}, the rest we leave as interesting
exercises. The variety $\mathsf{M}$ is semi-simple (with
$\mathbf{2}$ being a unique up to isomorphism nontrivial simple meet
semi-lattice), finitely generated, has no proper non-trivial
subvarieties, has the amalgamation property, and each epimorphism in
$\mathsf{M}$ is surjective. On the other hand, $\mathsf{M}$ is
neither congruence-distributive nor has the congruence extension
property.

We recall that $M\in\mathsf{M}$ is an \emph{implicative meet
semi-lattice} if for each $a\in M$ the order-preserving map
$a\wedge(-):M\to M$ has a right adjoint, denoted by $a\rightarrow
(-)$. If in addition $M$ is a bounded lattice, then $M$ is called a
\emph{Heyting algebra}. The same way meet semi-lattices and
distributive meet semi-lattices serve as a generalization of
lattices and distributive lattices, respectively, implicative meet
semi-lattices serve as a generalization of Heyting algebras: the
$(\wedge, \rightarrow, \top)$-reduct of a Heyting algebra is an
implicative meet semi-lattice.

For two implicative meet semi-lattices $L$ and $K$, we recall that a
map $h:L\to K$ is an {\it implicative meet semi-lattice
homomorphism} if $h$ is a meet semi-lattice homomorphism and $h(a\to
b)=h(a) \to h(b)$ for each $a,b\in L$. Since in an implicative meet
semi-lattice we have that $a\to a=\top$, it follows that each
implicative meet semi-lattice homomorphism preserves $\top$. Let
$\mathsf{IM}$ denote the category of implicative meet semi-lattices
and implicative meet semi-lattice homomorphisms.

We observe that the same way the lattice reduct of a Heyting algebra
is distributive, the meet semi-lattice reduct of an implicative meet
semi-lattice is distributive. This result belongs to folklore. We
give a short proof of it for the lack of proper reference.

\begin{proposition}\label{prop}
The $\wedge$-reduct of an implicative meet semi-lattice is a
distributive meet semi-lattice.
\end{proposition}

\begin{proof}
Suppose that $L$ is an implicative meet semi-lattice and $a, b_1,
b_2\in L$ with $b_1 \wedge b_2 \leq a$. Then $b_1 \leq b_2
\rightarrow a$ and $b_2 \leq b_1 \rightarrow a$. Set $c = (b_1
\rightarrow a) \wedge (b_2 \rightarrow a)$. Then $a \leq c$, $c \leq
b_1 \to a$, and $c \leq b_2 \to a$. Hence, $b_1 \leq c \to a$ and
$b_2 \leq c \to a$. Set $c_1 = (c \rightarrow a) \wedge (b_2
\rightarrow a)$ and $c_2 = (c \rightarrow a) \wedge (b_1 \rightarrow
a)$. Since $b_1 \leq c\to a$ and $b_1 \leq b_2\to a$, we have $b_1
\leq c_1$. Similarly $b_2 \leq c_2$. Moreover, $c_1 \wedge c_2=(c
\rightarrow a) \wedge (b_1 \rightarrow a) \wedge (b_2 \rightarrow a)
= ((b_1 \rightarrow a) \wedge (b_2 \rightarrow a)) \wedge (((b_1
\rightarrow a) \wedge (b_2 \rightarrow a))\to a) = a$. Therefore,
there exist $c_1,c_2 \in L$ such that $b_1\leq c_1$, $b_2\leq c_2$,
and $a=c_1\wedge c_2$. Thus, $L$ is a distributive meet
semi-lattice.
\end{proof}

It follows from Proposition \ref{prop} that $\mathsf{IM}$ is a
subcategory of $\mathsf{DM}$. Nevertheless, as a variety
$\mathsf{IM}$ behaves differently from $\mathsf{DM}$. We list some
of the properties of $\mathsf{IM}$ without providing any proofs.
Some of the proofs can be found in \cite{Nem65,NW71,NW73,Koh81}, the
rest we leave as interesting exercises. First of all, the lattice
$\mathsf{F}(L)$ of filters of an implicative meet semi-lattice $L$
is isomorphic to the lattice $\Theta(L)$ of congruences of $L$. As a
result, an implicative meet semi-lattice $L$ is subdirectly
irreducible iff the set $L-\{\top\}$ has a greatest element. Thus,
$\mathsf{IM}$ is neither semi-simple nor finitely generated. In
fact, $\mathsf{IM}$ is not even residually small, and the
cardinality of the lattice of subvarieties of $\mathsf{IM}$ is that
of the continuum. On the other hand, $\mathsf{IM}$ is locally finite and
has equationally definable principal congruences. Hence,
$\mathsf{IM}$ is congruence-distributive and has the congruence
extension property. It also has the amalgamation property and each
epimorphism in $\mathsf{IM}$ is surjective.

We conclude this section with a brief overview of Priestley's
duality for bounded distributive lattices and Esakia's duality for
Heyting algebras. For a partially ordered set $\la X,\leq \ra$ and
$A\subseteq X$ let ${\uparrow}A=\{x \in X: \exists a\in A$ with
$a\leq x\}$ and ${\downarrow}A=\{x \in X: \exists a\in A$ with
$x\leq a\}$. If $A$ is the singleton $\{a\}$, then we write
${\uparrow}a$ and ${\downarrow}a$ instead of ${\uparrow}\{a\}$ and
${\downarrow}\{a\}$, respectively. We call $A$ an \emph{upset}
(resp.\ \emph{downset}) if $A={\uparrow}A$ (resp.\
$A={\downarrow}A$).

We recall that a \textit{Priestley space} is an ordered topological
space $X=\la X, \tau, \leq \ra$ which is compact and satisfies the
\textit{Priestley separation axiom}: if $x \not \leq y$, then there
is a \emph{clopen} (closed and open) upset $U$ of $X$ such that $x
\in U$ and $y \not \in U$. It follows from the Priestley separation
axiom that $X$ is in fact Hausdorff and that clopen sets form a
basis for the topology. Thus, each Priestley space is a \emph{Stone
space} (compact Hausdorff 0-dimensional).

For two Priestely spaces $X$ and $Y$, a morphism $f:X\to Y$ is a
\emph{Priestley morphism} if $f$ is continuous and
order-preserving. We denote the category of Priestely spaces and
Prietley morphisms by $\mathsf{PS}$. Let also $\mathsf{BDL}$
denote the category of bounded distributive lattices and bounded
lattice homomorphisms. Then we have that $\mathsf{BDL}$ is dually
equivalent to $\mathsf{PS}$ \cite{Pri70}. We recall that the
functors $(-)_*:\mathsf{DL}\to\mathsf{PS}$ and
$(-)^*:\mathsf{PS}\to\mathsf{DL}$ establishing the dual
equivalence are constructed as follows. If $L$ is a bounded
distributive lattice, then $L_*=\la X,\tau,\leq \ra$, where $X$ is
the set of prime filters of $L$, $\leq$ is set-theoretic
inclusion, and $\tau$ is the topology generated by the subbasis
$\{\phi(a):a\in L\}\cup \{\phi(b)^c:b\in L\}$, where
$\phi(a)=\{x\in X:a\in x\}$ is the Stone map. If $h\in
\operatorname{hom}(L,K)$, then $h_*=h^{-1}$. If $X$ is a Priestley
space, then $X^*$ is the lattice of clopen upsets of $X$, and if
$f\in\operatorname{hom}(X,Y)$, then $f^*=f^{-1}$. It follows from
\cite{Pri70,Pri72} that the functors $(-)_*$ and $(-)^*$ are
well-defined, and that they establish the dual equivalence of
$\mathsf{BDL}$ and $\mathsf{PS}$.

Esakia's duality is a restricted Priestley duality. We recall that
an \textit{Esakia space} is a Priestley space $X=\la X, \tau,
\leq\ra$ in which the downset of each clopen is again clopen. We
also recall that an \emph{Esakia morphism} from an Esakia space $X$
to an Esakia space $Y$ is a Priestley morphism $f$ such that for all
$x \in X$ and $y \in Y$, from $f(x) \leq y$ it follows that there is
$z \in X$ with $x \leq z$ and $f(z) = y$. We denote the category of
Esakia spaces and Esakia morphisms by $\mathsf{ES}$. Let also
$\mathsf{HA}$ denote the category of Heyting algebras and Heyting
algebra homomorphisms. Then it was established in Esakia
\cite{Esa74} that $\mathsf{HA}$ is dually equivalent to
$\mathsf{ES}$. In fact, the same functors $(-)_*$ and $(-)^*$,
restricted to $\mathsf{HA}$ and $\mathsf{ES}$, respectively,
establish the desired equivalence.

\section{Filters and ideals of meet semi-lattices}

In this section we adapt the notions of a filter and an ideal of a
lattice to meet semi-lattices. We also introduce the notion of a
meet-prime filter of a distributive meet semi-lattice, which is an
analogue of the notion of a prime filter of a distributive
lattice, and present an analogue of the prime filter lemma for
distributive meet semi-lattices. Most of these results are
well-known. We present them here to keep the presentation as
self-contained as possible.

Let $L$ be a meet semi-lattice. We recall that a nonempty subset $F$
of $L$ is a \textit{filter} if (i) $a, b \in F$ implies $a \wedge b
\in F$ and (ii) $a \in F$ and $a \leq b$ imply $b \in F$. Clearly
$F$ is a filter of $L$ iff for each $a,b\in L$ we have $a, b \in F$
iff $a \wedge b \in F$. Similar to lattices, we have that $L$ is a
filter of $L$, and if $\top\in L$, then an arbitrary intersection of
filters of $L$ is again a filter of $L$. Therefore, for each $X
\subseteq L$, there exists a least filter containing $X$, which we
call the \emph{filter generated by $X$} and denote by $[X)$. It is
obvious that
$$a\in [X) \mbox{ iff there exists a finite } Y\subseteq X
\mbox{ such that } \bigwedge Y \leq a.$$ In particular, the filter
generated by $x\in L$ is the upset ${\uparrow}x = \{y \in L: x \leq
y\}.$ We also point out that if $X=\emptyset$, then $\bigwedge
X=\top$, and so $[X)=\{\top\}$.

Let $\mathsf{F}(L)$ denote the set of filters of $L$. Obviously the
structure $\la \mathsf{F}(L), \cap, \vee\ra$ forms a lattice, where
$F_1\vee F_2=[F_1\cup F_2)$, and $a \in F_1 \vee F_2$ iff there
exist $a_1 \in F_1$ and $a_2 \in F_2$ such that $a_1 \wedge a_2 \leq
a$. In particular, ${\uparrow}a\vee {\uparrow}b ={\uparrow}(a \wedge
b)$. The following characterization of a distributive meet
semi-lattice $L$ in terms of the filter lattice of $L$ follows from
\cite[Section II.5, Lemma 1]{Gra98}: A meet semi-lattice $L$ is
distributive iff the lattice $\la \mathsf{F}(L), \cap, \vee\ra$ is
distributive. We call a filter $F$ of $L$ \emph{proper} if $F\neq
L$.

\begin{definition}
{\rm A proper filter $F$ of a meet semi-lattice $L$ is said to be
\textit{meet-prime} if for any two filters $F_1, F_2$ of $L$ with
$F_1 \cap F_2 \subseteq F$, we have $F_1 \subseteq F$ or $F_2
\subseteq F$.}
\end{definition}

In \cite{Cel03b} meet-prime filters are called \emph{weakly
irreducible} filters. Meet-prime filters serve as an obvious
generalization of prime filters of a lattice as the following
proposition shows.

\begin{proposition}
A filter $F$ of a lattice $L$ is prime iff $F$ is meet-prime.
\end{proposition}

\begin{proof}
Suppose that $F$ is prime and $F_1\cap F_2\subseteq F$. If neither
$F_1 \subseteq F$ nor $F_2\subseteq F$, then there exist $a_1\in
F_1$ and $a_2\in F_2$ with $a_1,a_2\notin F$. From $a_1\in F_1$ and
$a_2\in F_2$ it follows that $a_1\vee a_2\in F_1 \cap F_2\subseteq
F$. Since $F$ is prime, either $a_1\in F$ or $a_2\in F$, a
contradiction. Thus, $F_1\subseteq F$ or $F_2\subseteq F$.
Conversely, suppose that $F$ is meet-prime and $a\vee b\in F$. Since
${\uparrow}(a\vee b)={\uparrow}a\cap {\uparrow}b$, we obtain
${\uparrow}a\cap {\uparrow}b\subseteq F$. As $F$ is meet-prime, the
last inclusion implies ${\uparrow}a\subseteq F$ or ${\uparrow}b
\subseteq F$. Thus, $a\in F$ or $b\in F$, and so $F$ is prime.
\end{proof}

From now on we will call meet-prime filters of a meet semi-lattice
simply \textit{prime}. Since a meet semi-lattice $L$ may not be a
lattice and so the join of two elements of $L$ may not exist, the
notion of an ideal of $L$ needs to be adjusted appropriately. For
a subset $A$ of a meet semi-lattice $L$, let $A^u=\{u\in L:
\forall a\in A$ we have $a\leq u\}$ denote the set of \emph{upper
bounds} of $A$, and let $A^l=\{l\in L: \forall a\in A$ we have
$l\leq a\}$ denote the set of \emph{lower bounds} of $A$. We call
a nonempty subset $I$ of a meet semi-lattice $L$ an \textit{ideal}
if (i) $a\in I$ and $b\leq a$ imply $b\in I$ and (ii) $a,b\in I$
implies $\{a,b\}^u\cap I\neq\emptyset$. Our notion of ideal is
dual to the notion of \emph{dual ideal} of \cite[p.\ 132]{Gra98}.
In \cite{Cel03b} ideals are called \emph{order ideals}.

\begin{remark}
{\rm Condition (ii) of the definition of ideal is obviously
equivalent to the following condition: (ii') $a,b\in I$ implies
$({\uparrow}a\cap {\uparrow}b)\cap I\neq \emptyset$. Thus, we have
that the following three conditions are equivalent:
\begin{enumerate}
\item $I$ is an ideal of $L$.
\item For each $a, b \in L$ we have $a, b\in I$ iff $\{a,b\}^u\cap
I\neq \emptyset$.
\item For each $a, b \in L$ we have $a, b\in I$ iff $({\uparrow}a\cap
{\uparrow}b)\cap I\neq \emptyset$.
\end{enumerate}
}
\end{remark}

We note that if $L$ has a top element, then $L$ itself is always an
ideal. However, unlike the case with filters, a nonempty
intersection of a family of ideals may not be an ideal as the
following example shows.

\begin{example}
{\rm Let $L$ be the meet semi-lattice shown in Fig.1. Then each
${\downarrow}c_n$ is an ideal of $L$, but
$\displaystyle{\bigcap_{n\in\omega}}{\downarrow}c_n=\{\bot,a,b\}$ is
not an ideal of $L$.}


\

\begin{center}
\begin{tikzpicture}[scale=.5,inner sep=.5mm]
\node[bull] (bot) at (0,0) [label=below:$\bot$] {};%
\node[bull] (left) at (-1,1) [label=left:$a$] {};%
\node[bull] (right) at (1,1) [label=right:$b$] {};%
\node (dots) at (0,3.5) {$\vdots$};%
\node (emp) at (0,4) {};
\node[bull] (c2) at (0,5) [label=right:$c_2$] {};%
\node[bull] (c1) at (0,6) [label=right:$c_1$] {};%
\node[bull] (top) at (0,7) [label=above:$\top$] {};%
\draw (left) -- (bot) -- (right);%
\draw (emp) -- (c2) -- (c1) -- (top);
\end{tikzpicture}

\

$L$

\

\ \ Fig.1

\end{center}

\end{example}

Nevertheless, we have the following analogue of the prime filter
lemma for distributive meet semi-latices. For a proof we refer to
\cite[Section II.5, Lemma 2]{Gra98} or \cite[Theorem 8]{Cel03b}.

\begin{lemma} [Prime Filter Lemma]
\label{meetprime} Suppose that $L$ is a distributive meet
semi-lattice. If $F$ is a filter and $I$ is an ideal of $L$ with $F
\cap I = \emptyset$, then there exists a prime filter $P$ of $L$
such that $F \subseteq P$ and $P \cap I = \emptyset$.
\end{lemma}

As a corollary, we obtain the following useful fact.

\begin{corollary}
Every proper filter $F$ of a distributive meet semi-lattice $L$ is
the intersection of the prime filters of $L$ containing $F$.
\end{corollary}

We call an ideal $I$ of a meet semi-lattice $L$ \emph{proper} if
$I\neq L$, and a proper ideal $I$ of $L$ {\it prime} if for each
$a,b\in L$ we have $a\wedge b \in I$ implies $a\in I$ or $b\in I$.
Then we have the following analogue of a well-known theorem for
lattices.

\begin{proposition}\label{filterideal}
A subset $F$ of a meet semi-lattice $L$ is a prime filter iff
$I=L-F$ is a prime ideal.
\end{proposition}

\begin{proof}
Suppose that $F$ is a prime filter of $L$. Clearly
$F\neq\emptyset,L$ implies $I=L - F\neq\emptyset,L$. Moreover, it is
obvious that condition (i) of the definition of filter implies that
$I$ satisfies condition (i) of the definition of ideal. To show that
$I$ also satisfies condition (ii), suppose that $a,b\in I$. If
$({\uparrow}a\cap {\uparrow}b)\cap I=\emptyset$, then
${\uparrow}a\cap {\uparrow}b\subseteq F$. Since $F$ is prime,
${\uparrow}a\subseteq F$ or ${\uparrow}b\subseteq F$, so either
$a\notin I$ or $b\notin I$, a contradiction. Thus, $({\uparrow}a\cap
{\uparrow}b)\cap I\neq\emptyset$, and so $I$ is an ideal. Finally,
to show that $I$ is prime, suppose that $a\wedge b\in I$. Then
$a\wedge b\notin F$. By condition (ii) of the definition of filter,
either $a\notin F$ or $b\notin F$. Thus, $a\in I$ or $b\in I$.

Conversely, suppose that $I=L - F$ is a prime ideal. Clearly
$I\neq\emptyset,L$ implies $F\neq\emptyset,L$. Moreover, it is
obvious that condition (i) of the definition of ideal implies that
$F$ satisfies condition (i) of the definition of filter. To show
that $F$ also satisfies condition (ii), suppose that $a,b\in F$.
Then $a,b\notin I$. Since $I$ is prime, we have $a\wedge b\notin I$.
Thus, $a\wedge b\in F$, and so $F$ is a filter. Finally, to show
that $F$ is prime, suppose that $F_1\cap F_2\subseteq F$. If $F_1
\not \subseteq F$ and $F_2 \not \subseteq F$, then $F_1\cap I\neq
\emptyset$ and $F_2\cap I\neq\emptyset$. Therefore, there exist
$a_1\in F_1\cap I$ and $a_2 \in F_2\cap I$. Obviously
${\uparrow}a_1\cap {\uparrow}a_2\subseteq F_1\cap F_2\subseteq F$.
On the other hand, since $I$ is an ideal, $({\uparrow}a_1\cap
{\uparrow}a_2)\cap I\neq \emptyset$. This implies $F\cap I\neq
\emptyset$, a contradiction. Thus, either $F_1 \subseteq F$ or $F_2
\subseteq F$, and so $F$ is prime.
\end{proof}

\section{Distributive envelopes, Frink ideals, and optimal filters}

In this section we introduce three new concepts which play a
fundamental role in developing our duality for distributive meet
semi-lattices. The first concept we introduce is that of the
\emph{distributive envelope} $D(L)$ of a distributive meet
semi-lattice $L$.
We analyze how the filters and ideals of $L$ are related to the
filters and ideals of $D(L)$. In general, it is not the case that to
each ideal of $D(L)$ there corresponds an ideal of $L$. This paves a
way for our second concept, that of \emph{Frink ideal}, which is
broader than our earlier concept of ideal. We establish the main
properties of Frink ideals, and prove that the ordered set of Frink
ideals of $L$ is isomorphic to the ordered set of ideals of $D(L)$.
Frink ideals give rise to the third concept, that of \emph{optimal
filters}, which are set-theoretic complements of prime Frink ideals.
We prove an analogue of the prime filter lemma for optimal filters
and Frink ideals, which we call the \emph{optimal filter lemma}, and
show that the ordered set of optimal filters of $L$ is isomorphic to
the ordered set of prime filters of $D(L)$. At the end of the
section, we give a table of relations between different notions of
filters and ideals of $L$ and $D(L)$, and conclude the section by
giving an alternative construction of $D(L)$, which will play a
crucial role in developing our duality.

\subsection{Distributive envelopes}

Let $L$ be a meet semi-lattice and let $\mathrm{Pr}(L)$ denote the
set of prime filters of $L$. We define $\sigma: L \to
\mathcal{P}(\mathrm{Pr}(L))$ by $\sigma(a) = \{x \in \mathrm{Pr}(L):
a \in x\}$ for each $a \in L$.  The next theorem goes back to Stone
\cite{Sto37}.

\begin{theorem}\label{stone}
If $L$ is a meet semi-lattice, then $\sigma: L \to
\mathcal{P}(\mathrm{Pr}(L))$ is a meet semi-lattice homomorphism. If
$L$ has top, then $\sigma$ preserves top, and if $L$ has bottom,
then $\sigma$ preserves bottom. In addition, if $L$ is distributive,
then $\sigma$ is a meet semi-lattice embedding.
\end{theorem}

\begin{proof} For a prime filter $P$ of $L$ we have:

\begin{center}
\begin{tabular}{ll}
$P \in \sigma(a \wedge b)$ & iff $a \wedge b \in P$\\ & iff $a, b
\in P$\\ & iff
$P\in\sigma(a),\sigma(b)$\\
 & iff
$P \in \sigma(a) \cap \sigma(b)$.
\end{tabular}
\end{center}

\noindent Thus, $\sigma$ is a meet semi-lattice homomorphism.
Suppose that $\top\in L$. Since each prime filter of $L$ contains
$\top$, we have $\sigma(\top)=\mathrm{Pr}(L)$. Therefore, $\sigma$
preserves $\top$. Let $\bot\in L$. As each prime filter of $L$
does not contain $\bot$, we have $\sigma(\bot)=\emptyset$. Thus,
$\sigma$ preserves $\bot$. Because $\sigma$ is a meet semi-lattice
homomorphism, we have that $a\leq b$ implies
$\sigma(a)\subseteq\sigma(b)$. Suppose that $L$ is a distributive
meet semi-lattice and $a \not\leq b$. Then ${\uparrow} a\cap
{\downarrow} b=\emptyset$, so by the prime filter lemma, there is
a prime filter $P$ of $L$ such that $a\in P$ and $b\notin P$.
Therefore, $P\in\sigma(a)$ and $P\notin\sigma(b)$. Thus,
$\sigma(a) \not\subseteq \sigma(b)$. Consequently, if $L$ is
distributive, $a\leq b$ iff $\sigma(a)\subseteq\sigma(b)$, and so
$\sigma$ is an embedding.
\end{proof}

Let $D(L)$ denote the sublattice of $\mathcal{P}(\mathrm{Pr}(L))$
generated by $\sigma[L]$. Since $\sigma[L]$ is closed under finite
intersections, for each $A \in \mathcal{P}(\mathrm{Pr}(L))$, we have
$A \in D(L)$ iff $A = \displaystyle{\bigcup_{i=1}^n}\sigma(a_i)$ for
some $a_i \in L$. It follows that $\sigma[L]$ is join-dense in
$D(L)$. Moreover, $\sigma$ is a meet semi-lattice homomorphism from
$L$ to $D(L)$, and $\sigma$ is a meet semi-lattice embedding
whenever $L$ is distributive.

\begin{definition}
{\rm For a distributive meet semi-lattice $L$, we call $D(L)$ the
\textit{distributive envelope} of $L$.}
\end{definition}

Whenever convenient we will identify a distributive meet
semi-lattice $L$ with $\sigma[L]$ and consider $L$ as a join-dense
$\wedge$-subalgebra of $D(L)$. Now we investigate the connection
between filters and ideals of $L$ and $D(L)$.

\begin{lemma}\label{lemma1}
Let $L$ be a distributive meet semi-lattice and let $D(L)$ be the
distributive envelope of $L$. If $F$ is a filter of $L$, then
${\uparrow}_{D(L)} \sigma[F]$ is a filter of $D(L)$, and if $I$ is
an ideal of $L$, then ${\downarrow}_{D(L)} \sigma[I]$ is an ideal of
$D(L)$.
\end{lemma}

\begin{proof}
Let $F$ be a filter of $L$. Clearly ${\uparrow}_{D(L)} \sigma[F]$ is
an upset of $D(L)$. Let $A,B\in {\uparrow}_{D(L)} \sigma[F]$. Then
there exist $a,b\in F$ such that $\sigma(a)\subseteq A$ and
$\sigma(b)\subseteq B$. Therefore, $\sigma(a\wedge b)=\sigma(a)\cap
\sigma(b)\subseteq A\cap B$. Since $F$ is a filter of $L$, we have
$a\wedge b\in F$. Thus, $A\cap B\in {\uparrow}_{D(L)} \sigma[F]$.
Consequently, ${\uparrow}_{D(L)} \sigma[F]$ is a filter of $D(L)$.

Let $I$ be an ideal of $L$. Clearly ${\downarrow}_{D(L)} \sigma[I]$
is a downset of $D(L)$. Let $A,B\in {\downarrow}_{D(L)} \sigma[I]$.
Then there exist $a,b\in I$ such that $A\subseteq \sigma(a)$ and
$B\subseteq \sigma(b)$. Since $I$ is an ideal of $L$, there exists
$e\in \{a,b\}^u\cap I$. Thus, $A,B\subseteq \sigma(e)$, implying
that $A\cup B\in {\downarrow}_{D(L)} \sigma[I]$. Consequently,
${\downarrow}_{D(L)} \sigma[I]$ is an ideal of $D(L)$.
\end{proof}

\begin{lemma}\label{lemma2}
Let $L\in\mathsf{DM}$ and let $D(L)$ be the distributive envelope of
$L$. If $F$ is a filter of $D(L)$, then $\sigma^{-1}(F)$ is a filter
of $L$.
\end{lemma}

\begin{proof}
Let $F$ be a filter of $D(L)$. Since $\top\in L$, we have
$\mathrm{Pr}(L)=\sigma(\top)\in D(L)$, so
$\sigma(\top)=\mathrm{Pr}(L)\in F$, and so
$\top\in\sigma^{-1}(F)$. Thus, $\sigma^{-1}(F)$ is nonempty.
Suppose that $a\in \sigma^{-1}(F)$ and $a\leq b$. Then
$\sigma(a)\in F$ and $\sigma(a)\subseteq\sigma(b)$. Since $F$ is
an upset of $D(L)$, it follows that $\sigma(b)\in F$. Therefore,
$b\in\sigma^{-1}(F)$. For $a,b\in \sigma^{-1}(F)$ we have
$\sigma(a),\sigma(b)\in F$. Since $F$ is a filter of $D(L)$, we
have $\sigma(a\wedge b)=\sigma(a)\cap\sigma(b)\in F$. Thus,
$a\wedge b\in \sigma^{-1}(F)$, and so $\sigma^{-1}(F)$ is a filter
of $L$.
\end{proof}

On the other hand, there exist ideals $I$ of $D(L)$ such that
$\sigma^{-1}(I)$ is not an ideal of $L$ as the following example
shows.

\begin{example}\label{example}
Consider the distributive meet semi-lattice $L$ shown in Fig.1. The
ordered set $\la \mathrm{Pr}(L),\subseteq \ra$ of prime filters of
$L$ together with the distributive envelope $D(L)$ of $L$ is shown
in Fig.2. We have that $I=\{\emptyset,\sigma(a),$
$\sigma(b),\sigma(a)\cup\sigma(b)\}$ is an ideal of $D(L)$, but that
$\sigma^{-1}(I)=\{\bot,a,b\}$ is not an ideal of $L$.
\end{example}


\begin{center}

$$
\begin{array}{ccc}
\begin{tikzpicture}[scale=.5,inner sep=.5mm]
\node[bull] (bot) at (0,0) [label=below:$\bot$] {};%
\node[bull] (left) at (-1,1) [label=left:$a$] {};%
\node[bull] (right) at (1,1) [label=right:$b$] {};%
\node (dots) at (0,3.5) {$\vdots$};%
\node (emp) at (0,4) {};
\node[bull] (c2) at (0,5) [label=right:$c_2$] {};%
\node[bull] (c1) at (0,6) [label=right:$c_1$] {};%
\node[bull] (top) at (0,7) [label=above:$\top$] {};%
\draw (left) -- (bot) -- (right);%
\draw (emp) -- (c2) -- (c1) -- (top);
\end{tikzpicture}&\qquad
\begin{tikzpicture}[scale=.5,inner sep=.5mm]
\node[bull] (top) at (0,0) [label=below:$\{\top\}$] {};%
\node[bull] (c1) at (0,1) [label=right:${\uparrow}c_1$] {};%
\node[bull] (c2) at (0,2) [label=right:${\uparrow}c_2$] {};%
\node (empdown) at (0,3) {};
\node (dots) at (0,4) {$\vdots$};%
\node (empup) at (0,6) {};
\node[bull] (left) at (-1,7) [label=left:${\uparrow}a$] {};%
\node[bull] (right) at (1,7) [label=right:${\uparrow}b$] {};%
\draw (top) -- (c1) -- (c2) -- (empdown);%
\draw (empup) -- (left);%
\draw (empup) -- (right);
\end{tikzpicture}\qquad&
\begin{tikzpicture}[scale=.5,inner sep=.5mm]
\node[bull] (bot) at (0,0) [label=below:$\emptyset$] {};%
\node[bull] (left) at (-1,1) [label=left:$\sigma(a)$] {};%
\node[bull] (right) at (1,1) [label=right:$\sigma(b)$] {};%
\node[holl] (join) at (0,2) [label=right:\ $\sigma(a)\cup\sigma(b)$]
{};
\node (dots) at (0,3.5) {$\vdots$};%
\node (emp) at (0,4) {};
\node[bull] (c2) at (0,5) [label=right:$\sigma(c_2)$] {};%
\node[bull] (c1) at (0,6) [label=right:$\sigma(c_1)$] {};%
\node[bull] (top) at (0,7) [label=above:$\mathrm{Pr}(L)$] {};%
\draw (left) -- (bot) -- (right);%
\draw (left) -- (join) -- (right);%
\draw (emp) -- (c2) -- (c1) -- (top);
\end{tikzpicture}\\
\\
L&\mathrm{Pr}(L)&\hskip-2em D(L)\\
\\
& \ \textrm{Fig.2}
\end{array}
$$

\end{center}

\begin{lemma}\label{lemma}
Let $L$ be a distributive meet semi-lattice, $D(L)$ be the
distributive envelope of $L$, and $I$ be an ideal of $D(L)$. Then
$I$ is the ideal of $D(L)$ generated by $\sigma[\sigma^{-1}(I)]=I
\cap \sigma[L]$.
\end{lemma}

\begin{proof}
Let $J$ be the ideal of $D(L)$ generated by $\sigma[\sigma^{-1}(I)]
= I \cap \sigma[L]$. Obviously $J \subseteq I$. On the other hand,
if $A \in I$, then as $A = \displaystyle{\bigcup_{i=1}^n}
\sigma(b_i)$ for some $b_i \in L$, we have $\sigma(b_i) \in I \cap
\sigma[L]$ for each $i \leq n$. Thus, $A \in J$.
\end{proof}

\begin{corollary}\label{corollary}
Let $L$ be a distributive meet semi-lattice and let $D(L)$ be the
distributive envelope of $L$. For ideals $I, J$ of $D(L)$ we have
that the following conditions are equivalent:
\begin{enumerate}
\item $I=J$.
\item $\sigma^{-1}(I) = \sigma^{-1}(J)$.
\item $I\cap \sigma[L]=J\cap\sigma[L]$.
\end{enumerate}
\end{corollary}

On the other hand, there exist filters $F$ of $D(L)$ such that $F$
is not generated by $\sigma[\sigma^{-1}(F)]=F\cap \sigma[L]$ as the
following example shows.

\begin{example}
Consider the distributive meet semi-lattice $L$ and its distributive
envelope $D(L)$ shown in Fig.2. Then $F=\{\sigma(a)\cup\sigma(b),
\sigma(c_n),\mathrm{Pr}(L):n\in\omega\}$ is a filter of $D(L)$ which
is not generated by $\sigma[\sigma^{-1}(F)]=F\cap
\sigma[L]=\{\sigma(c_n),\mathrm{Pr}(L):n\in\omega\}$.
\end{example}

For an ideal $I$ of $D(L)$, if $\sigma^{-1}(I)$ were an ideal of
$L$, then Lemma \ref{lemma} and Corollary \ref{corollary} would
imply that there is a 1-1 correspondence between ideals of $L$ and
$D(L)$. However, $\sigma^{-1}(I)$ is not necessarily an ideal of $L$
as we have shown in Example \ref{example}. This forces us to
introduce a weaker notion of an ideal of $L$, that of a \emph{Frink
ideal}.

\subsection{Frink ideals}

\begin{definition} {\rm (Frink \cite[p.\ 227]{Fri54})
Let $L$ be a meet semi-lattice. A nonempty subset $I$ of $L$ is
called a \textit{Frink ideal} (\emph{F-ideal} for short) if for each
finite subset $A$ of $I$ we have $A^{ul} \subseteq I$. Equivalently,
$I$ is a Frink ideal if for each $a_1, \ldots, a_n \in I$ and $c \in
L$, whenever $\displaystyle{\bigcap_{i=1}^n}{\uparrow} a_i \subseteq
{\uparrow}c$, we have $c \in I$.
We call an F-ideal $I$ of $L$ \emph{proper} if $I\neq L$, and we
call $I$ \textit{prime} if it is proper and $a \wedge b \in I$
implies $a \in I$ or $b \in I$ for each $a, b \in L$.}
\end{definition}

It is easy to verify that for each $a\in L$ we have ${\downarrow}a$
is an F-ideal. Moreover, unlike the case with ideals, a nonempty
intersection of a family of F-ideals is again an F-ideal. Therefore,
for each nonempty $X \subseteq L$, there exists a least F-ideal
containing $X$. We call it the {\it F-ideal generated by $X$}, and
denote it by $(X]$.

\begin{lemma} \label{F-ideal}
Let $L$ be a meet semi-lattice and let $X$ be a nonempty subset of
$L$. Then $(X] = \{a \in L: \exists$ finite $A \subseteq X$ with
$a\in A^{ul}\} = \{a \in L: \exists a_1, \dots, a_n \in X$ with $
\displaystyle{\bigcap_{i=1}^n}{\uparrow}a_i \subseteq
{\uparrow}a\}$.
\end{lemma}

\begin{proof}
Suppose that $L$ is a meet semi-lattice and $X$ is a nonempty subset
of $L$. Clearly $(X]$ is an F-ideal. Let $I$ be an F-ideal with $X
\subseteq I$. We show that $(X]\subseteq I$. If $a\in (X]$, then
there exist $a_1, \dots, a_n \in X$ such that
$\displaystyle{\bigcap_{i=1}^n}{\uparrow}a_i \subseteq {\uparrow}a$.
Since $a_1, \dots, a_n \in I$ and $I$ is an F-ideal, we have $a\in
I$. Thus, $(X]$ is the F-ideal generated by $X$.
\end{proof}

%

\begin{lemma}\label{semi}
Let $L$ be a meet semi-lattice. Then each ideal of $L$ is an F-ideal
and each F-ideal of $L$ is a downset. Moreover, if $L$ is a lattice,
then the two notions coincide with the usual notion of an ideal of a
lattice.
\end{lemma}

\begin{proof}
Suppose that $I$ is an ideal of $L$, $a_1,\dots,a_n\in I$, $a\in L$,
and $\displaystyle{\bigcap_{i=1}^n}{\uparrow}a_i\subseteq{\uparrow}
a$. Since $I$ is an ideal, it is easy to prove by induction on $n$
that $\displaystyle{\bigcap_{i=1}^n}{\uparrow}a_i\cap
I\neq\emptyset$. Let $c\in \displaystyle{\bigcap_{i=1}^n}
{\uparrow}a_i\cap I$. Then $c\in{\uparrow}a$, so $a\leq c$, and so
$a\in I$. Therefore, $I$ is an F-ideal. That every F-ideal is a
downset is obvious. Now suppose that $L$ is a lattice and $I$ is an
F-ideal of $L$. Then $I$ is a downset. Moreover, for $a,b\in I$ we
have ${\uparrow} a\cap {\uparrow}b={\uparrow}(a\vee b)$. Thus,
$a\vee b\in I$, and so the two notions coincide with the usual
notion of an ideal of $L$.
\end{proof}

In particular, if a meet semi-lattice $L$ is finite and $\top\in L$,
then $L$ is a lattice, and so each F-ideal of $L$ is an ideal. On
the other hand, there exist meet semi-lattices for which not every
F-ideal is an ideal. For example, if $L$ is the lattice shown in
Fig.1, then $I=\{\bot,a,b\}$ is an F-ideal which is not an ideal.
The next lemma is useful in obtaining a 1-1 correspondence between
F-ideals of a distributive meet semi-lattice $L$ and ideals of its
distributive envelope $D(L)$.

\begin{lemma}\label{sigmaup}
Let $L$ be a distributive meet semi-lattice. For each $a_1, \ldots,
a_n,$ $b \in L$ we have $\displaystyle{\bigcap_{i=1}^n}
{\uparrow}a_i \subseteq {\uparrow}b \ \text{ iff } \ \sigma(b)
\subseteq \displaystyle{\bigcup_{i=1}^n} \sigma(a_i)$.
\end{lemma}

\begin{proof}
First suppose that $\displaystyle{\bigcap_{i=1}^n} {\uparrow}a_i
\subseteq {\uparrow}b$ and $x\in\sigma(b)$. Then $b\in x$, so
${\uparrow}b\subseteq x$. Therefore, $\displaystyle{\bigcap_{i=1}^n}
{\uparrow}a_i \subseteq x$, and since $x$ is a prime filter of $L$,
there exists $i\leq n$ such that ${\uparrow}a_i \subseteq x$. Thus,
$a_i\in x$, so $x\in \sigma(a_i)$, and so $x\in
\displaystyle{\bigcup_{i=1}^n} \sigma(a_i)$. It follows that
$\sigma(b) \subseteq \displaystyle{\bigcup_{i=1}^n} \sigma(a_i)$.

Conversely, suppose that $\sigma(b) \subseteq
\displaystyle{\bigcup_{i=1}^n} \sigma(a_i)$ and $c\in
\displaystyle{\bigcap_{i=1}^n} {\uparrow}a_i$. Then $a_i\leq c$ for
each $i\leq n$. Therefore, $\sigma(a_i)\subseteq\sigma(c)$ for each
$i\leq n$, and so $\displaystyle{\bigcup_{i=1}^n} \sigma(a_i)
\subseteq \sigma(c)$. Thus, $\sigma(b)\subseteq\sigma(c)$, so $b\leq
c$, and so $c\in{\uparrow}b$. It follows that
$\displaystyle{\bigcap_{i=1}^n} {\uparrow}a_i \subseteq
{\uparrow}b$.
\end{proof}

\begin{theorem}\label{theorem}
Let $L$ be a distributive meet semi-lattice and let $D(L)$ be its
distributive envelope.
\begin{enumerate}
\item $I \subseteq L$ is an F-ideal of $L$ iff there is an ideal
$J$ of $D(L)$ such that $I = \sigma^{-1}(J)$.
\item $I \subseteq L$ is a prime F-ideal of $L$ iff there is a
prime ideal $J$ of $D(L)$ such that $I = \sigma^{-1}(J)$.
\end{enumerate}
\end{theorem}

\begin{proof}
(1) Let $I$ be an F-ideal of $L$ and $J$ be the ideal of $D(L)$
generated by $\sigma[I]=\{\sigma(a): a \in I\}$. We claim that $I =
\sigma^{-1}(J)$. It is clear that $I \subseteq \sigma^{-1}(J)$. Let
$b\in\sigma^{-1}(J)$. Then $\sigma(b) \in J$, so $\sigma(b)
\subseteq \displaystyle{\bigcup_{i=1}^n}\sigma(a_i)$ for some $a_1,
\ldots, a_n \in I$. By Lemma \ref{sigmaup},
$\displaystyle{\bigcap_{i=1}^n}{\uparrow} a_i \subseteq
{\uparrow}b$. Since $I$ is an F-ideal, it follows that $b \in I$.
Thus, $\sigma^{-1}(J)\subseteq I$. Consequently, $I =
\sigma^{-1}(J)$. Conversely, let $I = \sigma^{-1}(J)$ for $J$ an
ideal of $D(L)$. Clearly $I$ is nonempty. For $a_1,\dots,a_n \in I$
and $c \in L$ with $\displaystyle{\bigcap_{i=1}^n}{\uparrow} a_i
\subseteq {\uparrow}c$, by Lemma \ref{sigmaup}, we have $\sigma(c)
\subseteq \displaystyle{\bigcup_{i=1}^n}\sigma(a_i)$. Therefore,
$\sigma(c) \in J$, and so $c \in I$. Thus, $I$ is an F-ideal of $L$.

(2) Let $I$ be a prime F-ideal of $L$ and let $J$ be the ideal of
$D(L)$ generated by $\sigma[I]$. By (1), $I = \sigma^{-1}(J)$. We
show that $J$ is a prime ideal of $D(L)$. Suppose that $A\cap B\in
J$. Then $A=\displaystyle{\bigcup_{i=1}^n}\sigma(a_i)$ and
$B=\displaystyle{\bigcup_{j=1}^m}\sigma(b_j)$ for some
$a_1,\ldots,a_n, b_1,\ldots,b_m\in L$. Therefore,
$\displaystyle{\bigcup_{i,j}}(\sigma(a_i)\cap \sigma(b_j))\in J$. It
follows that $\sigma(a_i)\cap\sigma(b_j) \in \sigma[I]$ for all
$i,j$. Since $I$ is prime, either $a_i\in I$ or $b_j\in I$. We look
at $a_1\wedge b_1,\ldots,a_1\wedge b_m$. If $a_1\notin I$, then
$b_1,\ldots,b_m\in I$, so $B=\displaystyle{\bigcup_{j=1}^m}
\sigma(b_j)\in J$. Therefore, without loss of generality we may
assume that $a_1\in I$. Now we look at $a_2\wedge
b_1,\ldots,a_2\wedge b_m$. If $a_2\notin I$, then $b_1,\ldots,b_m\in
I$, so again $B=\displaystyle{\bigcup_{j=1}^m} \sigma(b_j)\in J$.
Thus, without loss of generality we may assume that $a_1,a_2\in I$.
Going through all $a_1,\dots,a_n$ we obtain that either
$B=\displaystyle{\bigcup_{j=1}^m} \sigma(b_j)\in J$ or
$A=\displaystyle{\bigcup_{i=1}^n} \sigma(a_i)\in J$. It follows that
$J$ is a prime ideal of $D(L)$. The converse implication easily
follows from (1) and the definition of prime F-ideals.
\end{proof}

As an immediate consequence of (the proof of) Theorem
\ref{theorem} and Corollary \ref{corollary}, we obtain the
following:

\begin{corollary}
Let $L$ be a distributive meet semi-lattice and let $D(L)$ be its
distributive envelope. The ordered set of Frink ideals of $L$ is
isomorphic to the ordered set of ideals of $D(L)$, and the ordered
set of prime Frink ideals of $L$ is isomorphic to the ordered set of
prime ideals of $D(L)$.
\end{corollary}

\subsection{Optimal filters}

\begin{definition}
{\rm Let $L$ be a distributive meet semi-lattice and let $D(L)$ be
its distributive envelope. A filter $F$ of $L$ is said to be
\textit{optimal} if there exists a prime filter $P$ of $D(L)$ such
that $F = \sigma^{-1}[P]$. We denote the set of optimal filters of
$L$ by $\mathrm{Opt}(L)$.}
\end{definition}

Clearly each optimal filter is proper.

\begin{lemma} [Optimal Filter Lemma]
Let $L$ be a distributive meet semi-lattice. If $F$ is a filter and
$I$ is an F-ideal of $L$ with $F\cap I=\emptyset$, then there exists
an optimal filter $G$ of $L$ such that $F \subseteq G$ and $G\cap
I=\emptyset$.
\end{lemma}

\begin{proof}
Let $F$ be a filter and $I$ be an F-ideal of $L$ with $F\cap
I=\emptyset$. Let also $\nabla$ be the filter and $\Delta$ be the
ideal of $D(L)$ generated by $\sigma[F]$ and $\sigma[I]$,
respectively. Suppose that there exists $A\in\nabla \cap \Delta$.
Then there are $a_1, \ldots , a_n\in I$ and $b \in F$ such that
$A=\displaystyle{\bigcup_{i=1}^n}\sigma(a_i)$ and $\sigma(b)
\subseteq A$. Therefore, $\sigma(b) \subseteq
\displaystyle{\bigcup_{i=1}^n}\sigma(a_i)$. By Lemma \ref{sigmaup},
$\displaystyle{\bigcap_{i=1}^n}{\uparrow}a_i \subseteq {\uparrow}b$.
Since $I$ is an F-ideal, we obtain $b \in I$, a contradiction. Thus,
$\nabla\cap\Delta=\emptyset$, and so there is a prime filter $P$ of
$D(L)$ such that $\nabla \subseteq P$ and $P \cap \Delta =
\emptyset$. It follows that $F \subseteq \sigma^{-1}[P]$ and
$\sigma^{-1}[P] \cap I = \emptyset$. If we set $G=\sigma^{-1}[P]$,
then $G$ is the desired optimal filter.
\end{proof}

\begin{corollary}
Every proper filter $F$ of a distributive meet semi-lattice is the
intersection of the optimal filters containing $F$.
\end{corollary}

\begin{proposition}\label{proposition1}
Let $L$ be a distributive meet semi-lattice and let $F$ be a filter
of $L$. Then the following conditions are equivalent:
\begin{enumerate}
\item $F$ is an optimal filter.
\item $F$ is a filter and $L - F$ is an F-ideal.
\item There is an F-ideal $I$ of $L$ such that $F \cap I = \emptyset$
and $F$ is maximal among the filters of $L$ with this property.
\end{enumerate}
\end{proposition}

\begin{proof}
(1)$\Rightarrow$(2): Let $F$ be an optimal filter of $L$, $P$ be a
prime filter of $D(L)$ such that $F = \sigma^{-1}[P]$, and $I = L -
F$. Then $F$ is a proper filter of $L$, and so $I$ is nonempty. For
$a_1,\dots,a_n\in I$ and $c\in L$ with
$\displaystyle{\bigcap_{i=1}^n}{\uparrow}a_i \subseteq {\uparrow}c$,
Lemma \ref{sigmaup} implies that $\sigma(c) \subseteq
\displaystyle{\bigcup_{i=1}^n}\sigma(a_i)$. If $c \not \in I$, then
$c \in F$, and so $\sigma(c) \in P$. Therefore,
$\displaystyle{\bigcup_{i=1}^n}\sigma(a_i) \in P$. Since $P$ is
prime, $\sigma(a_i) \in P$ for some $a_i\in I$. Thus, $a_i \in F\cap
I$, which is a contradiction. It follows that $c\in I$, and so $I$
is an F-ideal.

(2)$\Rightarrow$(3) is obvious.

(3)$\Rightarrow$(1): Suppose that $F$ is a filter and $I$ is an
F-ideal of $L$ such that $F \cap I = \emptyset$ and $F$ is maximal
among the filters of $L$ with this property. By the optimal filter
lemma, there is an optimal filter $G$ of $L$ such that $F \subseteq
G$ and $G \cap I = \emptyset$. Since $F$ is a maximal filter with
this property, $F = G$. Thus, $F$ is optimal.
\end{proof}

\begin{remark}\label{Hansoul}
Hansoul \cite{Han03} defines a \emph{weakly prime} ideal $I$ of a
distributive join semi-lattice $L$ as an ideal $I$ such that for
each $a_1, \ldots, a_n \not \in I$ and $b \in I$, there is $c \in
\displaystyle{\bigcap_{i=1}^n} {\downarrow}a_i$ such that $c \not
\leq b$. Clearly an ideal $I$ is weakly prime iff for each $a_1,
\ldots, a_n \not \in I$ and $b \in L$, from
$\displaystyle{\bigcap_{i=1}^n} {\downarrow}a_i \subseteq
{\downarrow}b$ it follows that $b \notin I$. Since the dual $L^d$ of
$L$ is a distributive meet semi-lattice and
$\displaystyle{\bigcap_{i=1}^n} {\downarrow}a_i \subseteq
{\downarrow}b$ iff $\displaystyle{\bigcap_{i=1}^n} {\uparrow}^{d}a_i
\subseteq {\uparrow}^{d}b$, it follows from Proposition
\ref{proposition1} that an ideal $I$ of $L$ is weakly prime iff it
is an optimal filter of $L^d$. A more detailed comparison with
Hansoul's work can be found in Section 11.1.
\end{remark}

\begin{lemma}\label{lemma3}
Let $L$ be a distributive meet semi-lattice. Then each prime filter
of $L$ is optimal. Moreover, if $L$ is a lattice, then the two
notions coincide with the usual notion of a prime filter of a
lattice.
\end{lemma}

\begin{proof}
Suppose that $L$ is a distributive meet semi-lattice and $P$ is a
prime filter of $L$. By Proposition \ref{filterideal}, $I=L - P$ is
a prime ideal, hence an $F$-ideal of $L$ by Lemma \ref{semi}.
Moreover, $P\cap I=\emptyset$ and by the construction of $I$, $P$ is
maximal among the filters of $L$ with this property. Thus, $P$ is
optimal by Proposition \ref{proposition1}. Now let in addition $L$
be a lattice and $F$ be an optimal filter of $L$. By Proposition
\ref{proposition1}, $I=L-F$ is an $F$-ideal of $L$, and by Lemma
\ref{semi}, $I$ is actually an ideal of a lattice. Thus, the two
notions coincide with the usual notion of a prime filter of a
lattice.
\end{proof}

In particular, if $L$ is a finite distributive meet semi-lattice and
$\top\in L$, then $L$ is a finite distributive lattice, and so each
optimal filter of $L$ is prime. On the other hand, there exist
distributive meet semi-lattices in which not every optimal filter is
prime. For example, in the distributive meet semi-lattice $L$ shown
in Fig.1 one can easily check that $F=L - \{\bot,a,b\}$ is an
optimal filter of $L$, but that it is not prime. Now we show that
optimal filters of a distributive meet semi-lattice $L$ are in a 1-1
correspondence with prime filters of the distributive envelope
$D(L)$ of $L$.

\begin{proposition}\label{4.25}
Let $L$ be a distributive meet semi-lattice and let $D(L)$ be its
distributive envelope.
\begin{enumerate}
\item If $P \in \mathrm{Pr}(D(L))$, then $\sigma^{-1}(P)$ is an optimal
filter of $L$ and ${\uparrow}_{D(L)}(P\cap \sigma[L]) = P$.
\item If $F \in \mathrm{Opt}(L)$, then $ {\uparrow}_{D(L)} \sigma[F]
\in \mathrm{Pr}(D(L))$.
\item For a filter $F$ of $D(L)$, $F \in \mathrm{Pr}(D(L))$ iff there is
$G \in \mathrm{Opt}(L)$ such that $F = {\uparrow}_{D(L)}
\sigma[G]$.
\item If $P, Q \in \mathrm{Pr}(D(L))$, then the following conditions are
equivalent:
\begin{enumerate}
\item $P \subseteq Q$.
\item $\sigma^{-1}(P)\subseteq\sigma^{-1}(Q)$.
\item $\sigma[\sigma^{-1}(P)]\subseteq\sigma[\sigma^{-1}(Q)]$.
\item $P \cap \sigma[L] \subseteq Q \cap \sigma[L]$.
\item ${\uparrow}_{D(L)}(P \cap \sigma[L]) \subseteq {\uparrow}_{D(L)}
(Q \cap \sigma[L])$.
\end{enumerate}
\end{enumerate}
\end{proposition}

\begin{proof}
(1) That $\sigma^{-1}(P)$ is an optimal filter of $L$ follows from
the definition. We show that ${\uparrow}_{D(L)}(P\cap \sigma[L]) =
P$. It is clear that ${\uparrow}_{D(L)}(P\cap \sigma[L]) \subseteq
P$. To prove the other inclusion let $A \in P$. Then $A =
\sigma(b_1) \cup \ldots \cup \sigma(b_n)$ for some $b_1, \ldots, b_n
\in L$. Since $P$ is prime, there exists $i\leq n$ such that
$\sigma(b_i) \in P$. Therefore, $\sigma(b_i) \in P\cap \sigma[L]$,
and so $A \in {\uparrow}_{D(L)}(P\cap\sigma[L])$.

(2) Let $F \in \mathrm{Opt}(L)$ and let $Q \in \mathrm{Pr}(L)$ be
such that $F = \sigma^{-1}(Q)$. Then $\sigma[F]=Q\cap \sigma[L]$,
and by (1), ${\uparrow}_{D(L)} \sigma[F] = {\uparrow}_{D(L)}(Q
\cap \sigma[L]) = Q$. Thus, $ {\uparrow}_{D(L)} \sigma[F] \in
\mathrm{Pr}(L)$.

(3) follows from (1) and (2).

(4)
(a)$\Rightarrow$(b)$\Rightarrow$(c)$\Rightarrow$(d)$\Rightarrow$(e)
is obvious, and (e)$\Rightarrow$(a) follows from (1).
\end{proof}

\begin{corollary}\label{isom-1}
Let $L$ be a distributive meet semi-lattice and let $D(L)$ be its
distributive envelope. Then the map $P \mapsto \sigma^{-1}(P)$ is an
order-isomorphism between $\la \mathrm{Pr}(D(L)), \subseteq \ra$ and
$\la \mathrm{Opt}(L), \subseteq \ra$ whose inverse is the map $F
\mapsto {\uparrow}_{D(L)}\sigma[F]$.
\end{corollary}

Thus, for a distributive meet semi-lattice $L$ and its distributive
envelope $D(L)$, we have that F-ideals of $L$ correspond to ideals
of $D(L)$, that prime F-ideals of $L$ correspond to prime ideals of
$D(L)$, and that optimal filters of $L$ correspond to prime filters
of $D(L)$.

\begin{lemma}
Let $L$ be a distributive meet semi-lattice and let $D(L)$ be its
distributive envelope. If $P$ is a prime filter of $L$, then
${\uparrow}_{D(L)}\sigma[P]$ is a prime filter of $D(L)$, and if $I$
is a prime ideal of $L$, then ${\downarrow}_{D(L)}\sigma[I]$ is a
prime ideal of $D(L)$.
\end{lemma}

\begin{proof}
Let $P$ be a prime filter of $L$. By Lemma \ref{lemma3}, $P$ is
optimal, and by Proposition \ref{4.25}.2,
${\uparrow}_{D(L)}\sigma[P]$ is a prime filter of $D(L)$.
Now let $I$ be a prime ideal of $L$. By Lemma \ref{lemma1},
${\downarrow}_{D(L)}\sigma[I]$ is an ideal of $D(L)$. We show that
${\downarrow}_{D(L)}\sigma[I]$ is prime. Let $A\cap B\in
{\downarrow}_{D(L)}\sigma[I]$, and let
$A=\displaystyle{\bigcup_{i=1}^n}\sigma(a_i)$ and
$B=\displaystyle{\bigcup_{j=1}^m}\sigma(b_j)$. From
$\displaystyle{\bigcup_{i=1}^n}\sigma(a_i)\cap\bigcup_{j=1}^m
\sigma(b_j)\in {\downarrow}_{D(L)}\sigma[I]$ it follows that
$\sigma(a_i)\cap \sigma(b_j)\in {\downarrow}_{D(L)}\sigma[I]$, and
so $a_i\wedge b_j\in I$ for each $i,j$. Since $I$ is prime, either
$a_i\in I$ or $b_j\in I$. We look at $a_1\wedge
b_1,\ldots,a_1\wedge b_m$. If $a_1\notin I$, then
$b_1,\ldots,b_m\in I$, and so there exists
$c\in\displaystyle{\bigcap_{j=1}^m}{\uparrow}b_j\cap I$.
Therefore, $B=\displaystyle{\bigcup_{j=1}^m}\sigma(b_j)\subseteq
\sigma(c)\in\sigma[I]$, so $B\in {\downarrow}_{D(L)}\sigma[I]$,
and so without loss of generality we may assume that $a_1\in I$.
Now we look at $a_2\wedge b_1,\ldots,a_2\wedge b_m$. If $a_2\notin
I$, then $b_1,\ldots,b_m\in I$, so again $B\in
{\downarrow}_{D(L)}\sigma[I]$. Thus, without loss of generality we
may assume that $a_1,a_2\in I$. Going through all $a_1,\dots,a_n$
we obtain that either $B\in {\downarrow}_{D(L)}\sigma[I]$ or $A\in
{\downarrow}_{D(L)}\sigma[I]$. It follows that
${\downarrow}_{D(L)}\sigma[I]$ is a prime ideal of $D(L)$.
\end{proof}

For the reader's convenience, we give a table of relations between
different notions of filters and ideals of $L$ and $D(L)$.


\begin{center}

\

{\bf Correspondences between filters of $L$ and $D(L)$}

\

\

\begin{tabular}{p{5em}|p{5em}|p{5em}|p{5em}|p{5em}}
$L$ &            & filters & prime filters & optimal filters\\
\hline%
$D(L)$ & filters & ${\uparrow}_{D(L)}\sigma[F]$, $F$ is a filter of
$L$ & ${\uparrow}_{D(L)}\sigma[F]$, $F$ is a prime filter of $L$ &
prime filters
\end{tabular}

\

\

{\bf Correspondences between ideals of $L$ and $D(L)$}

\

\

\begin{tabular}{p{5em}|p{5em}|p{5em}|p{5em}|p{5em}}
$L$ & F-ideals & prime F-ideals & ideals & prime ideals\\
\hline%
$D(L)$  & ideals & prime ideals & ${\downarrow}_{D(L)}\sigma[I]$,
$I$ is an ideal of $L$ & ${\downarrow}_{D(L)}\sigma[I]$, $I$ is a
prime ideal of $L$
\end{tabular}

\

\

\end{center}

Let $L$ be a distributive meet semi-lattice. We define a map $\phi:
L \to \mathcal{P}(\mathrm{Opt}(L))$ by $$\phi(a) = \{x \in
\mathrm{Opt}(L): a \in x\}.$$ It is easy to verify that $\phi$ is a
meet semi-lattice homomorphism, that it preserves top whenever $L$
has a top, and that it preserves bottom whenever $L$ has a bottom.
It also follows from the optimal filter lemma that $\phi$ is 1-1.
Thus, we obtain:

\begin{proposition}
Let $L$ be a distributive meet semi-lattice. Then $L$ is isomorphic
to the meet semi-lattice $\la \{\phi(a): a \in L\},\cap\ra$.
\end{proposition}

\begin{corollary}
For a distributive meet semi-lattice $L$, we have that the meet
semi-lattices $\la \{\sigma(a): a \in L\},\cap\ra$ and $\la
\{\phi(a): a \in L\},\cap\ra$ are isomorphic.
\end{corollary}

\begin{lemma}\label{basic}
Let $L$ be a distributive meet semi-lattice and let $a,b_1,$
$\ldots
b_n \in L$. Then $\phi(a) \subseteq
\displaystyle{\bigcup_{i=1}^n}\phi(b_i)$ iff
$\displaystyle{\bigcap_{i=1}^n}{\uparrow}b_i \subseteq {\uparrow}a$.
\end{lemma}

\begin{proof}
First suppose that $\phi(a) \subseteq \displaystyle{\bigcup_{i=1}^n}
\phi(b_i)$. If $c \in \displaystyle{\bigcap_{i=1}^n}{\uparrow}b_i$
and $c\notin{\uparrow}a$, then $b_i\leq c$ for each $i\leq n$ and $a
\not \leq c$. By the optimal filter lemma, there exists an optimal
filter $x$ of $L$ such that $a \in x$ and $c \not \in x$. But then
$b_i \not \in x$ for each $i\leq n$. Therefore, $x \in \phi(a)$ but
$x \not \in \displaystyle{\bigcup_{i=1}^n} \phi(b_i)$, a
contradiction. Thus, $a \leq c$, so $c \in {\uparrow}a$, and so
$\displaystyle{\bigcap_{i=1}^n}{\uparrow}b_i \subseteq {\uparrow}a$.
Now suppose that $\displaystyle{\bigcap_{i=1}^n}{\uparrow}b_i
\subseteq {\uparrow}a$. If $x \in \phi(a)$ and $x\notin
\displaystyle{\bigcup_{i=1}^n}\phi(b_i)$, then $a\in x$ and $b_i
\not \in x$ for each $i\leq n$. Since $x$ is an optimal filter, $L -
x$ is an F-ideal by Proposition \ref{proposition1}. So $b_i\in L -
x$ for each $i\leq n$ and $\displaystyle{\bigcap_{i=1}^n}
{\uparrow}b_i \subseteq {\uparrow}a$ imply $a \in L - x$, a
contradiction. Thus, $x \in \displaystyle{\bigcup_{i=1}^n}
\phi(b_i)$, and so $\phi(a) \subseteq \displaystyle{\bigcup_{i=1}^n}
\phi(b_i)$.
\end{proof}

\begin{proposition}\label{basic-2}
Let $L$ be a distributive meet semi-lattice and let $D(L)$ be its
distributive envelope. Then the closure under finite unions of
$\phi[L]$ is isomorphic to $D(L)$.
\end{proposition}

\begin{proof}
It follows from Lemmas \ref{sigmaup} and \ref{basic} that $\phi(a_1)
\cup \ldots \cup \phi(a_n) = \phi(b_1) \cup \ldots \cup \phi(b_m)$
iff $\sigma(a_1) \cup \ldots \cup \sigma(a_n) = \sigma(b_1) \cup
\ldots \cup \sigma(b_m)$. Thus, we can define a map $h$ from the
closure under finite unions of $\phi[L]$ to $D(L)$ by $h(\phi(a_1)
\cup \ldots \cup \phi(a_n)) = \sigma(a_1) \cup \ldots \cup
\sigma(a_n)$. This map is clearly onto, and it follows from Lemmas
\ref{sigmaup} and \ref{basic} that it is an order-isomorphism.
\end{proof} 
\section{Sup-homomorphisms and an abstract characterization of distributive envelopes}

Let $L$ and $K$ be distributive meet semi-lattices and let $h:L\to K$ be a meet semi-lattice homomorphism. If there exist $a,b\in L$ such that $a\vee b$ exists in $L$ and $h(a)\vee h(b)$ exists in $K$, it is not necessary that $h(a\vee b)=h(a)\vee h(b)$. Therefore, $h$ may not be extended to a lattice homomorphism from $D(L)$ to $D(K)$. In this section we introduce a stronger notion of a homomorphism between distributive meet semi-lattices, we call a sup-homomorphism. We show that sup-homomorphisms have the property that they preserve all existing joins and that they can be extended to lattice homomorphisms between the distributive envelopes. We also give an abstract characterization of the distributive envelope by means of sup-homomorphisms, and prove that the category of distributive lattices and lattice homomorphisms is a reflective subcategory of the category of distributive meet semi-lattices and sup-homomorphisms.

\begin{definition}
{\rm
Let $L$ and $K$ be distributive meet semi-lattices. We call a meet semi-lattice homomorphism $h: L\to K$ a \textit{sup-homomorphism} if
$$\bigcap_{i=1}^n{\uparrow}a_i \subseteq {\uparrow}b \text{ implies } \bigcap_{i=1}^n{\uparrow}h(a_i) \subseteq {\uparrow} h(b)$$ for each $a_1,\ldots,a_n,b\in L$.
}
\end{definition}

\begin{proposition} \label{prop-sup}
Let $L$ and $K$ be distributive meet semi-lattices and $h: L \to K$ be a meet semi-lattice homomorphism. Then the following conditions are equivalent.
\begin{enumerate}
\item $h$ is a sup-homomorphism.
\item $h^{-1}[I]$ is an F-ideal of $L$ for each F-ideal $I$ of $K$.
\item $h^{-1}[F]$ is an optimal filter of $L$ for each optimal filter $F$ of $K$.
\end{enumerate}
\end{proposition}

\begin{proof}
(1)$\Rightarrow$(2): Let $h:L\to K$ be a sup-homomorphism and let $I$ be an F-ideal of $K$. We show that $h^{-1}[I]$ is an F-ideal of $L$. Suppose that $a_1, \ldots, a_n\in h^{-1}[I]$ and $b \in L$ are such that $\displaystyle{\bigcap_{i=1}^n}{\uparrow}a_{i} \subseteq {\uparrow} b$. Then $h(a_{1}), \ldots, h(a_{n}) \in I$. Since $h$ is a sup-homomorphism, $\displaystyle{\bigcap_{i=1}^n}{\uparrow}h(a_{i}) \subseteq {\uparrow}h(b)$. As $I$ is an F-ideal, $h(b) \in I$. Therefore, $b \in h^{-1}[I]$, and so $h^{-1}[I]$ is an F-ideal of $L$.


(2)$\Rightarrow$(3): Let $h^{-1}[I]$ be an F-ideal of $L$ for each F-ideal $I$ of $K$. We show that $h^{-1}[F]$ is an optimal filter of $L$ for each optimal filter $F$ of $K$. Suppose that $F$ is an optimal filter of $L$. Since $h$ is a meet semi-lattice homomorphism, $h^{-1}[F]$ is a filter of $L$. Moreover, by Proposition \ref{proposition1}, $K - F$ is an F-ideal of $K$. Therefore, $h^{-1}[K-F]=L-h^{-1}[F]$ is an F-ideal of $L$. This, by Proposition \ref{proposition1}, means that $h^{-1}[F]$ is an optimal filter of $L$.

(3)$\Rightarrow$(1): Let $h^{-1}[F]$ be an optimal filter of $L$ for each optimal filter $F$ of $K$ and let $\displaystyle{\bigcap_{i = 1}^{n}}{\uparrow}a_{i} \subseteq {\uparrow}b$ for $a_{1}, \ldots, a_{n}, b \in L$. If $\displaystyle{\bigcap_{i = 1}^{n}}{\uparrow}h(b_{i}) \not \subseteq {\uparrow}h(b)$, then, by Lemma \ref{F-ideal}, $h(b)$ does not belong to the F-ideal generated by $h(a_{1}), \ldots, h(a_{n})$. By the optimal filter lemma, there is an optimal filter $G$ of $K$ such that $h(b) \in G$ and $h(a_{1}), \ldots, h(a_{n})\notin G$.  Therefore, $a_{1}, \ldots, a_{n}\notin h^{-1}[G]$. But $h^{-1}[G]$ is an optimal filter of $L$. Thus, $L - h^{-1}[G]$ is an F-ideal of $L$, so $b \in L - h^{-1}[G]$, and so $h(b) \not \in G$, a contradiction. Consequently, $\displaystyle{\bigcap_{i = 1}^{n}}{\uparrow}h(b_{i}) \subseteq {\uparrow}h(b)$, implying that $h$ is a sup-homomorphism.
\end{proof}

We show that sup-homomorphisms are exactly those meet semi-lattice homomorphisms which preserve all existing joins. Let $L$ and $K$ be distributive meet semi-lattices and $h: L \to K$ be a meet semi-lattice homomorphism. We say that $h$ \textit{preserves all existing joins} if for each $a_{1}, \ldots, a_{n} \in L$, if $a_{1}\vee \ldots\vee a_{n}$ exists in $L$, then $h(a_{1})\vee \ldots\vee h(a_{n})$ exists in $K$ and $h(a_1\vee\ldots\vee a_n) = h(a_1)\vee\ldots\vee h(a_n)$.

\begin{proposition} \label{sup-hom-pres-exist-joins}
Let $L$ and $K$ be distributive meet semi-lattices and $h: L \to K$ be a meet semi-lattice homomorphism. Then $h$ is a sup-homomorphism iff $h$ preserves all existing joins.
\end{proposition}

\begin{proof}
Let $h$ be a sup-homomorphism, $a_{1}, \ldots, a_{n} \in L$, and $a_1\vee\ldots\vee a_n$ exist in $L$. Then $\displaystyle{\bigcap_{i=1}^n}{\uparrow}a_{i} = {\uparrow}(a_1\vee\ldots\vee a_n)$. Since $h$ is order-preserving, by the definition of  sup-homomorphisms, the last equality implies that $\displaystyle{\bigcap_{i=1}^n}{\uparrow}h(a_{i}) = {\uparrow} h(a_1\vee\ldots\vee a_n)$. Therefore, $h(a_1\vee\ldots\vee a_n)$ is the join  of $h(a_{1}), \ldots, h(a_{n})$ in $K$. Thus, $h$ preserves all existing joins. Conversely, suppose that $h$ preserves all existing joins. Let $a_1,\ldots,a_n,b\in L$ be such that $\displaystyle{\bigcap_{i=1}^n}{\uparrow}a_i \subseteq {\uparrow}b$. Then, in the lattice of filters of $L$, we have $(\displaystyle{\bigcap_{i=1}^n}{\uparrow}a_i) \vee  {\uparrow}b = {\uparrow}b$. Since  $L$ is a distributive meet semi-lattice, the lattice of filters of $L$ is distributive. Therefore, $\displaystyle{\bigcap_{i=1}^n}({\uparrow}a_i \vee  {\uparrow}b) =  {\uparrow}b.$ Since ${\uparrow}a_i \vee  {\uparrow}b = {\uparrow} (a_{i} \wedge b)$, we obtain $\displaystyle{\bigcap_{i=1}^n}{\uparrow} (a_i  \wedge b) =  {\uparrow}b$ for each $i=1,\ldots,n$. This implies that $b$ is the join of $a_{1} \wedge b, \ldots, a_{n} \wedge b$ in $L$. Therefore, since $h$ preserves all existing joins, the join of $h(a_{1} \wedge b), \ldots, h(a_{n} \wedge b)$ exists in $K$ and is equal $h(b)$. Thus, $\displaystyle{\bigcap_{i=1}^n}{\uparrow} h(a_i  \wedge b) =  {\uparrow}h(b)$, which means that $${\uparrow} h(b) = \bigcap_{i=1}^n{\uparrow}( h(a_i) \wedge h(b) )= \bigcap_{i=1}^n({\uparrow}h(a_i) \vee  {\uparrow}h(b)).$$ Using the distributivity of the lattice of filters of $K$, we obtain
$${\uparrow} h(b) = {\uparrow}h(b) \vee \bigcap_{i=1}^n{\uparrow}h(a_i).$$ Consequently, $\displaystyle{\bigcap_{i=1}^n}{\uparrow}h(a_i) \subseteq {\uparrow} h(b)$, and so $h$ is a sup-homomorphism.
\end{proof}


\begin{remark}
In \cite{Han03} Hansoul introduced the notion of a weak morphism for distributive join semi-lattices, which is a join semi-lattice homomorphism preserving all existing meets. It follows from Proposition \ref{sup-hom-pres-exist-joins} that $h:L\to K$ is a sup-homomorphism iff $h:L^d\to K^d$ is a weak morphism.
\end{remark}

\begin{lemma} \label{1-1}
Let $L$ and $K$ be distributive meet semi-lattices. If  $h:L \to K$ is a 1-1 sup-homomorphism, then
$$\bigcap_{i = 1}^{n}{\uparrow}h(a_{i}) \subseteq {\uparrow}h(b) \text{ implies } \bigcap_{i = 1}^{n}{\uparrow}a_{i} \subseteq {\uparrow}b$$ for each $a_{1}, \ldots, a_{n}, b \in L$.
\end{lemma}

\begin{proof}
Suppose that $\displaystyle{\bigcap_{i = 1}^{n}}{\uparrow}h(a_{i}) \subseteq {\uparrow}h(b)$. If $d \in \displaystyle{\bigcap_{i = 1}^{n}}{\uparrow}a_{i}$, then $h(d) \in \displaystyle{\bigcap_{i = 1}^{n}}{\uparrow}h(a_{i})$. Therefore, $h(d) \in {\uparrow}h(b)$. Thus, $h(b) \leq h(d)$. Since $h$ is 1-1, this implies that $b \leq d$. It follows that $d \in {\uparrow}b$, and so $\displaystyle{\bigcap_{i = 1}^{n}}{\uparrow}a_{i} \subseteq {\uparrow}b$.
\end{proof}

\begin{proposition} \label{extension}
Let $L$ and $K$ be distributive meet semi-lattices. If $h: L \to K$ is a sup-homomorphism, then there is a unique lattice homomorphism $D(h): D(L) \to D(K)$ such that $D(h) \circ \sigma = \sigma \circ h$. Moreover, if $h$ is 1-1, then so is $D(h)$.
\end{proposition}

\begin{proof}
Suppose that $a_{1}, \ldots, a_{n}, b_{1}\ldots, b_{m} \in L$ are such that $\sigma(a_{1}) \cup \ldots \cup \sigma(a_{n}) = \sigma(b_{1}) \cup \ldots \cup \sigma(b_{m})$. Then, for each $1 \leq i \leq n$, we have $\sigma(a_{i}) \subseteq   \sigma(b_{1}) \cup \ldots \cup \sigma(b_{m})$. Therefore, by Lemma \ref{sigmaup}, $\displaystyle{\bigcap_{j = 1}^{m}} {\uparrow}b_{j} \subseteq {\uparrow}a_{i}$. Since $h$ is a sup-homomorphism, this implies that $\displaystyle{\bigcap_{j = 1}^{m}} {\uparrow}h(b_{j}) \subseteq {\uparrow}h(a_{i})$. Thus, $\sigma(h(a_{i})) \subseteq \sigma(h(b_{1})) \cup \ldots \cup \sigma(h(b_{m}))$, and so $\sigma(h(a_{1})) \cup \ldots \cup \sigma(h(a_{n})) \subseteq \sigma(h(b_{1})) \cup \ldots \cup \sigma(h(b_{m}))$. By a similar argument we obtain that $\sigma(h(b_{1})) \cup \ldots \cup \sigma(h(b_{m})) \subseteq \sigma(h(a_{1})) \cup \ldots \cup \sigma(h(a_{n}))$. Consequently, $\sigma(h(a_{1})) \cup \ldots \cup \sigma(h(a_{n})) = \sigma(h(b_{1})) \cup \ldots \cup \sigma(h(b_{m}))$, and so we can define $D(h):D(L)\to D(K)$ by
$$D(h)(\sigma(a_{1}) \cup \ldots \cup \sigma(a_{n})) = \sigma(h(a_{1})) \cup \ldots \cup \sigma(h(a_{n})).$$
That $D(h)$ is a lattice homomorphism from $D(L)$ to $D(K)$ and that $D(h) \circ \sigma = \sigma \circ h$ is obvious.

If $k: D(L) \to D(K)$ is a lattice homomorphism such that $k \circ \sigma = \sigma \circ h$, then $D(h)(\sigma(a_{1}) \cup \ldots \cup \sigma(a_{n})) = \sigma(h(a_{1})) \cup \ldots \cup \sigma(h(a_{n})) = k(\sigma(a_{1})) \cup \ldots \cup k(\sigma(a_{n})) = k(\sigma(a_{1}) \cup \ldots \cup \sigma(a_{n}))$. Therefore, $k = D(h)$.

Finally, suppose that $h$ is 1-1. If $D(h)(\sigma(a_{1}) \cup \ldots \cup \sigma(a_{n})) = D(h)(\sigma(b_{1}) \cup \ldots \cup \sigma(b_{m}))$, then $\sigma(h(a_{1})) \cup \ldots \cup \sigma(h(a_{n})) = \sigma(h(b_{1})) \cup \ldots \cup \sigma(h(b_{m}))$. This, by Lemma \ref{sigmaup}, means that
$\displaystyle{\bigcap_{j = 1}^{m}}{\uparrow} h(b_{j})\subseteq {\uparrow}h(a_{i})$ for each $1\leq i\leq n$ and $\displaystyle{\bigcap_{i = 1}^{n}}{\uparrow} h(a_{i})\subseteq {\uparrow}h(b_{j})$ for each $1\leq j\leq m$. Since $h$ is 1-1, by Lemma \ref{1-1}, we obtain that $\displaystyle{\bigcap_{j = 1}^{m}}{\uparrow} b_{j}\subseteq {\uparrow}a_{i}$ for each $1 \leq i \leq n$ and $\displaystyle{\bigcap_{i = 1}^{n}}{\uparrow} a_{i}\subseteq {\uparrow}b_{j}$ for each $1 \leq j \leq m$. Therefore, $\sigma(a_{1}) \cup \ldots \cup \sigma(a_{n}) = \sigma(b_{1}) \cup \ldots \cup \sigma(b_{m})$.  Thus, $D(h)$ is 1-1.
\end{proof}

As a consequence, we obtain that taking the distributive envelope of a distributive meet semi-lattice can be extended to a functor $D$ from the category of distributive meet semi-lattices and sup-homomorphisms to the category of distributive lattices and lattice homomorphisms.

Noting that if $K$ is a distributive lattice, then $D(K)$ is (isomorphic to) $K$, the following is an immediate consequence of Proposition \ref{extension}.

\begin{corollary} \label{D(L)}
Let $L$ be a distributive meet semi-lattice and $D$ be a distributive lattice. If $h: L \to D$ is a sup-homomorphism, then there is a unique lattice homomorphism $D(h): D(L) \to D$ such that $D(h) \circ \sigma = h$. Moreover, if $h$ is 1-1, then so is $D(h)$.
\end{corollary}

It follows that the functor $D$ is left adjoint to the inclusion functor. Consequently, the category of distributive lattices and lattice homomorphisms is a reflective subcategory of the category of distributive meet semi-lattices and sup-homomorphisms.

\begin{theorem}\label{abst-char}
Let $L$ be a distributive meet semi-lattice. The distributive envelope $D(L)$ of $L$ is up to isomorphism the unique distributive lattice $E$ for which there is a 1-1 sup-homomorphism $e: L \to E$ such that for each distributive lattice $D$ and a 1-1 sup-homomor\-phism $h:L \to D$, there is a unique 1-1 lattice homomorphism $k: E \to D$ with $k\circ e = h$.
\end{theorem}

\begin{proof}
As an immediate consequence of Corollary \ref{D(L)}, we obtain that $D(L)$ has the property stated in the theorem because $\sigma:L\to D(L)$ is a 1-1 sup-homomorphism and for each distributive lattice $D$ and a 1-1 sup-homomor\-phism $h: L \to D$, there exists a unique 1-1 lattice homomorphism $D(h):  D(L) \to D$ with $D(h)\circ \sigma = h$. Now suppose that $E$ is a distributive lattice for which there is a 1-1 sup-homomorphism $e: L \to E$ such that for each distributive lattice $D$ and a 1-1 sup-homomorphism $h:L \to D$, there is a unique 1-1 lattice homomorphism $k: E \to D$ with $k\circ e = h$. Then there is a 1-1 lattice homomorphism $k:E\to D(L)$ such that $k\circ e=\sigma$. Also, by Corollary \ref{D(L)}, there is a 1-1 lattice homomorphism $D(e): D(L) \to E$ such that $D(e)\circ \sigma = e$. First we show that each element $a$ of $E$ has the form $e(a_1)\vee\ldots\vee e(a_n)$ for some $a_1,\ldots,a_n\in L$. Let $a\in E$. Then $k(a)\in D(L)$. Therefore, there exist $a_1,\ldots,a_n\in L$ such that $k(a)=\sigma(a_1)\cup\ldots\cup\sigma(a_n)$. Since $k\circ e=\sigma$, we have $D(e)(\sigma(a_1)\cup\ldots\cup\sigma(a_n))=D(e)(k(e(a_1))\vee\ldots\vee k(e(a_n)))$. As $D(e)$ is 1-1, the last equality implies $\sigma(a_1)\cup\ldots\cup\sigma(a_n)=k(e(a_1))\vee\ldots\vee k(e(a_n)=k(e(a_1)\vee\ldots\vee e(a_n))$. Because $k$ is 1-1, we obtain $a=e(a_1)\vee\ldots\vee e(a_n)$. Next we show that $E$ is isomorphic to $D(L)$. Since $k\circ e=\sigma$ and $D(e)\circ\sigma=e$, we have $D(e)(k(e(a_1)\vee\ldots\vee e(a_n)))=D(e)(k(e(a_1))\vee\ldots\vee k(e(a_n)))=D(e)(\sigma(a_1)\cup\ldots\cup\sigma(a_n))=D(e)(\sigma(a_1))\cup\ldots\cup D(e)(\sigma(a_n))=e(a_1)\vee\ldots\vee e(a_n)$ and  $k(D(e)(\sigma(a_1)\cup\ldots\cup\sigma(a_n)))=k(D(e)(\sigma(a_1))\vee \ldots\vee D(e)(\sigma(a_n)))=k(e(a_1)\vee\ldots\vee e(a_n))=k(e(a_1))\vee\ldots\vee k(e(a_n))=\sigma(a_1)\cup\ldots\cup\sigma(a_n)$. Therefore, the composition $D(e)\circ k$ is the identity function on $E$ and the composition $k\circ D(e)$ is the identity function on $D(L)$. Thus, $k:E\to D(L)$ and $D(e):D(L)\to E$ are inverses of each other, and so $E$ is isomorphic to $D(L)$.
\end{proof}

Theorem \ref{abst-char} provides an abstract characterization of the distributive envelope of a distributive meet semi-lattice. A similar characterization can also be found in Hansoul \cite{Han03} for distributive join semi-lattices. 
\section{Generalized Priestley spaces}

In this section we introduce one of the main concepts of this paper,
that of \emph{generalized Priestley space}. We show how to construct
the generalized Priestley space $L_*$ from a bounded distributive
meet semi-lattice $L$, and conversely, how a generalized Priestley
space $X$ gives rise to the bounded distributive meet semi-lattice
$X^*$. We further prove that a bounded distributive meet
semi-lattice $L$ is isomorphic to ${L_*}^*$, thus providing a new
representation theorem for bounded distributive meet semi-lattices.
Furthermore, we show that a generalized Priestley space $X$ is
order-isomorphic and homeomorphic to ${X^*}_*$.

Let $L$ be a bounded distributive meet semi-lattice and let $D(L)$
be its distributive envelope. Then $D(L)$ is a bounded distributive
lattice. We let
$$\text{$L_{*} = \mathrm{Opt}(L)$,\hspace{0.5cm} $L_{+} =
\mathrm{Pr}(L)$, \hspace{0.3cm} and \hspace{0.3cm} $D(L)_{*} =
\mathrm{Pr}(D(L))$.}$$ By Lemma \ref{lemma3}, $L_{+} \subseteq
L_{*}$, and by Corollary \ref{isom-1}, $\la L_{*} \subseteq \ra
\simeq \la D(L)_{*}, \subseteq \ra$. We recall that $\phi: L \to
\mathcal{P}(L_{*})$ is defined by $\phi(a) = \{x \in L_{*}: a \in
x\}$. We define $\phi_D: D( L) \to \mathcal{P}(D(L)_*)$ by
$\phi_D(A) = \{x \in D(L)_{*}: A\in x\}$. For $a \in L$ and $x\in
D(L)_*$ we have $$\text{$x\in\phi_D(\sigma(a))\;$ iff $\;
\sigma(a)\in x\; $ iff $\; a\in\sigma^{-1}(x)$.}$$ Therefore,
$$\phi(a) = \{\sigma^{-1}(x): x \in \phi_{D}(\sigma(a))\}.$$ Let
$\mathfrak{B}_L = \phi[L]$. Then $h:\mathfrak{B}_L\to
\{\phi_{D}(\sigma(a)): a \in L\}$ given by
$h(\phi(a))=\phi_D(\sigma(a))$ is a bounded meet semi-lattice
isomorphism, so $\mathfrak{B}_L$ and $\la\{\phi_{D}(\sigma(a)): a
\in L\},\cap,D(L)_*,\emptyset\ra$ are isomorphic to each other and
to $L$.

We recall that $\{\phi_{D}(A)-\phi_D(B): A,B \in D(L)\}$ is a basis for the Priestley topology $\tau_P$ on $D(L)_*$.

\begin{proposition}\label{proposition2}
Let $L$ be a bounded distributive meet semi-lattice and let $D(L)$
be its distributive envelope. Then the Priestley topology on
$D(L)_{*}$ has $\{\phi_{D}(\sigma(a)): a \in
L\}\cup\{\phi_D(\sigma(b))^c: b \in L\}$ as a subbasis.
\end{proposition}

\begin{proof}
Let $A,B\in D(L)$. Then there exist $a_1, \ldots ,a_n , b_1, \ldots,
b_m\in L$ such that $A=\displaystyle{\bigcup_{i=1}^n} \sigma(a_i)$
and $B=\displaystyle{\bigcup_{j=1}^m}\sigma(b_j)$. Therefore,
\[
\phi_D(A)=\phi_D(\displaystyle{\bigcup_{i=1}^n}\sigma(a_i))=
\displaystyle{\bigcup_{i=1}^n} \phi_D(\sigma(a_i)) \mbox{ and } \phi_D(B)
=\phi_D(\displaystyle{\bigcup_{j=1}^m}\sigma(b_j))=
\displaystyle{\bigcup_{j=1}^m}\phi_D(\sigma(b_j)).
\] 
Thus,
$\phi_{D}(A)-\phi_D(B)=\displaystyle{\bigcup_{i=1}^n} \phi_D(\sigma(a_i))-\displaystyle{\bigcup_{j=1}^m}\phi_D(\sigma(b_j))= \displaystyle{\bigcup_{i=1}^n}(\phi_D(\sigma(a_i))\cap\displaystyle{\bigcap_{j=1}^m}\phi_D(\sigma(b_j))^c)$. It follows that the elements of the basis $\{\phi_{D}(A)-\phi_D(B): A,B \in D(L)\}$ of $D(L)_*$ are finite intersections of the elements of $\{\phi_{D}(\sigma(a)): a \in L\} \cup \{\phi_D(\sigma(b))^c: b \in L\}$. Consequently, $\{\phi_{D}(\sigma(a)): a \in L\}\cup\{\phi_D(\sigma(b))^c: b \in L\}$ is a subbasis for the Priestley topology on $D(L)_*$.
\end{proof}

We  define a topology $\tau$ on $L_{*}$ by letting $\{\phi(a): a \in
L\} \cup \{\phi(b)^c: b\in L\}$ be a subbasis for~$\tau$.

\begin{lemma}\label{Priestley-1}
Let $L$ be a bounded distributive meet semi-lattice and let $D(L)$
be its distributive envelope.
\begin{enumerate}
\item $\la L_{*},\tau,\subseteq\ra$ is order-isomorphic and
homeomorphic to $\la D(L)_{*},\tau_P,\subseteq\ra$.
\item $\la L_{*},\tau,\subseteq\ra$ is a Priestley space.
\item $L_{+}$ is dense in $\la L_{*},\tau\ra$.
\end{enumerate}
\end{lemma}

\begin{proof}
(1) By Corollary \ref{isom-1}, $\la L_{*},\subseteq \ra$ is order-isomorphic to $\la D(L)_{*},\subseteq \ra$, and by Proposition \ref{proposition2}, $\la L_{*}, \tau \ra$ is homeomorphic to $\la D(L)_{*}, \tau_P \ra$. The result follows.

(2) follows from (1) and the well-known fact that $\la
D(L)_{*},\tau_P,\subseteq\ra$ is a Priestley space.

(3) Since $\mathcal{S}=\{\phi(a): a \in L\} \cup \{\phi(b)^c: b\in
L\}$ is a subbasis for $\tau$ and $\{\phi(a): a \in L\}$ is closed
under finite intersections, an element of the basis for $\tau$ that
$\mathcal{S}$ generates has the form $\phi(a) \cap
\phi(b_{1})^{c}\cap \ldots \cap \phi(b_n)^{c}$ for some $a,b_1,
\ldots , b_n\in L$. If $\phi(a) \cap \phi(b_{1})^{c}\cap \ldots \cap
\phi(b_n)^{c}\neq\emptyset$, then $\phi(a)\not\subseteq
\displaystyle{\bigcup_{i=1}^n}\phi(b)$. Therefore, by Lemma
\ref{basic}, $\displaystyle{\bigcap_{i=1}^n}{\uparrow}b_i \not
\subseteq {\uparrow}a$. Thus, by Lemma \ref{sigmaup}, $\sigma(a)
\not \subseteq \displaystyle{\bigcup_{i=1}^n}\sigma(b_i)$. Hence,
there is $y \in L_{+}$ such that $a \in y$ and $b_1, \ldots, b_n
\not \in y$. It follows that $\phi(a) \cap \phi(b_{1})^{c}\cap
\ldots \phi(b_n)^{c} \cap L_{+} \not = \emptyset$, and so $L_{+}$ is
dense in $L_*$.
\end{proof}

\begin{lemma}\label{Priestley-2}
For a bounded distributive meet semi-lattice $L$ we have that every
open upset of $L_*$ is a union of elements of $\mathfrak{B}_L$.
\end{lemma}

\begin{proof}
Let $U$ be an open upset of $L_*$ and let $x\in U$. It is
sufficient to find $a\in L$ such that $x\in\phi(a)\subseteq U$.
For each $y\notin U$ we have $x\not\subseteq y$. Therefore, there
is $a_y\in L$ such that $a_y\in x$ and $a_y\notin y$. Thus,
$\displaystyle{\bigcap}\{\phi(a_y): y\notin U\}\cap
U^c=\emptyset$. This by compactness of $L_*$ implies that there
exist $a_1,\dots,a_n\in L$ such that $\phi(a_1)\cap\cdots
\cap\phi(a_n)\cap U^c=\emptyset$. Moreover, $x\in\phi(a_i)$ for
each $i\leq n$. Therefore, $x\in\phi(a_1\wedge\cdots\wedge
a_n)\subseteq U$, and so there exists $a=a_1\wedge\cdots\wedge
a_n$ in $L$ with $x\in\phi(a)\subseteq U$.
\end{proof}

Let $D(\mathfrak{B}_L)$ denote the distributive lattice generated
(in $\mathcal{P}(L_*)$) by $\mathfrak{B}_L$. Then $A\in
D(\mathfrak{B}_L)$ iff $A$ is a finite union of elements of
$\mathfrak{B}_L$. Let also $\mathfrak{CU}(L_*)$ denote the lattice
of clopen upsets of $L_*$.

\begin{theorem}\label{isom-2}
For a bounded distributive meet semi-lattice $L$ we have $D(L)
\simeq D(\mathfrak{B}_{L}) = \mathfrak{CU}(L_{*})$.
\end{theorem}

\begin{proof}
It follows from Proposition \ref{basic-2} that $D(L)$ is isomorphic
to $D(\mathfrak{B}_L)$. By Lemma \ref{Priestley-2}, each element of
$\mathfrak{CU}(L_{*})$ is a union of elements of $\mathfrak{B}_L$.
Now since each element of $\mathfrak{CU}(L_{*})$ is compact, it is a
finite union of elements of $\mathfrak{B}_L$. Thus,
$D(\mathfrak{B}_{L}) = \mathfrak{CU}(L_{*})$.
\end{proof}

We now study some properties of the tuple $\la L_{*}, \tau,
\subseteq, L_+\ra$ which will motivate the notion of a generalized
Priestley space we introduce below.

\begin{proposition}\label{5.8}
Let $L$ be a bounded distributive meet semi-lattice. Then for each
$x \in L_{*}$, there is $y \in L_{+}$ such that $x \subseteq y$.
\end{proposition}

\begin{proof}
Since $x$ is an optimal filter of $L$, we have $L - x$ is nonempty.
Let $a \in L - x$. Then $x$ is disjoint from the ideal
${\downarrow}a$, and by the prime filter lemma, there is a prime
filter $y$ of $L$ such that $x \subseteq y$ and $a \notin y$.
\end{proof}

\begin{proposition}\label{proposition3}
Let $L$ be a bounded distributive meet semi-lattice and let $U$ be a
clopen upset of $L_*$. Then the following conditions are equivalent:
\begin{enumerate}
\item $U=\phi(a)$ for some $a \in L$.
\item $L_{*} - U = {\downarrow}(L_{+} - U)$.
\item $\mathrm{max}(L_{*} - U)\subseteq L_{+}$.
\end{enumerate}
\end{proposition}

\begin{proof}
(1)$\Rightarrow$(2): Let $U=\phi(a)$. Since $L_+\subseteq L_*$, we
have $L_+ - \phi(a)\subseteq L_* - \phi(a)$, and as $\phi(a)$ is an
upset, $L_* - \phi(a)$ is a downset, so ${\downarrow}(L_+ -
\phi(a))\subseteq L_* - \phi(a)$. Conversely, if $x \in L_* -
\phi(a)$, then $x\notin \phi(a)$. Therefore, $a \notin x$. So $x\cap
{\downarrow}a=\emptyset$, and by the prime filter lemma, there is a
prime filter $y$ of $L$ such that $x \subseteq y$ and $a \not \in
y$. Thus, $x\subseteq y$ and $y\in L_+ - \phi(a)$, implying that
$x\in {\downarrow}(L_+ - \phi(a))$. It follows that $L_* - \phi(a) =
{\downarrow}(L_+ - \phi(a))$.

(2)$\Rightarrow$(3): Let $L_* - U = {\downarrow}(L_+ - U)$ and let
$x\in \mathrm{max}(L_* - U)$. Then $x\in {\downarrow}(L_+ - U)$, and
as $x$ is a maximal point of $L_* - U$, we have $x\in L_+ - U
\subseteq L_+$.

(3)$\Rightarrow$(1): Let $U$ be a clopen upset of $L_*$ and let
$\mathrm{max}(L_* - U)\subseteq L_+ $. Then there exist $a_1, \ldots
a_n \in L$ such that $U = \phi(a_1) \cup \ldots \cup \phi(a_n)$. Let
$F$ be the filter $\displaystyle{\bigcap_{i=1}^n}{\uparrow}a_i$ and
let $I$ be the Frink ideal generated by $a_1, \ldots, a_n$. If
$F\cap I=\emptyset$, then, by the optimal filter lemma, there exists
an optimal filter $x$ of $L$ such that $F \subseteq x$ and $x \cap I
= \emptyset$. Since $a_i\in I$, we have $a_i\notin x$ for each
$i\leq n$, so $x\notin \phi(a_1) \cup \ldots \cup \phi(a_n) = U$.
Therefore, $x\in L_* - U$. Since $\la L_*,\tau,\subseteq\ra$ is a
Priestley space and $L_* - U$ is a closed subset of $L_*$, there
exists $y\in \mathrm{max}(L_{*} - U)$ such that $x\subseteq y$. But
then $y\in L_+$ and $y\notin U$. Moreover,
$\displaystyle{\bigcap_{i=1}^n}{\uparrow}a_i \subseteq x\subseteq y$
and as $y$ is prime, there is $i \leq n$ such that ${\uparrow}a_i
\subseteq y$. Hence, $y \in \phi(a_i)\subseteq U$, which is a
contradiction. Therefore, there is $a \in F \cap I$. Thus, $a\in
\displaystyle{\bigcap_{i=1}^n}{\uparrow}a_i$ and
$\displaystyle{\bigcap_{i=1}^n}{\uparrow}a_i \subseteq {\uparrow}
a$. So $\displaystyle{\bigcap_{i=1}^n}{\uparrow}a_i = {\uparrow} a$,
$a = a_1 \vee \ldots \vee a_n$, and so $\phi(a) = \phi(a_1) \cup
\ldots \cup \phi(a_n) = U$.
\end{proof}

Proposition \ref{proposition3} can be restated as follows:

\begin{proposition}\label{proposition4}
Let $L$ be a bounded distributive meet semi-lattice and let $V$ be a
clopen downset of $L_*$. Then the following conditions are
equivalent:
\begin{enumerate}
\item $V=\phi(a)^{c}$ for some $a \in L$.
\item $V = {\downarrow}(L_{+} \cap V)$.
\item $\mathrm{max}(V)\subseteq L_{+}$.
\end{enumerate}
\end{proposition}

The following is an immediate consequence of Proposition
\ref{proposition3}.

\begin{corollary}\label{5.10}
Let $L$ be a bounded distributive meet semi-lattice. Then:
$$\mathfrak{B}_{L} = \{U \in \mathfrak{CU}(L_{*}): L_{*} - U =
{\downarrow}(L_{+} - U)\} = \{U \in \mathfrak{CU}(L_{*}):
\mathrm{max}(L_{*} - U) \subseteq L_{+}\}.$$
\end{corollary}

For a family $\mathcal{F}=\{\phi(a_i):a_i\in L\}$ we recall that
$\mathcal{F}$ is \emph{directed} if for each
$\phi(a_i),\phi(a_j)\in \mathcal{F}$ there exists
$\phi(a_k)\in\mathcal{F}$ such that $\phi(a_i) \cup \phi(a_j)
\subseteq \phi(a_k)$. Clearly $\mathcal{F}$ is directed iff for
each $\phi(a_{i_1}),\dots,\phi(a_{i_n})\in \mathcal{F}$ there
exists $\phi(a_k)\in\mathcal{F}$ such that $\phi(a_{i_1}) \cup
\cdots \cup \phi(a_{i_n}) \subseteq \phi(a_k)$. For each $x\in
L_*$ let $\mathcal{I}_x=\{\phi(a):x\notin\phi(a)\}$.

\begin{proposition}\label{proposition5}
Let $L$ be a bounded distributive meet semi-lattice. Then $x\in
L_{+}$ iff $\mathcal{I}_x$ is directed.
\end{proposition}

\begin{proof}
Let $x \in L_{+}$ and let $\phi(a),\phi(b)\in\mathcal{I}_x$. Then
$x\notin\phi(a),\phi(b)$. Therefore, $a,b\notin x$. Since $x$ is a
prime filter of $L$, it follows that ${\uparrow}a\cap{\uparrow}b\not
\subseteq x$. Thus, there exists $c\in {\uparrow}a\cap{\uparrow}b$
such that $c\notin x$. Consequently, $\phi(a)\cup\phi(b)\subseteq
\phi(c)$ and $x\notin\phi(c)$. So $\phi(c)\in\mathcal{I}_x$, and so
$\mathcal{I}_x$ is directed. Conversely, suppose that
$\mathcal{I}_x$ is directed. We show that $x\in L_+$. If not, then
there exist filters $F_1$ and $F_2$ of $L$ such that $F_1 \cap F_2
\subseteq x$ but $F_1 \not \subseteq x$ and $F_2 \not \subseteq x$.
Let $a_1 \in F_1 - x$ and $a_2 \in F_2 - x$. Then
$x\notin\phi(a_1),\phi(a_2)$, and so
$\phi(a_1),\phi(a_2)\in\mathcal{I}_x$. Since $\mathcal{I}_x$ is
directed, there exists $\phi(a)\in\mathcal{I}_x$ such that
$\phi(a_1)\cup\phi(a_2)\subseteq\phi(a)$. From
$\phi(a)\in\mathcal{I}_x$ it follows that $a\notin x$, and from
$\phi(a_1)\cup\phi(a_2)\subseteq\phi(a)$ it follows that $a \in
{\uparrow}a_1\cap{\uparrow}a_2\subseteq F_1 \cap F_2\subseteq x$.
The obtained contradiction proves that $x\in L_+$.
\end{proof}

The results we have established about the dual space of a bounded
distributive meet semi-lattice $L$ justify the following definition
of a generalized Priestley space. Let $\la X, \tau, \leq\ra$ be a
Priestley space and let $X_0$ be a dense subset of $X$. For a clopen subset $U$ of $X$, we say that $X_0$ is \emph{cofinal in $U$} if $\mathrm{max}(U)\subseteq X_0$. We call a clopen upset $U$ of $X$ \emph{admissible} if $X_0$ is cofinal in $U^c$. Let $X^{*}$ denote the set of admissible clopen upsets of  $X$. We note that $U \in X^{*}$ iff $U$ is a clopen upset such that $\mathrm{max}(U^c)\subseteq X_0$, which happens iff $U$ is a clopen upset such that $U^c = {\downarrow}(X_0 - U)$. For $x\in X$ let
$\mathcal{I}_x=\{U\in X^*:x\notin U\}$.

\begin{definition}\label{g-Priestley}
{\rm We call a quadruple $X=\la X, \tau, \leq, X_0\ra$ a
\textit{generalized Priestley space} if:
\begin{enumerate}
\item $\la X, \tau, \leq\ra$ is a Priestley space.
\item $X_0$ is a dense subset of $X$.
\item For each $x \in X$ there is $y \in X_0$ such that $x \leq y$.
\item $x \in X_0$ iff $\mathcal{I}_x$ is directed.
\item For all $x, y \in X$, we have $x \leq y$ iff $(\forall U \in
X^{*})(x \in U \Rightarrow y \in U)$.
\end{enumerate}
}
\end{definition}

\begin{remark}
Clearly condition (3) of Definition \ref{g-Priestley} is equivalent
to $\mathrm{max}X$ $\subseteq X_0$, which means that $\emptyset$ is admissible. Also, when in a generalized
Priestley space $X$ we have $X_0=X$, then $X^*=\mathfrak{CU}(X)$, so
conditions (2)--(5) of Definition \ref{g-Priestley} become
redundant, and so $X$ becomes a Priestley space. Thus, the notion of
a generalized Priestley space generalizes that of a Priestley space.
\end{remark}

\begin{proposition}\label{proposition6}
For a bounded distributive meet semi-lattice $L$, the quadruple
$L_*=\la L_{*}, \tau, \subseteq, L_{+}\ra$ is a generalized
Priestley space.
\end{proposition}

\begin{proof}
It follows from Lemma \ref{Priestley-1} that $\la L_{*}, \tau,
\subseteq\ra$ is a Priestley space, and that $L_+$ is dense in $\la
L_{*}, \tau\ra$. Therefore, $L_*$ satisfies conditions (1) and (2)
of Definition \ref{g-Priestley}. By Propositions \ref{5.8} and
\ref{proposition5}, $L_*$ satisfies conditions (3) and (4) of
Definition \ref{g-Priestley}. Lastly it follows from Proposition
\ref{proposition3} that ${L_*}^*=\phi[L]$. Clearly $x\subseteq y$
iff $(\forall a\in L)(a\in x\Rightarrow b\in x)$, thus $L_*$
satisfies condition (5) of Definition \ref{g-Priestley}.
Consequently, $L_*$ is a generalized Priestley space.
\end{proof}

Since for each bounded distributive meet semi-lattice $L$ we have
${L_*}^*=\phi[L]$, we immediately obtain:

\begin{theorem}\label{proposition8} (Representation Theorem)
For each bounded distributive meet semi-lattice $L$ we have $L
\simeq {L_{*}}^{*}$; that is, for each bounded distributive meet
semi-lattice $L$ there exists a generalized Priestley space $X$ such
that $L$ is isomorphic to $X^*$.
\end{theorem}

\begin{proposition}\label{proposition7}
Let $X$ be a generalized Priestley space. Then $X^*=\la X^{*}, \cap,
X,\emptyset \ra$ is a bounded distributive meet semi-lattice.
\end{proposition}

\begin{proof}
First we show that $X^{*}$ is closed under $\cap$. If $U, V \in
X^{*}$, then $\mathrm{max}((U \cap V)^c) = \mathrm{max}(U^c \cup
V^c)\subseteq \mathrm{max}(U^c) \cup \mathrm{max}(V^c) \subseteq
X_0$. Thus, $U \cap V \in X^{*}$. Next $\mathrm{max}(X^c) =
\mathrm{max}(\emptyset)=\emptyset\subseteq X_0$, so $X\in X^*$.
Also, $\mathrm{max}(\emptyset^c)=\mathrm{max}(X)\subseteq X_0$ by
condition (3) of Definition \ref{g-Priestley}, so $\emptyset\in
X^*$. Lastly we show that $\la X^{*}, \cap \ra$ is distributive.
Let $U, V, W \in X^{*}$ with $U \cap V \subseteq W$. Then $W^c
\subseteq U^c \cup V^c$. For each $x\in\mathrm{max}(W^c)$ we have
that $x\in U^c$ or $x\in V^c$. Therefore, $W\in\mathcal{I}_x$ and
either $U\in \mathcal{I}_x$ or $V\in \mathcal{I}_x$. Since $x\in
X_0$, by condition (4) of Definition \ref{g-Priestley}, from
$W,U\in\mathcal{I}_x$ it follows that there exists
$U_x\in\mathcal{I}_x$ such that $W\cup U\subseteq U_x$; and from
$W,V\in\mathcal{I}_x$ it follows that there exists
$V_x\in\mathcal{I}_x$ such that $W\cup V\subseteq V_x$. Thus,
$W^c=\displaystyle{\bigcup}\{K_x: x\in\mathrm{max}(W^c)\}$, where
$K_x={U_x}^c$ or $K_x={V_x}^c$. Since $W^c$ is compact and each
$K_x$ is open, there exist finite subsets $A,B$ of
$\mathrm{max}(W^c)$ such that $W^c=
\displaystyle{\bigcup}\{{U_x}^c: x\in A\}
\cup\displaystyle{\bigcup}\{{V_x}^c: x\in B\}$. Let
$U'=\displaystyle{\bigcap}\{U_x: x\in A\}$ and
$V'=\displaystyle{\bigcap}\{V_x: x\in B\}$. Clearly $U\subseteq
U'$, $V\subseteq V'$, and $U,V\in X^*$. Moreover, $W^c =
\displaystyle{\bigcup}\{{U_x}^c: x\in A\}\cup
\displaystyle{\bigcup}\{{V_x}^c: x\in B\}$ implies $W = U' \cap
V'$. Thus, $\la X^{*}, \cap \ra$ is distributive. Consequently,
$\la X^{*}, \cap, X, \emptyset \ra$ is a bounded distributive meet
semi-lattice.
\end{proof}

\begin{proposition}\label{CU-finun}
Let $X$ be a generalized Priestley space. Then the closure of
$X^{*}$ under finite unions is $\mathfrak{CU}(X)$.
\end{proposition}

\begin{proof}
Clearly we have $X^{*} \subseteq \mathfrak{CU}(X)$, so the closure
of $X^{*}$ under finite unions is contained in $\mathfrak{CU}(X)$.
Conversely, suppose that $U \in \mathfrak{CU}(X)$. Since
$X,\emptyset\in X^*$, we may assume without loss of generality
that $U \not = X,\emptyset$. For each $x \in U$ and $y \not \in
U$, as $U$ is an upset, we have $x \not \leq y$. Therefore, by
condition (5) of Definition \ref{g-Priestley}, there is $V_y \in
X^{*}$ such that $x \in V_y$ and $y \not \in V_y$. Then $U^c \cap
\displaystyle{\bigcap}\{V_y: y \notin U\} = \emptyset$. Since $X$
is compact, there are $y_1, \ldots, y_n\in U^c$ such that $U^c
\cap V_{y_1}\cap \ldots \cap V_{y_n}=\emptyset$. Therefore, $x \in
V_{y_1}\cap \ldots \cap V_{y_n} \subseteq U$. Let $W_x =
V_{y_1}\cap \ldots \cap V_{y_n}$. Since $X^{*}$ is closed under
finite intersections, $W_x \in X^{*}$. Then $U =
\displaystyle{\bigcup}\{W_x: x \in U\}$, and as $X$ is compact,
there are $x_1, \ldots, x_k \in U$ such that $U = W_{x_1} \cup
\ldots \cup W_{x_k}$. Thus, $U$ is a finite union of elements of
$X^*$.
\end{proof}

\begin{corollary}\label{corollary1}
For a generalized Priestley space $X$, the family $X^{*} \cup
\{U^{c}: U \in X^{*}\}$ is a subbasis for the topology on $X$.
\end{corollary}

\begin{proof}
Since $X$ is a Priestley space, $\{U-V: U,V\in \mathfrak{CU}(X)\}$ is a basis for the topology on $X$. By Proposition \ref{CU-finun}, $U=\displaystyle{\bigcup_{i=1}^n}U_i$ and $V=\displaystyle{\bigcup_{j=1}^m}V_j$ for some $U_1,\dots,U_n,V_1,\dots,V_m\in X^*$. Therefore, $U-V=\displaystyle{\bigcup_{i=1}^n}U_i-\displaystyle{\bigcup_{j=1}^m}V_j= \displaystyle{\bigcup_{i=1}^n}(U_i\cap\displaystyle{\bigcap_{j=1}^m}V_j^c)$. Thus, the elements of the basis $\{U-V: U,V\in \mathfrak{CU}(X)\}$ of $X$ are finite intersections of elements of $X^{*} \cup \{U^{c}: U \in X^{*}\}$. Consequently, $X^{*} \cup \{U^{c}: U \in X^{*}\}$ is a subbasis for the topology on $X$.
\end{proof}

\begin{lemma}\label{lemma4}
Let $X$ be a generalized Priestley space, and let $U,U_1, \ldots,
U_n\in X^*$. Then $\displaystyle{\bigcap_{i=1}^n}{\uparrow}U_i
\subseteq {\uparrow}U$ iff $U \subseteq
\displaystyle{\bigcup_{i=1}^n}U_i$.
\end{lemma}

\begin{proof}
Suppose that $\displaystyle{\bigcap_{i=1}^n}{\uparrow}U_i \subseteq
{\uparrow}U$. We show that $U\cap X_0\subseteq
\displaystyle{\bigcup_{i=1}^n}U_i$. Let $x\in U\cap X_0$. If
$x\notin \displaystyle{\bigcup_{i=1}^n}U_i$, then
$U_1,\ldots,U_n\in\mathcal{I}_x$, so by condition (4) of Definition
\ref{g-Priestley}, there exists $V\in\mathcal{I}_x$ such that
$\displaystyle{\bigcup_{i=1}^n}U_i\subseteq V$. Therefore, $V\in
\displaystyle{\bigcap_{i=1}^n}{\uparrow}U_i$, and so $V\in
{\uparrow}U$. Thus, $U\subseteq V$, implying that $x\in V$, a
contradiction. Consequently, $x\in \displaystyle{\bigcup_{i=1}^n}
U_i$, and so $U\cap X_0 \subseteq \displaystyle{\bigcup_{i=1}^n}
U_i$. By condition (2) of Definition \ref{g-Priestley}, $X_0$ is a
dense subset of $X$. Since $U$ is open in $X$, we have $U\cap X_0$
is dense in $U$. Therefore, $U\cap X_0 \subseteq
\displaystyle{\bigcup_{i=1}^n}U_i$ implies $U \subseteq
\displaystyle{\bigcup_{i=1}^n}U_i$. Conversely, suppose that $U
\subseteq \displaystyle{\bigcup_{i=1}^n}U_i$. For $V\in X^*$ with
$V\in \displaystyle{\bigcap_{i=1}^n}{\uparrow}U_i$, we have
$\displaystyle{\bigcup_{i=1}^n}U_i \subseteq V$. Therefore,
$U\subseteq V$, implying that $V\in {\uparrow}U$. Thus,
$\displaystyle{\bigcap_{i=1}^n}{\uparrow}U_i \subseteq {\uparrow}U$.
\end{proof}

\begin{corollary}
Let $X$ be a generalized Priestley space. Then $\mathfrak{CU}(X)$ is
isomorphic to $D(X^{*})$.
\end{corollary}

\begin{proof}
By Proposition \ref{CU-finun}, $\mathfrak{CU}(X)$ is the closure of
$X^{*}$ under finite unions. Let $U_1,\ldots,U_n,$
$V_1,\ldots,V_m\in X^*$. It follows from Lemmas \ref{sigmaup} and
\ref{lemma4} that $U_1 \cup \ldots \cup U_n = V_1 \cup \ldots \cup
V_m$ iff $\sigma(U_1) \cup \ldots \cup \sigma(U_n) = \sigma(V_1)
\cup \ldots \cup \sigma(V_m)$. Thus, we can define a map
$h:\mathfrak{CU}(X) \to D(X^*)$ by $h(U_1 \cup \ldots \cup U_n) =
\sigma(U_1) \cup \ldots \cup \sigma(U_n)$. This map is clearly onto,
and it follows from Lemmas \ref{sigmaup} and \ref{lemma4} that it is
an order-isomorphism.
\end{proof}

Let $X$ be a generalized Priestley space. We define $\psi: X \to
{X^{*}}_{*}$  by $$\psi(x) = \{U \in X^{*}: x \in U\}.$$ First we
show that $\psi$ is well-defined.

\begin{proposition}\label{psi-opt-prime}
Let $X$ be a generalized Priestley space. For each $x \in X$, we
have $\psi(x)$ is an optimal filter of $X^{*}$. Moreover, if $x \in
X_0$, then $\psi(x)$ is a prime filter of $X^{*}$.
\end{proposition}

\begin{proof}
Let $X$ be a generalized Priestley space. 
By Proposition \ref{proposition7}, $\la X^*,\cap,X,\emptyset \ra$
is a bounded distributive meet semi-lattice. Clearly $\psi(x)$ is a
filter of $X^*$ and $X^*-\psi(x)$ is nonempty. We show that
$X^*-\psi(x)$ is an F-ideal of $X^*$. Suppose that $U_1, \ldots, U_n
\in X^{*}-\psi(x)$, $U \in X^{*}$, and
$\displaystyle{\bigcap_{i=1}^n} {\uparrow}U_i \subseteq
{\uparrow}U$. By Lemma \ref{lemma4}, $U\subseteq
\displaystyle{\bigcup_{i=1}^n}U_i$. Since $x\notin
\displaystyle{\bigcup_{i=1}^n}U_i$, it follows that $x\notin U$.
Therefore, $U \in X^* - \psi(x)$, so $X^* - \psi(x)$ is an F-ideal,
and so, by Proposition \ref{proposition1}, $\psi(x)$ is an optimal
filter. Now suppose that $x \in X_0$ and that $\psi(x)$ is not a
prime filter of $X^*$. Then there exist filters $F_1$ and $F_2$ of
$X^{*}$ such that $F_1 \cap F_2 \subseteq \psi(x)$, but $F_1 \not
\subseteq \psi(x)$ and $F_2 \not \subseteq \psi(x)$. Let $U_1 \in
F_1 - \psi(x)$ and $U_2 \in F_2 - \psi(x)$. Then $x \notin U_1,U_2$,
and so $U_1,U_2 \in \mathcal{I}_x$. By condition (4) of Definition
\ref{g-Priestley}, there exists $V\in \mathcal{I}_x$ such that
$U_1\cup U_2 \subseteq V$. Hence, $V\in{\uparrow}U_1 \cap{\uparrow}
U_2 \subseteq F_1 \cap F_2\subseteq \psi(x)$. Thus, $x \in V$, a
contradiction. We conclude that $\psi(x)$ is a prime filter of
$X^*$.
\end{proof}

\begin{proposition}\label{psi-bij}
The map $\psi: X \to {X^{*}}_{*}$ is 1-1 and onto. Moreover, if $P$
is a prime filter of $X^{*}$, then $\psi^{-1}(P) \in X_0$.
\end{proposition}

\begin{proof}
It follows from condition (5) of Definition \ref{g-Priestley} that
$\psi$ is 1-1. We show that $\psi$ is onto. Suppose that $P$ is an
optimal filter of $X^{*}$. Let $I = X^{*} - P$. By Proposition
\ref{proposition1}, $I$ is an F-ideal of $X^*$. Let $G$ be the
filter of $\mathfrak{CU}(X)$ generated by $P$ and $J$ be the ideal
of $\mathfrak{CU}(X)$ generated by $I$. We claim that $G \cap J =
\emptyset$. If not, then there exist $V \in \mathfrak{CU}(X)$, $U
\in P$, and $U_1, \ldots, U_n \in I$ such that $U \subseteq V$ and
$V \subseteq U_1 \cup \ldots \cup U_k$. By Lemma \ref{lemma4},
$\displaystyle{\bigcap_{i=1}^n}{\uparrow}U_i\subseteq {\uparrow}V
\subseteq {\uparrow}U$. Since $I$ is an F-ideal, we have $U \in I$,
so $U \not \in P$, a contradiction. Thus, by the prime filter lemma,
there is a prime filter $F$ of $\mathfrak{CU}(X)$ such that $G
\subseteq F$ and $F \cap J = \emptyset$. By the Priestley duality,
there exists $x \in X$ such that $F = \{U \in \mathfrak{CU}(X): x
\in U\}$. We show that $P = F \cap X^{*}$. It is clear that $P
\subseteq F \cap X^{*}$. Conversely, if $U \in F \cap X^{*}$ and $U
\not \in P$, then $U \in I$, which is a contradiction since $F$ is
disjoint from $I$. Thus, $P = F \cap X^{*}=\{U\in X^*:x\in U\}$.
Consequently $\psi(x) = P$, and so $\psi$ is onto.

Now suppose that $P$ is a prime filter of $X^{*}$. Since $\psi$ is
onto, there exists $x \in X$ such that $\psi(x) = P$. If $x \notin
X_0$, then by condition (4) of Definition \ref{g-Priestley},
$\mathcal{I}_x$ is not directed, so there exist
$U,V\in\mathcal{I}_x$ such that for no $W\in \mathcal{I}_x$ we
have $U\cup V\subseteq W$. Therefore, for each $W\in X^*$, from
$W\in{\uparrow}U\cap{\uparrow}V$ it follows that $x\in W$, and so
$W\in\psi(x)$. Thus, ${\uparrow}U\cap{\uparrow}V \subseteq
\psi(x)$, and as $\psi(x)$ is a prime filter of $X^*$,
${\uparrow}U \subseteq \psi(x)$ or ${\uparrow}V \subseteq
\psi(x)$. Consequently, $U\in \psi(x)$ or $V\in\psi(x)$, and so $x
\in U$ or $x \in V$, a contradiction. Thus, $x \in X_0$.
\end{proof}

\begin{theorem}\label{homeom}
For a generalized Priestley space $X$, the map $\psi: X \to
{X^{*}}_{*}$ is an order-isomorphism and a homeomorphism. Moreover,
$\psi[X_0] = {X^{*}}_{+}$.
\end{theorem}

\begin{proof}
It follows from Propositions \ref{psi-opt-prime} and \ref{psi-bij}
that $\psi$ is a bijection and that $\psi[X_0] = {X^{*}}_{+}$.
Condition (5) of Definition \ref{g-Priestley} implies that for each
$x, y \in X$, we have $x \leq y$ iff $\psi(x) \subseteq \psi(y)$.
Thus, $\psi$ is an order-isomorphism. We show that $\psi$ is a
homeomorphism. By Corollary \ref{corollary1}, $X^{*} \cup \{U^{c}: U
\in X^{*}\}$ is a subbasis for the topology on $X$, and $\{\phi(U):
U \in X^{*}\} \cup \{\phi(U)^{c}: U \in X^{*}\}$ is a subbasis for
the topology on ${X^{*}}_{*}$. For $U\in X^*$ we have: $$x \in
\psi^{-1}(\phi(U))\; \mbox{ iff }\; \psi(x) \in \phi(U) \; \mbox{
iff }\;  U \in \psi(x)\;  \mbox{ iff } \; x \in U.$$ Thus,
$\psi^{-1}(\phi(U))=U$ and $\psi^{-1}(\phi(U)^c)=U^c$. It follows
that $\psi$ is continuous. Now since $\psi$ is a continuous map
between compact Hausdorff spaces, $\psi$ is a homeomorphism.
\end{proof}

\section{Generalized Esakia spaces}

In this section we introduce our second main concept of the paper,
that of \emph{generalized Esakia space}, which we obtain by
augmenting the concept of generalized Priestley space. As with
bounded distributive meet semi-lattices, we show how to construct
the generalized Esakia space $L_*$ from a bounded implicative meet
semi-lattice $L$, and conversely, how a generalized Esakia space
$X$ gives rise to the bounded implicative meet semi-lattice $X^*$.
We also show that a bounded implicative meet semi-lattice $L$ is
isomorphic to ${L_*}^*$, thus providing a new representation
theorem for bounded implicative meet semi-lattices, and prove that
a generalized Esakia space $X$ is order-isomorphic and
homeomorphic to ${X^*}_*$. We conclude the section by showing that
the distributive envelope of a bounded implicative meet
semi-lattice may not be a Heyting algebra.

Let $L$ be a bounded implicative meet semi-lattice. For $a,b\in L$,
let $$\phi(a)\to\phi(b) = [{\downarrow}(\phi(a)-\phi(b))]^c=\{x\in
L_*:{\uparrow}x\cap \phi(a)\subseteq\phi(b)\}.$$

\begin{lemma}\label{leo-1}
Let $L$ be a bounded implicative meet semi-lattice and let $a, b \in
L$. Then $\phi(a \to b) = \phi(a) \to \phi(b)$.
\end{lemma}

\begin{proof}
First suppose that $x\in\phi(a \to b)$ and $y\in{\uparrow}x\cap
\phi(a)$. Then $a\to b\in x$, $x\subseteq y$, and $a\in y$.
Therefore, $a,a\to b\in y$, so $b\in y$, and so $y\in\phi(b)$. Thus,
${\uparrow}x \cap \phi(a) \subseteq \phi(b)$, so
$x\in\phi(a)\to\phi(b)$, and so $\phi(a \to b) \subseteq \phi(a) \to
\phi(b)$. Now suppose that $x \in \phi(a) \to \phi(b)$. If $x\notin
\phi(a\to b)$, then $a \to b \notin x$. Let $F$ be the filter of $L$
generated by $\{a\} \cup x$. If there is $c \in F \cap {\downarrow}
b$, then there is $d \in x$ such that $a \wedge d \leq c \leq b$.
Therefore, $d \leq a \to b$, and so $a \to b \in x$, a
contradiction. Thus, $F \cap {\downarrow} b=\emptyset$, and by the
prime filter lemma, there is $y \in L_+ \subseteq L_{*}$ such that
$F\subseteq y$ and $b \notin y$. It follows that $x\subseteq y$,
$a\in y$, and $b\notin y$. Therefore, $y\in {\uparrow}x\cap\phi(a)$
and $y\notin\phi(b)$. Thus, ${\uparrow}x \cap \phi(a) \not\subseteq
\phi(b)$, and so $x \notin \phi(a)\to \phi(b)$, a contradiction. We
conclude that $x\in\phi(a\to b)$, so $\phi(a) \to \phi(b) \subseteq
\phi(a \to b)$, and so $\phi(a \to b) = \phi(a) \to \phi(b)$.
\end{proof}

Let $X$ be a Priestley space. We recall that each clopen $U$ in $X$ has the form $\displaystyle{\bigcup_{i=1}^n}(U_i-V_i)$, where $U_i,V_i\in\mathfrak{CU}(X)$, and that $X$ is an Esakia space whenever ${\downarrow}U$ is clopen for each clopen $U$ in $X$.

Let $X$ be a generalized Priestley space. Then each clopen $U$ in $X$ has the form $\displaystyle{\bigcup_{i=1}^n\bigcap_{j=1}^m}(U_i-V_j)$, where $U_i,V_j\in X^*$.

\begin{definition}
{\rm Let $X$ be a generalized Priestley space and $U$ be clopen in $X$. We call $U$ \emph{Esakia clopen} if $U=\displaystyle{\bigcup_{i=1}^n}(U_i-V_i)$ for some $U_1,\dots,U_n,V_1,\dots,V_n\in X^*$.}
\end{definition}

\begin{lemma} \label{lemma-Es-clo}
If $X$ is a generalized Priestley space and $U$ is Esakia clopen in $X$, then $\mathrm{max}(U)\subseteq X_0$.
\end{lemma}

\begin{proof}
Let $U$ be Esakia clopen. Then there exist $U_1,\dots,U_n,V_1,\dots,V_n\in X^*$ such that $U=\displaystyle{\bigcup_{i=1}^n}(U_i-V_i)$. Therefore, $\mathrm{max}(U)=\mathrm{max}[\displaystyle{\bigcup_{i=1}^n}(U_i-V_i)]\subseteq\displaystyle{\bigcup_{i=1}^n}\mathrm{max}(U_i-V_i)=
\displaystyle{\bigcup_{i=1}^n}\mathrm{max}(U_i\cap V_i^c)$. Since $U_i$ is an upset, $V_i^c$ is a downset, and $V_i\in X^*$, we have $\mathrm{max}(U_i\cap V_i^c) \subseteq \mathrm{max}(V_i^c)\subseteq X_0$. Therefore, $\mathrm{max}(U)\subseteq \displaystyle{\bigcup_{i=1}^n}\mathrm{max}(U_i\cap V_i^c) \subseteq \displaystyle{\bigcup_{i=1}^n}\mathrm{max}(V_i^c) \subseteq X_0$.
\end{proof}

On the other hand, it is worth pointing out that the converse of Lemma \ref{lemma-Es-clo} is not true in general.

\begin{definition}\label{def-gEs}
{\rm We call a generalized Priestely space $X$ a \textit{generalized Esakia space} if for each Esakia clopen $U$ in $X$, we have ${\downarrow}U$ is clopen.}
\end{definition}

For a generalized Esakia space $X$ and $U,V\in X^*$, let $$U\to V = [{\downarrow}(U-V)]^c=\{x\in X:{\uparrow}x\cap U\subseteq V\}.$$

\begin{proposition}\label{lemma-ims-ges}
\begin{enumerate}
\item[]
\item If $L$ is a bounded implicative meet semi-lattice, then $L_*= \la L_*,\tau,\subseteq,L_+ \ra$ is a generalized Esakia space.
\item If $X$ is a generalized Esakia space, then $X^*=\la X^*,\cap,\to,X,\emptyset\ra$ is a bounded implicative meet semi-lattice.
\end{enumerate}
\end{proposition}

\begin{proof}
(1) Suppose that $L$ is a bounded implicative meet semi-lattice. Then $L$ is a bounded distributive meet semi-lattice, and so $L_*$ is a generalized Priestley space. Let $U$ be Esakia clopen in $L_*$. Then $U=\displaystyle{\bigcup_{i=1}^n}(\phi(a_i)-\phi(b_i))$ for some $a_1,\ldots,a_n,b_1,\ldots,b_n\in L$. By Lemma \ref{leo-1}, $\phi(a_i) \to \phi(b_i) = [{\downarrow}(\phi(a_i)-\phi(b_i))]^c\in {L_*}^*$. Therefore, ${\downarrow}(\phi(a_i)-\phi(b_i))$ is clopen in $L_*$ for each $i\leq n$. Thus, ${\downarrow}U = \displaystyle{\bigcup_{i=1}^n}{\downarrow}(\phi(a_i)-\phi(b_i))$ is clopen in $L_*$, and so $L_*$ is a generalized Esakia space.

(2) Suppose that $X$ is a generalized Esakia space. Then $X$ is a generalized Priestley space, and so $\la X^*,\cap,X,\emptyset\ra$ is a bounded distributive meet semi-lattice. Let $U,V\in X^*$. Then $U-V$ is Esakia clopen. Therefore, ${\downarrow}(U - V)$ is clopen in $X$ and $\mathrm{max}{\downarrow}(U - V)=\mathrm{max}(U - V) \subseteq X_0$. Therefore, $U \to V = [{\downarrow}(U - V)]^c\in X^*$. Moreover, it is routine to verify that for $U,V,W\in X^*$ we have $U\cap W\subseteq V$ iff $W\subseteq U\to V$. Thus, $\la X^*,\cap,\to,X,\emptyset\ra$ is a bounded implicative meet semi-lattice.
\end{proof}

It follows that for each bounded implicative meet semi-lattice $L$ we have ${L_*}^*=\phi[L]$. Thus, we immediately obtain:

\begin{theorem}\label{proposition8rep} (Representation Theorem)
For each bounded implicative meet semi-lattice $L$ we have $L \simeq
{L_{*}}^{*}$; that is, for each bounded implicative meet
semi-lattice $L$ there exists a generalized Esakia space $X$ such
that $L$ is isomorphic to $X^*$.
\end{theorem}

We show that there exist generalized Esakia spaces which are not Esakia spaces. This implies that the distributive envelope $D(L)$ of a bounded
implicative meet semi-lattice $L$ may \emph{not} be a Heyting algebra.

\begin{example}

Consider the implicative meet semi-lattice $L$, its distributive envelope $D(L)$, and its dual space $L_*$ shown in Fig.3, where the black circles indicate the elements of $L$ and the white circles indicate the elements of $D(L)-L$. 


\begin{center}

$$
\begin{array}{ccc}
\begin{tikzpicture}[xscale=.5,yscale=.37,inner sep=.5mm]
\node[bull] (bot) at (0,0) [label=below:$\bot$] {};%
\node[bull] (d1) at (0,1) [label=left:$d_1$] {};%
\node[bull] (c) at (1,1) [label=right:$c$] {};%
\node[bull] (d2) at (0,2) [label=left:$d_2$] {};%
\node[bull] (12) at (1,2) {};%
\node[bull] (d3) at (0,3) {};%
\node[bull] (13) at (1,3) {};%
\node[bull] (e1) at (2,3) [label=right:$e_1$] {};%
\node[bull] (d4) at (0,4) {};%
\node[bull] (14) at (1,4) {};%
\node[bull] (24) at (2,4) {};%
\node[bull] (d5) at (0,5) {};%
\node[bull] (15) at (1,5) {};%
\node[bull] (25) at (2,5) {};%
\node[bull] (e2) at (3,5) [label=right:$e_2$] {};%
\node (v0) at (0,5.65) {};%
\node (v1v) at (1,5.65) {};%
\node (v1) at (.65,5.65) {};%
\node (v2v) at (2,5.65) {};%
\node (v2) at (1.65,5.65) {};%
\node (v3v) at (3,5.65) {};%
\node (v3) at (2.65,5.65) {};%
\node (v4) at (3.65,5.65) {};%
\node (vv0) at (0,6.5) {$\vdots$};%
\node (vv1) at (1,7) {$\vdots$};%
\node (vv2) at (2,7.5) {$\vdots$};%
\node (vv2) at (3,8) {$\vdots$};%
\node[bull] (b) at (0,7) [label=left:$b$] {};%
\node[holl] (h1) at (1,8) {};%
\node[holl] (h2) at (2,9) {};%
\node[holl] (h3) at (3,10) {};%
\node (ddl) at (3.5,10.5) {};%
\node (dd) at (4,11.25) {$\iddots$};%
\node (w0) at (1,9.5) {$\vdots$};%
\node (ww0) at (1,10.4) {};%
\node (w1) at (2,10.5) {$\vdots$};%
\node (ww1) at (2,11.4) {};%
\node (w2) at (3,11.5) {$\vdots$};%
\node (ww2) at (3,12.4) {};%
\node (wi) at (5,14) {$\vdots$};%
\node (wwi) at (5,14.4) {};%
\node (p0) at (0,9) {};%
\node (p1) at (1,10) {};%
\node (p2) at (2,11) {};%
\node[bull] (f2) at (1,11) [label=left:$f_2$] {};%
\node[bull] (212) at (2,12) {};%
\node[bull] (313) at (3,13) {};%
\node (414) at (4,14) {};%
\node (dd2l) at (3.5,13.5) {};%
\node (dd2) at (4,14.25) {$\iddots$};%
\node[bull] (g2) at (5,15) [label=right:$g_2$] {};
\node[bull] (f1) at (1,12) [label=left:$f_1$] {};%
\node[bull] (213) at (2,13) {};%
\node[bull] (314) at (3,14) {};%
\node (415) at (4,15) {};
\node (dd1l) at (3.5,14.5) {};%
\node (dd1) at (4,15.25) {$\iddots$};%
\node[bull] (g1) at (5,16) [label=right:$g_1$] {};
\node[bull] (a) at (1,13) [label=left:$a$] {};%
\node[bull] (214) at (2,14) {};%
\node[bull] (315) at (3,15) {};%
\node (416) at (4,16) {};
\node (dd0l) at (3.5,15.5) {};%
\node (dd0) at (4,16.25) {$\iddots$};%
\node[bull] (top) at (5,17) [label=above:$\top$] {};%
\draw (bot) -- (d1) -- (d2) -- (d3) -- (d4) -- (d5) -- (v0);%
\draw (c) -- (12) -- (13) -- (14) -- (15) -- (v1v);%
\draw (e1) -- (24) -- (25) -- (v2v);%
\draw (e2) -- (v3v);%
\draw (bot) -- (c);%
\draw (d1) -- (12) -- (e1);%
\draw (d2) -- (13) -- (24) -- (e2);%
\draw (d3) -- (14) -- (25) -- (v3);%
\draw (d4) -- (15) -- (v2);%
\draw (d5) -- (v1);%
\draw (b) -- (h1) -- (h2) -- (h3) -- (ddl);%
\draw (ww0) -- (f2) -- (f1) -- (a);%
\draw (ww1) -- (212) -- (213) -- (214);%
\draw (ww2) -- (313) -- (314) -- (315);%
\draw (f2) -- (212) -- (313) -- (dd2l);%
\draw (f1) -- (213) -- (314) -- (dd1l);%
\draw (a) -- (214) -- (315) -- (dd0l);%
\draw (wwi) -- (g2) -- (g1) -- (top);%
\end{tikzpicture}&\ \ \ \ \ \ \ &
\begin{tikzpicture}[yscale=.65,inner sep=.5mm]
\node[bull] (a) at (0,0) [label=left:${\uparrow}a$] {};%
\node[bull] (f1) at (0,1) [label=left:${\uparrow}f_1$] {};%
\node[bull] (f2) at (0,2) [label=left:${\uparrow}f_2$] {};%
\node (f) at (0,2.4) {};%
\node (vf) at (0,3.3) {$\vdots$};%
\node[bull] (P) at (1,3) [label=right:$P$] {};%
\node[holl] (Q) at (0,4) [label=left:$Q$] {};%
\node[bull] (b) at (0,5) [label=left:${\uparrow}b$] {};%
\node (ve) at (1,4.7) {$\vdots$};%
\node (vb) at (0,6.1) {$\vdots$};%
\node (e) at (1,5.6) {};%
\node (d) at (0,6.6) {};%
\node[bull] (e4) at (1,6) [label=right:${\uparrow}e_4$] {};%
\node[bull] (d4) at (0,7) [label=left:${\uparrow}d_4$] {};%
\node[bull] (e3) at (1,7) [label=right:${\uparrow}e_3$] {};%
\node[bull] (d3) at (0,8) [label=left:${\uparrow}d_3$] {};%
\node[bull] (e2) at (1,8) [label=right:${\uparrow}e_2$] {};%
\node[bull] (d2) at (0,9) [label=left:${\uparrow}d_2$] {};%
\node[bull] (e1) at (1,9) [label=right:${\uparrow}e_1$] {};%
\node[bull] (d1) at (0,10) [label=left:${\uparrow}d_1$] {};%
\node[bull] (c) at (1,10) [label=right:${\uparrow}c$] {};%
\draw (a) -- (f1) -- (f2) -- (f);%
\draw (Q) -- (b);%
\draw (d) -- (d4) -- (d3) -- (d2) -- (d1);%
\draw (e) -- (e4) -- (e3) -- (e2) -- (e1) -- (c);%
\draw (P) -- (Q);%
\draw (e4) -- (d4);%
\draw (e3) -- (d3);%
\draw (e2) -- (d2);%
\draw (e1) -- (d1);%
\draw (Q) -- (c);
\end{tikzpicture}\\
\\
\textrm{$L$ and $D(L)$} & & L_*\simeq D(L)_*
\end{array}
$$

\

\ \ Fig.3
\end{center}

Then $L_*$ is a generalized Esakia space order-homeomorphic to the dual space $D(L)_*$ of $D(L)$. We denote by $P$ the filter $\{\top,g_1,g_2,\ldots\}$ and by $Q$ the filter ${\uparrow}a\cup\displaystyle{\bigcup} \{{\uparrow}f_n: n\in\omega\}$. It is easy to see that $Q$ is the only optimal filter of $L$ which is not prime. Thus, $L_+=L_*-\{Q\}$. It is also easy to calculate that $a\to b=b$ and $a\to c=c$ in $L$, but that $\sigma(a)\to(\sigma(b) \cup \sigma(c))$ does not exist in $D(L)$. Consequently, $D(L)$ is not a Heyting algebra. Put dually, $U=\{{\uparrow}a, {\uparrow}f_1,{\uparrow}f_2,\ldots,Q\}$ is clopen in $L_*$, but ${\downarrow}U=U\cup\{P\}$ is not clopen in $L_*$. Thus, $L_*$ is not an Esakia space. Of course, $U$ is not Esakia clopen because $L_*$ is a generalized Esakia space.
\end{example}

\section{Categorical equivalences}

This section consists of two subsections. In the first one we introduce generalized Priestley morphisms, total generalized Priestey morphisms, and functional generalized Priestley morphisms, the category $\mathsf{GPS}$ of generalized Priestley spaces and generalized Priestley morphisms, the category $\mathsf{GPS}^{\mathsf{T}}$ of generalized Priestley spaces and total generalized Priestley morphisms, and the category $\mathsf{GPS}^{\mathsf{F}}$ of generalized Priestley spaces and functional generalized Priestley morphisms. We also introduce the category $\mathsf{BDM}$ of bounded distributive meet semi-lattices and semi-lattice homomorphisms preserving top, the category $\mathsf{BDM^{\bot}}$ of bounded distributive meet semi-lattices and bounded semi-lattice homomorphisms, and the category  $\mathsf{BDM}^{\mathsf{S}}$ of bounded distributive meet semi-lattices and sup-homomorphisms. We prove that $\mathsf{BDM}$ is dually equivalent to $\mathsf{GPS}$, that $\mathsf{BDM^{\bot}}$ is dually equivalent to $\mathsf{GPS}^{\mathsf{T}}$, and that $\mathsf{BDM}^{\mathsf{S}}$ is dually equivalent to $\mathsf{GPS}^{\mathsf{F}}$. In the second subsection we introduce generalized Esakia morphisms, total generalized Esakia morphisms, and functional generalized Esakia morphisms, the category $\mathsf{GES}$ of generalized Esakia spaces and generalized Esakia morphisms, the category $\mathsf{GES}^{\mathsf{T}}$ of generalized Esakia spaces and total generalized Esakia morphisms, and the category $\mathsf{GES}^{\mathsf{F}}$ of generalized Esakia spaces and functional generalized Esakia morphisms. We also introduce the category $\mathsf{BIM}$ of bounded implicative meet semi-lattices and implicative meet semi-lattice homomorphisms,  the category $\mathsf{BIM^{\bot}}$ of bounded implicative meet semi-lattices and bounded implicative meet semi-lattice homomorphisms, and the category $\mathsf{BIM}^{\mathsf{S}}$ of bounded implicative meet semi-lattices and implicative meet semi-lattice sup-homomorphisms. We prove that $\mathsf{BIM}$ is dually equivalent to $\mathsf{GES}$, that $\mathsf{BIM^{\bot}}$ is dually equivalent to $\mathsf{GES}^{\mathsf{T}}$, and that $\mathsf{BIM}^{\mathsf{S}}$ is dually equivalent to $\mathsf{GES}^{\mathsf{F}}$.

\subsection{The categories $\mathsf{GPS}$, $\mathsf{GPS}^{\mathsf{T}}$, and $\mathsf{GPS}^{\mathsf{F}}$}

Let $X$ and $Y$ be nonempty sets. Given a relation $R \subseteq X
\times Y$, for each $A \subseteq Y$ we define $$\Box_{R}A = \{x \in
X: (\forall y \in Y)(xRy \Rightarrow y \in A)\} = \{x \in X: R[x]
\subseteq A\}.$$ It is easy to verify that for each $A, B \subseteq
Y$ we have $\Box_{R}(A \cap B) = \Box_{R}A \cap\Box_{R}B$ and
$\Box_R(Y)=X$.

Let $L$ and $K$ be bounded distributive meet semi-lattices and let
$h: L \to K$ be a meet semi-lattice homomorphism preserving top. We
define $R_{h} \subseteq K_{*} \times L_{*}$ by
$$\text{$x R_{h} y\ $ iff $\ h^{-1}(x) \subseteq y$}$$ for
each $x \in K_{*}$ and $y \in L_{*}$. We call $R_h$ the dual of $h$.

\begin{proposition}\label{proposition9}
Let $L$ and $K$ be bounded distributive meet semi-lattices and let
$h: L \to K$ be a meet semi-lattice homomorphism preserving top.
Then:
\begin{enumerate}
\item $(\subseteq_{K_{*}} \circ R_{h}) \subseteq R_{h}$.
\item $(R_{h} \circ \subseteq_{L_{*}}) \subseteq R_{h}$.
\item If $x \nR_{h} y$, then there is $a \in L$ such that $y \notin
\phi(a)$ and $R_{h}[x] \subseteq \phi(a)$.
\item $\phi(h(a)) = \Box_{R_{h}} \phi(a)$.
\end{enumerate}
\end{proposition}

\begin{proof}
(1) Suppose that $x, y \in K_{*}$, $z \in L_{*}$, $x \subseteq y$,
and $y R_hz$. Then $h^{-1}(x) \subseteq h^{-1}(y) \subseteq z$.
Thus, $x R_{h}z$.

(2) is proved similarly to (1).

(3) Suppose that $x \nR_{h} y$. Then $h^{-1}(x) \not \subseteq y$,
so there is $a\in L$ such that $a \in h^{-1}(x)$ and $a\notin y$.
Therefore, $y \notin \phi(a)$, and if $x R_{h}z$, then $a \in z$.
Thus, $R_{h}[x] \subseteq \phi(a)$.

(4) If $x \in \phi(h(a))$, then $a \in h^{-1}(x)$. Therefore, for
each $z \in L_{*}$ with $x R_hz$, we have $a \in z$. Thus, $R_{h}[x]
\subseteq \phi(a)$, and so $\phi(h(a)) \subseteq \Box_{R_{h}}
\phi(a)$. Conversely, if $x\in \Box_{R_{h}} \phi(a)$, then $R_{h}[x]
\subseteq \phi(a)$. If $x \not \in \phi(h(a))$, then $a \notin
h^{-1}(x)$. So $h^{-1}(x)\cap{\downarrow}a=\emptyset$, and by the
prime filter lemma, there exists $y\in L_+\subseteq L_*$ such that
$h^{-1}(x) \subseteq y$ and $a \notin y$. But $h^{-1}(x) \subseteq
y$ implies $y \in R_{h}[x]$, so $a \in y$, which is a contradiction.
We conclude that $x \in \phi(h(a))$. Thus, $\Box_{R_{h}} \phi(a)
\subseteq \phi(h(a))$, and so $\phi(h(a)) = \Box_{R_{h}} \phi(a)$.
\end{proof}

\begin{definition}\label{definition}
{\rm Let $X$ and $Y$ be generalized Priestley spaces. A relation $R
\subseteq X \times Y$ is called a \textit{generalized Priestley
morphism} if the following conditions are satisfied:
\begin{enumerate}
\item If $x \nR y$, then there is $U \in Y^{*}$ such that $y \notin U$
and $R[x] \subseteq U$.
\item If $U \in Y^{*}$, then $\Box_{R} U \in X^{*}$.
\end{enumerate}
}
\end{definition}

\begin{lemma}\label{lemma-g-morphisms}
Let $X$ and $Y$ be generalized Priestley spaces and $R \subseteq X \times Y$ be a generalized Priestley morphism. Then:
\begin{enumerate}
\item $({\leq_{X}} \circ R) \subseteq R$.
\item $(R \circ {\leq_{Y}}) \subseteq R$.
\end{enumerate}
\end{lemma}

\begin{proof}
(1) Suppose that $x \leq_{X} y$ and $y Rz$. If $x \nR z$, then by condition (1) of Definition \ref{definition}, there is $U \in Y^*$ such that $R[x] \subseteq U$ and $z \not \in U$. By condition (2) of Definition \ref{definition}, $\Box_{R}U \in X^*$. From $R[x] \subseteq U$ it follows that $x \in  \Box_{R}U$. Because $\Box_{R}U \in X^*$, then $\Box_{R}U$ is an upset. Therefore, $y \in \Box_{R}U$, so $R[y] \subseteq U$, and so $z \in U$, a contradiction. Thus, $x R z$.

(2) Suppose that $x R y$ and $y \leq_{Y} z$. If $x \nR z$, then, by condition (1) of Definition \ref{definition}, there is $U \in Y^*$ such that $R[x] \subseteq U$ and $z \not \in U$. Therefore, $y \not \in U$. On the other hand, $R[x] \subseteq U$ implies $y \in U$. The obtained contradiction proves that $x R z$.
\end{proof}

\begin{remark}
{\rm It is easy to verify that condition (1) of Lemma
\ref{lemma-g-morphisms} is equivalent to saying that for each  $B \subseteq
Y$ we have $R^{-1}[B]$ is a downset of $X$, that condition (2) of
Lemma \ref{lemma-g-morphisms} is equivalent to saying that for each $A
\subseteq X$ we have $R[A]$ is an upset of $Y$, and that conditions
(1) and (2) together are equivalent to $({\leq_{X}} \circ R \circ
{\leq_{Y}}) \subseteq R$.}
\end{remark}

Given a generalized Priestley morphism $R\subseteq X\times Y$, we
define a map $h_{R}:Y^{*}\to X^{*}$ by $$h_{R}(U) = \Box_{R}U$$ for
each $U\in Y^*$.

\begin{lemma}\label{lemma5}
If $R\subseteq X\times Y$ is a generalized Priestley morphism, then
$h_{R}:Y^*\to X^*$ is a meet semi-lattice homomorphism preserving
top.
\end{lemma}

\begin{proof}
Let $U, V \in Y^{*}$. Then $h_{R}(U \cap V) = \Box_{R}(U \cap V) =
\Box_{R} U \cap \Box_{R} V=h_{R}(U)\cap h_{R}(V)$. Moreover,
$h_{R}(Y)=\Box_{R}Y = X$.
\end{proof}

\begin{proposition}\label{h-Rh}
Let $L$ and $K$ be bounded distributive meet semi-lattices and let
$h:L\to K$ be a meet semi-lattice homomorphism preserving top. Then
for each $a \in L$ we have $\phi(h(a)) = h_{R_h}(\phi(a))$.
\end{proposition}

\begin{proof}
We have $x \in \phi(h(a))$ iff $h(a) \in x$ iff $a \in h^{-1}(x)$,
and $x \in h_{R_h}\phi(a)$ iff $x \in \Box_{R_h}\phi(a)$ iff
$(\forall y \in L_{*})(xR_hy \Rightarrow a \in y)$ iff $(\forall y
\in L_{*})(h^{-1}(x) \subseteq y \Rightarrow a \in y)$. Now either
$h^{-1}(x)=L$ or $h^{-1}(x)$ is a proper filter of $L$. If
$h^{-1}(x)=L$, then for all $y \in L_{*}$ we have $h^{-1}(x)
\not\subseteq y$. Therefore, both $a \in h^{-1}(x)$ and $(\forall
y \in L_{*})(h^{-1}(x) \subseteq y \Rightarrow a \in y)$ are
trivially true, and so $a \in h^{-1}(x)$ iff $(\forall y \in
L_{*})(h^{-1}(x) \subseteq y \Rightarrow a \in y)$. On the other
hand, if $h^{-1}(x)$ is a proper filter of $L$, then by the
optimal filter lemma, $h^{-1}(x)$ is the intersection of all the
optimal filters of $L$ containing $h^{-1}(x)$. Hence, $a\in
h^{-1}(x)$ iff $(\forall y \in L_{*})(h^{-1}(x) \subseteq y
\Rightarrow a \in y)$. Thus, in either case we have $\phi(h(a)) =
h_{R_h}(\phi(a))$.
\end{proof}

\begin{proposition}\label{R-hR}
Let $R\subseteq X\times Y$ be a generalized Priestley morphism. Then
for each $x \in X$ and $y \in Y$ we have $x R y$ iff $\psi(x)
R_{h_R} \psi(y)$.
\end{proposition}

\begin{proof}
Let $x R y$. We show $\psi(x) R_{h_R} \psi(y)$. If $U \in
h_{R}^{-1}(\psi(x))$, then $h_R(U)\in \psi(x)$, so $x \in h_{R}(U)$,
and so $R[x] \subseteq U$. Therefore, $y \in U$, so $U \in \psi(y)$,
and so $h_{R}^{-1}(\psi(x))\subseteq \psi(y)$. Thus, $\psi(x)
R_{h_R} \psi(y)$. Now let $x \nR y$. Then, by condition (1) of
Definition \ref{definition}, there is $U \in Y^{*}$ such that $y
\notin U$ and $R[x] \subseteq U$. Therefore, $y \notin U$ and $x \in
h_{R}(U)$. Thus, we have $U \notin \psi(y)$ and $h_{R}(U) \in
\psi(x)$. It follows that $h_{R}^{-1}(\psi(x)) \not \subseteq
\psi(y)$. Consequently, $x R y$ iff $\psi(x) R_{h_R} \psi(y)$.
\end{proof}

Unfortunately, the usual set-theoretic composition of two generalized Priestley morphisms may not be a generalized Priestley morphism. Therefore, we introduce the composition of two generalized Priestley morphisms as follows. Let $X,Y,$ and $Z$ be generalized Priestley spaces, and $R\subseteq X\times Y$ and $S\subseteq Y\times Z$ be generalized Priestley morphisms. Define $S {\ast} R \subseteq X \times Z$ by
$$x (S {\ast} R) z \;\; \text{ iff } \;\; \psi(x) R_{(h_R \circ h_S)} \psi(z).$$
Note that
$$x (S {\ast} R) z \;\; \text{ iff } \;\; (\forall U \in Z^{\ast}) (x \in \Box_{R}\Box_{S}U \Rightarrow z \in U).$$

\begin{lemma} \label{catGPS-1}
If $X, Y,$ and $Z$ are generalized Priestley spaces, and $R\subseteq X\times Y$ and $S\subseteq Y\times Z$ are generalized Priestley morphisms, then for each $U \in Z^{\ast},$ we have $$\Box_R \Box_S U = (h_R \circ h_S)(U) = \Box_{(S{\ast}R)}U.$$
\end{lemma}

\begin{proof}
Let $U \in Z^{\ast}$. Clearly $\Box_R \Box_S U = (h_R \circ h_S)(U)$. On the other hand, for $x \in U$, we have:

\begin{center}
\begin{tabular}{lll}
$x \in (h_R \circ h_S)(U)$ & iff & $(h_R \circ h_S)(U) \in \psi(x)$\\
& iff & $U \in (h_R \circ h_S)^{-1}[\psi(x)]$\\
& iff & $(\forall z \in Z)(\psi(x) R_{(h_R \circ h_S)} \psi(z) \Rightarrow z \in U)$\\
& iff & $(\forall z \in Z)(x (S{\ast}R) z \Rightarrow z \in U)$\\
& iff & $(S{\ast}R)[x] \subseteq U$\\
& iff & $x \in \Box_{(S{\ast}R)}U$.
\end{tabular}
\end{center}

\noindent Therefore, $(h_R \circ h_S)(U) = \Box_{(S{\ast}R)}U$.
\end{proof}

\begin{lemma} \label{catGPS-2}
If $X, Y,$ and $Z$ are generalized Priestley spaces, and $R\subseteq X\times Y$ and $S\subseteq Y\times Z$ are generalized Priestley morphisms, then $S {\ast} R\subseteq X\times Z$ is a generalized Priestley morphism.
\end{lemma}

\begin{proof}
To see that condition (1) of Definition \ref{definition} is satisfied, let $(x,y)\notin S {\ast} R$. Then there is $U \in Z^{*}$ such that $x \in  \Box_{R}\Box_{S}U$ and $z \not \in U$. By Lemma \ref{catGPS-1}, this means that $(S {\ast} R)[x] \subseteq U$ and $z\notin U$. To see that condition (2) of Definition \ref{definition} is also satisfied, let $U \in Z^{\ast}$. Then $(h_{R}\circ h_{S})(U) \in X^{\ast}$, which, by Lemma \ref{catGPS-1}, means that $\Box_{(S{\ast}R)}U \in X^{\ast}$. Thus, $S {\ast} R$ is a generalized Priestley morphism.
\end{proof}

\begin{lemma}\label{catGPS-3}
Let  $X, Y, Z,$ and $W$ be generalized Priestley spaces, and $R\subseteq X\times Y$, $S\subseteq Y\times Z$, and $T\subseteq Z\times W$ be generalized  Priestley morphisms. Then $$T {\ast} (S {\ast} R) = (T {\ast} S) {\ast} R.$$
\end{lemma}

\begin{proof}
Let $x \in X$ and $w \in W$. Then, by Lemma \ref{catGPS-1}, we have:

\begin{center}
\begin{tabular}{lll}
$x [T {\ast} (S {\ast} R)] w$ & iff & $(\forall U \in W^*) (x \in \Box_{(S {\ast} R)}\Box_{T}U \Rightarrow w \in U)$\\
& iff & $(\forall U \in W^*) (x \in \Box_{R} \Box_{S}\Box_{T} U \Rightarrow w \in U)$\\
& iff & $(\forall U \in W^*) (x \in \Box_{R} \Box_{(T {\ast} S)}U \Rightarrow w \in U)$\\
& iff & $x [(T {\ast} S) {\ast} R] w$.
\end{tabular}
\end{center}

\end{proof}

\begin{lemma}\label{catGPS-4}
Let $X$ be a generalized Priestely space. Then:
\begin{enumerate}
\item $\leq_{X}\subseteq X\times X$ is a generalized Priestely  morphism.
\item If $R\subseteq X\times Y$ is a generalized Priestely morphism, then $\leq_X {\circ}\ R = R$.
\item If $S\subseteq Y\times X$ is a generalized Priestely morphism, then $S {\circ} \leq_X = S$.
\end{enumerate}
\end{lemma}

\begin{proof}
(1) follows from the definition of a generalized Priestely space, while (2) and (3) follow from Lemma \ref{lemma-g-morphisms} and the reflexivity of $\leq_X$.
\end{proof}

As an immediate consequence of Lemmas \ref{catGPS-1}, \ref{catGPS-2}, \ref{catGPS-3}, and \ref{catGPS-4}, we obtain that generalized Priestley spaces and generalized Priestley morphisms form a category, in which $\ast$ is the composition of two morphisms and $\leq_X$ is the identity morphism for each object $X$. We denote this category by $\mathsf{GPS}$. Let also $\mathsf{BDM}$ denote the category of bounded distributive meet semi-lattices and semi-lattice homomorphisms preserving top.

We show that $\mathsf{BDM}$ is dually equivalent to $\mathsf{GPS}$. Define two functors $(-)_{*}:\mathsf{BDM}\to \mathsf{GPS}$ and $(-)^{*}:\mathsf{GPS}\to \mathsf{BDM}$ as follows. For a bounded distributive meet semi-lattice $L$, set $L_{*} = \la L_{*}, \tau, \subseteq,
L_+\ra$, and for a meet semi-lattice homomorphism $h$ preserving top, set $h_{*} = R_h$; for a generalized Priestley space $X$, let $X^{*}$ be the bounded distributive meet semi-lattice of admissible clopen upsets of $X$, and for a generalized Priestley morphism $R$, let $R^*=h_{R}$.

In order to prove that the functors $(-)_{*}$ and $(-)^{*}$ establish the dual equivalence of $\mathsf{BDM}$ and $\mathsf{GPS}$, we define the natural transformations from the identity functor $\mathrm{id}_{\mathsf{BDM}}:\mathsf{BDM}\to\mathsf{BDM}$ to the functor ${(-)_{*}}^{*}:\mathsf{BDM}\to \mathsf{BDM}$ and from the identity functor $\mathrm{id}_{\mathsf{GPS}}:\mathsf{GPS}\to\mathsf{GPS}$ to the functor ${(-)^{*}}_{*}:\mathsf{GPS}\to \mathsf{GPS}$.

The first natural transformation associates with each object $L$ of $\mathsf{BDM}$ the isomorphism $\phi_L:L \to {L_{*}}^{*}$; and the second natural transformation associates with each object $X$ of $\mathsf{GPS}$ the generalized Priestley morphism $R_{\varepsilon_X} \subseteq X \times {X^{*}}_{*}$ given by $$x R_{\varepsilon_X} \varepsilon(y) \ \ \text{ iff  } \ \ \varepsilon_X(x) \subseteq  \varepsilon_X(y)$$ for each $x, y \in X$.

\begin{theorem}\label{theorem-duality}
The functors $(-)_{*}$ and $(-)^{*}$ establish the dual equivalence of $\mathsf{BDM}$ and $\mathsf{GPS}$.
\end{theorem}

\begin{proof}
It follows from Propositions \ref{proposition6} and \ref{proposition9} that $(-)_{*}$ is well-defined. Proposition \ref{proposition7} and Lemma \ref{lemma5} imply that $(-)^{*}$ is well-defined. Now apply Theorems \ref{proposition8} and \ref{homeom} and Propositions \ref{h-Rh} and \ref{R-hR}.
\end{proof}

Now we turn our attention to meet semi-lattice homomorphisms preserving bottom and to sup-homomorphisms.

\begin{lemma} \label{tot-fun}
Let $L$ and $K$ be bounded distributive meet semi-lattices and let $h: L \to K$ be a meet semi-lattice homomorphism preserving top. Then:
\begin{enumerate}
\item $h$ preserves bottom iff ${R_h}^{-1}[L_*] = K_*$.
\item $h$ is a sup-homomorphism iff $R_{h}[x]$ has a least element for each $x \in K_{*}$.
\end{enumerate}
\end{lemma}

\begin{proof}
(1) By Proposition \ref{h-Rh}, we have:

\begin{center}
\begin{tabular}{lll}
$h$ preserves bottom & iff & $h(\bot)=\bot$\\
& iff & $\phi(h(\bot))=\phi(\bot)$\\
& iff & $h_{R_h}(\phi(\bot))=\phi(\bot)$\\
& iff & $h_{R_h}(\emptyset)=\emptyset$\\
& iff & $R_h^{-1}[L_*]=K_*$.
\end{tabular}
\end{center}

(2) For each $x \in K_{*}$, we show that $R_{h}[x]$ has a least element iff $h^{-1}[x] \in L_{*}$. If $h^{-1}[x] \in L_{*}$, then it is clear that $h^{-1}[x]$ is the least element of $R_{h}[x]$. Conversely, let $y$ be the least element of $R_{h}[x]$. By the optimal filter lemma, $h^{-1}[x] = \bigcap\{z \in L_{*}: h^{-1}[x] \subseteq z\} = \bigcap R_{h}[x]=y$. Therefore, $h^{-1}[x] \in L_{*}$. By Proposition \ref{prop-sup}, $h$ is a sup-homomorphism iff $h^{-1}[x]\in L_*$ for each $x\in K_*$. Thus, $h$ is a sup-homomorphism iff $R_{h}[x]$ has a least element for each $x \in K_{*}$.
\end{proof}

\begin{definition}
{\rm Let $X$ and $Y$ be generalized Priestley spaces and let $R\subseteq X\times Y$ be a generalized Priestley morphism.
\begin{enumerate}
\item We call $R$ \emph{total} if $R^{-1}[Y] = X$.
\item We call $R$ \emph{functional} if for each $x \in X$ there is $y \in Y$ such that $R[x] = {\uparrow}y$.
\end{enumerate}
}
\end{definition}

Obviously $R$ is functional iff $R[x]$ has a least element. It is also clear that each functional generalized Priestley morphism is total. As an immediate consequence of Theorem \ref{theorem-duality} and Lemma \ref{tot-fun}, we obtain:

\begin{corollary} \label{cor-bot-sup}
Let $X$ and $Y$ be generalized Priestley spaces and $R\subseteq X\times Y$ be a generalized Priestley morphism. Then:
\begin{enumerate}
\item $h_R$ preserves bottom iff $R$ is total.
\item $h_R$ is a sup-homomorphism iff $R$ is functional.
\end{enumerate}
\end{corollary}


In particular, it follows that each sup-homomorphism preserves bottom.

\begin{lemma} \label{tot-fun-comp}
Let $X,Y,$ and $Z$ be generalized Priestley spaces, and $R\subseteq X\times Y$ and $S\subseteq Y\times Z$ be generalized Priestley morphisms.
\begin{enumerate}
\item $S {\ast} R$ is total whenever $R$ and $S$ are total.
\item $S {\ast} R$ is functional whenever $R$ and $S$ are functional.
\end{enumerate}
\end{lemma}

\begin{proof}
(1) Let $R$ and $S$ be total generalized Priestley morphisms and let $z \in Z$. Since $S$ and $R$ are total, there exist $y\in Y$ and $x \in X$ such that $y S z$ and $x R y$. We show that $x (S {\ast} R) z$. Let $U \in Z^{*}$ be such that $x \in \Box_{R}\Box_{S}U$. Then $y \in \Box_S U$, and so $z \in U$. Therefore, $x (S {\ast} R) z$, which implies that $S {\ast} R$ is total.

(2) Let $R$ and $S$ be functional generalized Priestley morphisms and let $x \in X$. Then there exist $y \in Y$ and $z\in Z$ such that $R[x] = {\uparrow}y$ and $S[y] = {\uparrow}z$. We show that $(S \ast R)[x] = {\uparrow}z$. Let $U \in Z^{\ast}$ and let $x \in \Box_{R}\Box_{S}U$. From $xRy$ it follows that $y \in \Box_{S}U$; and $y S z$ implies $z \in U$. Therefore, $x (S \ast R) z$. Thus, ${\uparrow}z\subseteq (S*R)[x]$. Conversely, let $x (S \ast R) u$. If $z \not \leq u$, then there exists $V \in Z^{\ast}$ such that $z \in V$ and $u \notin V$. Therefore, $x \notin \Box_{R}\Box_{S}V$. Thus, there exist $y'\in Y$ and $z'\in Z$ such that $xRy'$, $y'Sz'$, and $z' \notin V$. Since $R[x]={\uparrow}y$, we obtain $y \leq y'$, and so $y S z'$. From $S[y]={\uparrow}z$ it follows that $z \leq z'$. Therefore, $z' \in V$, a contradiction. We conclude that $u\in{\uparrow}z$, so $(S \ast R)[x] \subseteq {\uparrow} z$, and so $(S \ast R)[x] = {\uparrow} z$. Thus, $S \ast R$ is functional.
\end{proof}

\begin{remark}
Let $R$ and $S$ be functional generalized Priestley morphisms. Then it is easy to verify that $S\circ R$ is again a functional generalized Priestley morphism. Therefore, $S*R=S\circ R$.
\end{remark}

Let $\mathsf{GPS}^{\mathsf{T}}$ denote the category of generalized Priestley spaces and total generalized Priestley morphisms. It follows from Lemma \ref{tot-fun-comp}.1 that $\mathsf{GPS}^{\mathsf{T}}$ forms a category, which is obviously a proper subcategory of $\mathsf{GPS}$. Let also $\mathsf{GPS}^{\mathsf{F}}$ denote the category of generalized Priestley spaces and functional generalized Priestley morphisms. By Lemma \ref{tot-fun-comp}.2, $\mathsf{GPS}^{\mathsf{F}}$ forms a category, which is clearly a proper subcategory of $\mathsf{GPS}^{\mathsf{T}}$.

We let $\mathsf{BDM^{\bot}}$ denote the category of bounded distributive meet semi-lattices and bounded semi-lattice homomorphisms, and $\mathsf{BDM}^{\mathsf{S}}$ denote the category of bounded distributive meet semi-lattices and sup-homomorphisms. Similarly, we have that $\mathsf{BDM}^{\mathsf{S}}$ is a proper subcategory of $\mathsf{BDM^{\bot}}$, and that $\mathsf{BDM^{\bot}}$ is a proper subcategory of $\mathsf{BDM}$.
\begin{theorem}\label{theorem-duality-sup-hom}
\begin{enumerate}
\item[]
\item $\mathsf{BDM}^{\bot}$ is dually equivalent to $\mathsf{GPS}^{\mathsf{T}}$.
\item $\mathsf{BDM}^{\mathsf{S}}$ is dually equivalent to $\mathsf{GPS}^{\mathsf{F}}$.
\end{enumerate}
\end{theorem}

\begin{proof}
(1) Let $(-)_{*}:\mathsf{BDM}^{\bot}\to\mathsf{GPS}^{\mathsf{T}}$ and $(-)^{*}:\mathsf{GPS}^{\mathsf{T}}\to\mathsf{BDM}^{\bot}$ denote the restrictions of $(-)_*$ and $(-)^*$ to $\mathsf{BDM}^{\bot}$ and $\mathsf{GPS}^{\mathsf{T}}$, respectively. By Lemma \ref{tot-fun}.1 and Corollary \ref{cor-bot-sup}.1, whenever $h:L\to K$ preserves bottom, then $R_h$ is total, and conversely, whenever $R\subseteq X\times Y$ is total, then $h_R$ preserves bottom. Now apply Theorem \ref{theorem-duality}.

(2) Let $(-)_{*}:\mathsf{BDM}^{\mathsf{S}}\to\mathsf{GPS}^{\mathsf{F}}$ and $(-)^{*}:\mathsf{GPS}^{\mathsf{F}}\to\mathsf{BDM}^{\mathsf{S}}$ denote the restrictions of $(-)_*$ and $(-)^*$ to $\mathsf{BDM}^{\mathsf{S}}$ and $\mathsf{GPS}^{\mathsf{F}}$, respectively. By Lemma \ref{tot-fun}.2 and Corollary \ref{cor-bot-sup}.2, whenever $h:L\to K$ is a sup-homomorphism, then $R_h$ is functional, and conversely, whenever $R\subseteq X\times Y$ is functional, then $h_R$ is a sup-homomorphism. Now apply Theorem \ref{theorem-duality}.
\end{proof}

\subsection{The categories $\mathsf{GES}$, $\mathsf{GES}^{\mathsf{T}}$, and $\mathsf{GES}^{\mathsf{F}}$}

\begin{proposition}\label{prop-gem}
Let $L$ and $K$ be bounded implicative meet semi-lattices and let
$h: L \to K$ be an implicative meet semi-lattice homomorphism. Then
for each $x \in K_{*}$ and $y \in L_{+}$ we have $x R_{h}y$ implies
there exists $z \in K_{+}$ such that $x \subseteq z$ and $R_{h}[z] =
{\uparrow}y$.
\end{proposition}

\begin{proof}
Suppose that $x \in K_{*}$, $y \in L_{+}$, and $xR_{h}y$. Then
$h^{-1}(x) \subseteq y$. We show that ${\downarrow}_{K}h(L - y)$ is
an ideal of $K$. Let $a, b \in {\downarrow}_{K}h(L - y)$. Then there
exist $c,d \in L - y$ such that $a \leq h(c)$ and $b \leq h(d)$.
Since $y$ is a prime filter of $L$, we have $L - y$ is a (prime)
ideal of $L$, so ${\uparrow}_{L}c \cap {\uparrow}_{L}d \cap (L - y)
\neq \emptyset$. Therefore, there exists $e \in {\uparrow}_{L}c \cap
{\uparrow}_ {L}d$ such that $e \notin y$. Thus, $a \leq h(e)$, $b
\leq h(e)$, and $h(e) \in {\downarrow}_{S} h(L - y)$. It follows
that ${\uparrow}_{K}a \cap {\uparrow}_{K}b \cap {\downarrow}_{K} h(L
- y) \neq \emptyset$, and so ${\downarrow}_{K} h(L - y)$ is an ideal
of $K$. Let $F$ be the filter of $K$ generated by $x \cup h(y)$. If
$a \in F \cap {\downarrow}_{K} h(L - y)$, then there exist $b \in
x$, $c \in y$, and $d \in L - y$ such that $b \wedge h(c) \leq a
\leq h(d)$. Then $b \leq h(c \to d)$. Therefore, $h(c \to d) \in x$,
and so $c \to d \in y$. Now $c \in y$ and $c \to d \in y$ imply $d
\in y$, a contradiction. Thus, $F \cap {\downarrow}_{K} h(L -
y)=\emptyset$, and so, by the prime filter lemma, there exists a
prime filter $z$ of $K$ such that $F \subseteq z$ and $z \cap
{\downarrow}_{K} h(L - y) = \emptyset$. Then $x \subseteq z$ and
$h^{-1}(z) = y$. Thus, there exists $z \in K_{+}$ such that $x
\subseteq z$ and $R_{h}[z] = {\uparrow}y$.
\end{proof}

Let $L$ and $K$ be bounded implicative meet semi-lattices and let
$h: L \to K$ be an implicative meet semi-lattice homomorphism.
Consider the following condition, which is a strengthening of the
condition of Proposition \ref{prop-gem}: For each $x \in K_{*}$ and
$y \in L_{*}$ we have $x R_{h}y$ implies there exists $z \in K_{*}$
such that $x \subseteq z$ and $R_{h}[z] = {\uparrow}y$. The next
example shows that this stronger condition does not necessarily
hold.

\begin{example}
Let $X$ and $Y$ be generalized Esakia spaces shown in Fig.4, where
$Y_0=Y-\{z\}$, and as a topological space, $Y$ is the one-point
compactification of $Y_0$ (as a discrete space). Then each point of
$Y_0$ is an isolated point of $Y$, and $z$ is the only limit point
of $Y$. Let $R\subseteq X\times Y$ be defined as follows:
$R[x_1]={\uparrow}y_1$, $R[x_2]=\{y_2\}$, and $R[x_3]=\{y_3\}$. It
is easy to verify that $R$ is a generalized Esakia morphism, that
$x_1Rz$, but that there is no $x\in X$ such that $R[x]={\uparrow}z$.
Thus, the condition above is not satisfied. The dual implicative
meet semi-lattice $X^*\simeq D(X^*)$ of $X$, which is a Heyting
algebra, together with the dual implicative meet semi-lattice $Y^*$
and its distributive envelope $D(Y^*)$ are shown in Fig.4. The
elements of $Y^*$ are depicted in the black circles, while the only
element of $D(Y^*)-Y^*$ is depicted in the white circle. The
implicative meet semi-lattice homomorphism $h_R:Y^*\to X^*$ is shown
in Fig.4 by means of brown arrows. Note that $h_R$ is undefined on
$\{y_2,y_3\}$, thus $h_R$ is not a homomorphism from $D(Y^*)$ to
$D(X^*)\simeq X^*$.
\end{example}


\

\begin{center}
$$
\begin{array}{ccc}
\begin{tikzpicture}[scale=.5,inner sep=.5mm]
\node[bull] (u1) at (0,0) [label=right:$u_1$] {};%
\node[bull] (u2) at (0,1) [label=right:$u_2$] {};%
\node[bull] (u3) at (0,2) [label=right:$u_3$] {};%
\node (empdown) at (0,3) {};
\node (dots) at (0,4) {$\vdots$};%
\node[bull] (y1) at (-1,5) [label=below:$y_1$] {};%
\node[holl] (z) at (0,6) [label=left:$z$] {};
\node[bull] (y2) at (-1,7) [label=left:$y_2$] {};%
\node[bull] (y3) at (1,7) [label=right:$y_3$] {};%
\node[bull] (x1) at (-5,5) [label=below:$x_1$] {};%
\node[bull] (x2) at (-6,7) [label=left:$x_2$] {};%
\node[bull] (x3) at (-4,7) [label=right:$x_3$] {};%
\draw (u1) -- (u2) -- (u3) -- (empdown);%
\draw (y1) -- (z) -- (y3);%
\draw (z) -- (y2);%
\draw (x1) -- (x2);%
\draw (x1) -- (x3);%
\draw [brown,->] (x1.north east).. controls +(1,1) and +(-1,1) ..
(y1.north west);%
\draw [brown,->] (x2.north east).. controls +(1,1) and
+(-1,1) .. (y2.north west);%
\draw [brown,->] (x3.north east).. controls +(1,1) and +(-1,1) ..
(y3.north west);
\end{tikzpicture}&\ &
\begin{tikzpicture}[scale=.5,inner sep=.5mm]
\node[bull] (emp) at (0,0) [label=below:$\emptyset$] {};%
\node[bull] (y2) at (-1,1) [label=left:$\{y_2\}$] {};%
\node[bull] (y3) at (1,1) [label=right:$\{y_3\}$] {};%
\node[holl] (join) at (0,2) [label=above:$\qquad\{y_2{,}y_3\}$] {};%
\node (lvdots) at (-1,4) {$\vdots$};%
\node (rvdots) at (1,5) {$\vdots$};%
\node (empl) at (-1,4.5) {};%
\node (empr) at (1,5.5) {};
\node (l) at (-.9,3.2) {};%
\node[bull] (u3) at (-1,5) [label=left:${\uparrow}u_3$] {};%
\node[bull] (u2) at (-1,6) [label=left:${\uparrow}u_2$] {};%
\node[bull] (u1) at (-1,7) [label=left:${\uparrow}u_1$] {};%
\node (r) at (1.1,4.3) {};%
\node[bull] (yu3) at (1,6) [label=right:$\{y_1\}\cup{\uparrow}u_3$] {};%
\node[bull] (yu2) at (1,7) [label=right:$\{y_1\}\cup{\uparrow}u_2$] {};%
\node[bull] (yu1) at (1,8) [label=right:$Y$] {};%
\node[bull] (empx) at (7,2) [label=below:$\emptyset$] {};%
\node[bull] (x2) at (6,3) [label=left:$\{x_2\}$] {};%
\node[bull] (x3) at (8,3) [label=right:$\{x_3\}$] {};%
\node[bull] (joinr) at (7,4) [label=right:$\ \{x_2{,}x_3\}$] {};%
\node[bull] (X) at (7,5) [label=above:$X$] {};%
\draw[brown,rounded corners=.3ex] (-1.2,3) rectangle (-.8,7.4);
\draw[brown,rounded corners=.3ex] (.8,4) rectangle (1.2,8.4);
\draw (y2) -- (emp) -- (y3);%
\draw (y2) -- (join) -- (y3);%
\draw (empl) -- (u3) -- (u2) -- (u1);%
\draw (empr) -- (yu3) -- (yu2) -- (yu1);%
\draw (u3) -- (yu3);%
\draw (u2) -- (yu2);%
\draw (u1) -- (yu1);%
\draw (X) -- (joinr) -- (x2) -- (empx) -- (x3) -- (joinr);%
\draw [brown,->] (emp) -- (empx);%
\draw [brown,->] (y2) -- (x2);%
\draw [brown,->] (y3) -- (x3);%
\draw [brown,->] (l) -- (joinr);%
\draw [brown,->] (r) -- (X);%
\end{tikzpicture}\\
X\ \ \ \ \ \ \ \ \ \ \ \ \ \ \ \ Y& &\textrm{$Y^*$ and $D(Y^*)$}\ \
\ \ \ X^*\simeq D(X^*)
\end{array}
$$

\

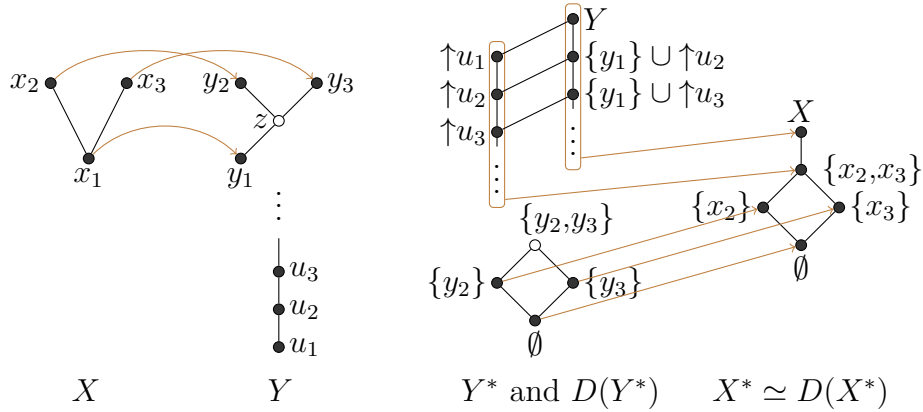
Fig.4\ \ \ \ \ \ \ \ \ \ \ \ \ \ {}

\end{center}

\begin{definition}\label{def-gem}
{\rm Let $X$ and $Y$ be generalized Esakia spaces. We call a
generalized Priestley morphism $R\subseteq X\times Y$ a
\textit{generalized Esakia morphism} if for each $x \in X$ and $y
\in Y_0$, from $x R y$ it follows that there exists $z \in X_0$ such
that $x \leq z$ and $R[z] = {\uparrow}y$. }
\end{definition}

\begin{proposition}\label{prop-gem1}
Let $X$ and $Y$ be generalized Esakia spaces and let $R\subseteq
X\times Y$ be a generalized Esakia morphism. Then $h_{R}:Y^*\to X^*$
is an implicative meet semi-lattice homomorphism.
\end{proposition}

\begin{proof}
Since $R\subseteq X\times Y$ is a generalized Esakia morphism, it is
a generalized Priestley morphism, so $h_R:Y^*\to X^*$ is a meet
semi-lattice homomorphism preserving top. We show that for each $U,
V \in Y^{*}$, we have $h_{R}(U \to V) = h_{R}(U) \to h_{R}(V)$.
Suppose that $x \in h_{R}(U \to V)$. Then $R[x] \subseteq U \to V$.
If $x \notin h_{R}(U) \to h_{R}(V)$, then there exists $y\in
{\uparrow}x\cap h_R(U)$ such that $y \notin h_{R}(V)$. Therefore,
$R[y] \subseteq U$ and $R[y] \not \subseteq V$. Thus, there exists
$z\in Y$ such that $y Rz$ and $z \notin V$. Since $x \leq y$, we
have $R[y] \subseteq R[x]$. Therefore, $R[y] \subseteq U \to V$, and
so $z \in U \to V$. This together with $z \in U$ imply $z \in V$, a
contradiction. Thus, $x \in h_{R}(U) \to h_{R}(V)$, and so $h_{R}(U
\to V)\subseteq h_{R}(U) \to h_{R}(V)$. Conversely, suppose that $x
\in h_{R}(U) \to h_{R}(V)$ and $x \notin h_{R}(U \to V)$. Then
${\uparrow}x\cap h_R(U)\subseteq h_R(V)$ and $R[x]\not\subseteq U\to
V$. Therefore, there exists $y\in Y$ such that $xRy$ and $y\notin
U\to V = [{\downarrow}(U - V)]^c$. Thus, $y\in {\downarrow}(U - V)$.
Since $Y$ is a generalized Esakia space, ${\downarrow}(U - V)$ is
clopen in $Y$ and $\mathrm{max}[{\downarrow}(U - V)]=\mathrm{max}(U
- V)\subseteq Y_0$. Therefore, there exists $z\in \mathrm{max}(U -
V)\subseteq Y_0$ such that $y\leq z$. Then $xRy\leq z\in Y_0$, so
$xRz\in Y_0$, and since $R$ is a generalized Esakia morphism, there
exists $u \in X_0$ such that $x \leq u$ and $R[u] = {\uparrow}z$. As
$z\in U$, we have $R[u]\subseteq U$, so $u\in h_R(U)$. Therefore, $u
\in {\uparrow}x\cap h_R(U)\subseteq h_R(V)$, so $R[u]\subseteq V$.
Thus, $z\in{\uparrow}z=R[u]\subseteq V$, a contradiction.
Consequently, our assumption that $x \notin h_{R}(U \to V)$ is
false, so $h_{R}(U) \to h_{R}(V)\subseteq h_{R}(U \to V)$, and so
$h_{R}(U \to V)=h_{R}(U) \to h_{R}(V)$.
\end{proof}

\begin{lemma}
Let $X$,  $Y$, and $Z$ be generalized Esakia spaces, and $R\subseteq X\times Y$ and $S \subseteq Y \times Z$ be generalized Esakia morphisms. Then $S \ast R$ is a generalized Esakia morphism.
\end{lemma}

\begin{proof}
By Proposition \ref{prop-gem1}, $h_{R}: Y^{*} \to X^{*}$ and $h_{S}: Z^{*} \to Y^{*}$ are implicative meet semi-lattice homomorphisms. Therefore, $h_{R} \circ h_{S}: Z^{*} \to X^{\ast}$ is an implicative meet semi-lattice homomorphism. By Lemma \ref{catGPS-1}, $h_{(S \ast R)} = h_{R} \circ h_{S}$. Therefore, $h_{(S \ast R)}$ is an implicative  meet semi-lattice homomorphism. By Proposition \ref{prop-gem}, $R_{h_{(S \ast R)}}$ is a generalized Esakia morphism. This, by Theorem \ref{theorem-duality}, implies that $S \ast R$ is a generalized Esakia morphism.
\end{proof}

Let $\mathsf{GES}$ denote the category of generalized Esakia spaces and generalized Esakia morphisms, in which $*$ is the composition of two morphisms and $\leq_X$ is the identity morphism for each object $X$. Let also $\mathsf{GES}^{\mathsf{T}}$ denote the subcategory of $\mathsf{GES}$ whose objects are generalized Esakia spaces and whose morphisms are total generalized Esakia morphisms, and $\mathsf{GES}^{\mathsf{F}}$ denote the subcategory of $\mathsf{GES}^{\mathsf{T}}$ whose objects are generalized Esakia spaces and whose morphisms are functional generalized Esakia morphisms. Clearly $\mathsf{GES}^{\mathsf{F}}$ is a proper subcategory of $\mathsf{GES}^{\mathsf{T}}$ and $\mathsf{GES}^{\mathsf{T}}$ is a proper subcategory of $\mathsf{GES}$.

We let $\mathsf{BIM}$ denote the category of bounded implicative meet semi-lattices and implicative meet semi-lattice homomorphisms, $\mathsf{BIM}^{\bot}$ denote the category of bounded implicative meet semi-lattices and bounded implicative meet semi-lattice homomorphisms, and $\mathsf{BIM}^{\mathsf{S}}$ denote the category of bounded implicative meet semi-lattices and implicative meet semi-lattice sup-homomorphisms. We have that $\mathsf{BIM}^{\mathsf{S}}$ is a proper subcategory of $\mathsf{BIM}^{\bot}$ and that $\mathsf{BIM}^{\bot}$ is a proper subcategory of $\mathsf{BIM}$.

\begin{theorem}\label{theorem-ges}
\begin{enumerate}
\item[]
\item The category $\mathsf{BIM}$ is dually equivalent to the category $\mathsf{GES}$.
\item The category $\mathsf{BIM}^{\bot}$ is dually equivalent to the category $\mathsf{GES}^{\mathsf{T}}$.
\item The category $\mathsf{BIM}^{\mathsf{S}}$ is dually equivalent to the category $\mathsf{GES}^{\mathsf{F}}$.
\end{enumerate}
\end{theorem}

\begin{proof}
Apply Theorems \ref{theorem-duality} and \ref{theorem-duality-sup-hom} and Propositions \ref{lemma-ims-ges}, \ref{prop-gem}, and \ref{prop-gem1}.
\end{proof} 
\section{Functional morphisms}

In this section we show that functional generalized Priestley (resp.\ Esakia) morphisms can be characterized by means of special functions between generalized Priestley (resp.\ Esakia) spaces we call strong Priestley (resp.\ Esakia) morphisms.

\subsection{Functional generalized Priestley morphisms}


Let $X$ and $Y$ be Priestley spaces. We recall that a map $f:X\to Y$ is a Priestley morphism if $f$ is continuous and order-preserving.

\begin{definition}
{\rm Let $X$ and $Y$ be generalized Priestley spaces. We call a map $f: X \to Y$ a \textit{strong Priestley morphism} if $f$ is order-preserving and $U \in Y^{*}$ implies $f^{-1}(U) \in X^{*}$.
}
\end{definition}

Since $X^*\cup\{U^c:U\in X^*\}$ and $Y^*\cup\{V^c:V\in Y^*\}$ form subbases for the Priestley topologies on $X$ and $Y$, respectively, and $f^{-1}(V^{c}) = f^{-1}(V)^{c}$ for each $V \subseteq Y$, it follows that each strong Priestley morphism is a continuous function, hence a Priestley morphism. We note that the composition of strong Priestley morphisms is again a strong Priestley morphism, and that the identity map $\mathrm{id}_X:X\to X$ is a strong Priestley morphism. Therefore, generalized Priestley spaces and strong Priestley morphisms form a category in which composition is the usual set-theoretic composition of functions and the identity morphism is the usual identity function.
We denote this category by $\mathsf{PS}^{\mathsf{S}}$.
%
%
%
%
%

Let $X$ and $Y$ be generalized Priestley spaces and $R\subseteq X\times Y$ be a functional generalized Priestley morphism. We define
$f^{R}: X\to Y$ by $$f^R(x) = \text{the least element of } R[x].$$

\begin{lemma}\label{lemma-a}
If $X$ and $Y$ are generalized Priestley spaces and $R\subseteq X\times Y$ is a functional generalized Priestley morphism, then $f^{R}: X \to Y$
is a strong Priestley morphism. Moreover, if $X, Y,$ and $Z$ are generalized Priestley spaces and $R\subseteq X\times Y$ and $S\subseteq Y\times Z$ are functional generalized Priestley morphisms, then $f^{S * R} = f^{S} \circ f^{R}$.
\end{lemma}

\begin{proof}
Let $x, y \in X$. If $x \leq y$, then $R[y] \subseteq R[x]$, and so $f^R(x)\leq f^{R}(y)$. Therefore, $f^R$ is order-preserving. Now let $U \in Y^{*}$. Then $x\in (f^R)^{-1}(U)$ iff $f^{R}(x)\in U$ iff ${\uparrow}f^{R}(x)\subseteq U$ iff $R[x]\subseteq U$ iff $x\in\Box_R(U)$. Thus, $(f^R)^{-1}(U) = \Box_{R}U$, so $(f^R)^{-1}(U) \in X^{*}$, and so $f^R$ is a strong Priestley morphism. Moreover, since for functional generalized Priestley morphisms we have that $*$ coincides with $\circ$, it is easy to see that if  $R\subseteq X\times Y$ and $S\subseteq Y\times Z$ are functional generalized Priestley morphisms, then $f^{S * R} = f^{R \circ S} = f^{S} \circ f^{R}$.
\end{proof}

Now let $X$ and $Y$ be generalized Priestley spaces and $f:X\to Y$ be a strong Priestley morphism. Define $R^{f}\subseteq X \times Y$ by
$$\text{$x R^{f} y\ $ iff $\ f(x)\leq y$.}$$

\begin{lemma}\label{lemma-b}
If $X$ and $Y$ are generalized Priestley spaces and  $f:X\to Y$ is a strong Priestley morphism, then $R^{f}$ is a functional generalized Priestley morphism. Moreover, if $X, Y,$ and $Z$ are generalized Priestley spaces and $f:X\to Y$ and $g:Y\to Z$ are strong Priestley morphisms, then $R^{g \circ f} =  R^{g} * R^{f}$.
\end{lemma}

\begin{proof}
Let $x\in X$ and $y\in Y$. If $x\nR^f y$, then $f(x)\not\leq y$. Therefore, there exists $U \in Y^{*}$ such that $f(x)\in U$ and $y\notin U$. Thus, $R^f[x]={\uparrow}f(x)\subseteq U$ and $y\notin U$, and so condition (1) of Definition \ref{definition} is satisfied. Let $U \in Y^{*}$. For $x\in X$ we have $x\in\Box_{R^f}(U)$ iff $R^f[x]\subseteq U$ iff ${\uparrow}f(x)\subseteq U$ iff $f(x)\in U$ iff $x\in f^{-1}(U)$.
Therefore, $\Box_{R^f}(U)=f^{-1}(U)$, so $\Box_{R^f}(U) \in X^{*}$ and so condition (2) of Definition \ref{definition} is satisfied. Thus, $R^f$ is a  generalized Priestley morphism. Now, since $R^f[x]={\uparrow} f(x)$ for each $x\in X$, we have that $R^f$ is a functional generalized Priestley morphism. Moreover, since for functional generalized Priestley morphisms we have that $*$ coincides with $\circ$, it is easy to see that if  $f:X\to Y$ and $g:Y\to Z$ are strong Priestley morphisms, then $R^{g \circ f} = R^{f} \circ R^{g} = R^{g} * R^{f}$.
\end{proof}

\begin{lemma}\label{lemma-c}
Let $X$ and $Y$ be generalized Priestley spaces, $R\subseteq X\times Y$ be a functional generalized Priestley morphism, and $f:X\to Y$ be a strong Priestley morphism. Then $R^{f^R}=R$ and $f^{R^f}=f$.
\end{lemma}

\begin{proof}
For each $x\in X$ let $l_x$ be the least element of $R[x]$. Then $x R^{f^R} y$ iff $f^R(x)\leq y$ iff $l_x\leq y$ iff $x R l_x \leq y$ iff $x R y$. Thus, $R^{f^R}=R$. Also, $f^{R^f}(x)= $ the least element of $R^f[x]=f(x)$, so $f^{R^f}=f$.
\end{proof}


As an immediate consequence of Lemmas \ref{lemma-a}, \ref{lemma-b}, and \ref{lemma-c}, we obtain:

\begin{proposition} \label{sp-fr-isom}
The categories $\mathsf{GPS}^{\mathsf{F}}$ and $\mathsf{GPS}^{\mathsf{S}}$ are isomorphic.
\end{proposition}

This together with Theorem \ref{theorem-duality-sup-hom}.2, immediately give us:

\begin{theorem} \label{bdmsh-gpsf}
The categories $\mathsf{BDM}^{\mathsf{S}}$ and $\mathsf{GPS}^{\mathsf{S}}$ are dually equivalent.
\end{theorem}

An explicit construction of functors from $\mathsf{BDM}^{\mathsf{S}}$ to $\mathsf{GPS}^{\mathsf{S}}$ and vice versa can be obtained based on the following observation.

\begin{lemma} \label{obs-functional-1}
\begin{enumerate}
\item[]
\item Let $X$ and $Y$ be generalized Priestley spaces and $f: X \to Y$ be a strong Priestley morphism. Then for each $U \in Y^{\ast}$ we have $h_{R^f}(U) = f^{-1}(U)$.
\item Let $L$ and $K$ be bounded distributive meet semi-lattices and $h: L \to K$ be a sup-homomorphism. Then $f^{R_h}(y) = h^{-1}(y)$ for each $y \in K_{\ast}$.
\end{enumerate}
\end{lemma}

\begin{proof}
(1) Let $x \in h_{R^f}(U) =  \Box_{R^f}U$. Then $x R^f y$ implies $y \in U$. Since $x R^f f(x)$, it follows that $f(x) \in U$, and so $x \in f^{-1}(U)$. Now let $x \in f^{-1}(U)$. Then $f(x) \in U$. Therefore, if $x R^f y$, then $f(x) \leq y$, and so $y \in U$. Thus, $x \in \Box_{R^f}U = h_{R^f}(U)$, and so $h_{R^f}(U) = f^{-1}(U)$.

(2) Since $h$ is a sup-homomorphism, $R_h$ is a functional generalized Priestley morphism, and $h^{-1}(y)$ is the least element of $R_h[y]$. Thus, $a \in f^{R_h}(y)$ iff $a$ belongs to the least element of $R_h[y]$ iff $a \in h^{-1}(y)$.
\end{proof}

Now we can define the functors $(-)_{\star}: \mathsf{BDM}^{\mathsf{S}} \to \mathsf{GPS}^{\mathsf{S}}$ and $(-)^{\star}: \mathsf{GPS}^{\mathsf{S}} \to \mathsf{BDM}^{\mathsf{S}}$ explicitly as follows: If $L$ is a bounded distributive meet semi-lattice, then $L_\star=L_{*}$ and if $h:L\to K$ is a sup-homomorphism, then $h_{\star} = h^{-1}$; also, if $X$ is a generalized Priestley space, then $X^\star=X^*$, and if $f:X \to Y$ is a strong Priestley morphism, then $f^{\star} = f^{-1}$. Therefore, the functors $(-)_\star:\mathsf{BDM}^{\mathsf{S}}\to\mathsf{GPS}^{\mathsf{S}}$ and $(-)^\star:\mathsf{GPS}^{\mathsf{S}}\to\mathsf{BDM}^{\mathsf{S}}$ behave exactly like the Priestley functors $(-)_*:\mathsf{BDL}\to\mathsf{PS}$ and $(-)^*:\mathsf{PS}\to\mathsf{BDL}$.


\subsection{Functional generalized Esakia morphisms}

Our next task is to discuss the strong Priestley morphisms between generalized Esakia spaces which correspond to the functional generalized Esakia morphisms.

\begin{lemma} \label{funtionmorphism-Esakiamorphisms}
Let $X$ and $Y$ be generalized Esakia spaces and $f: X \to Y$ be a strong Priestley morphism. Then $R^{f}$ is a generalized Esakia morphism iff for each $x \in X$ and $y \in Y_0$, from $f(x) \leq y$ it follows that there exists $z \in X_0$ such that $x \leq z$ and $f(z) = y$.
\end{lemma}

\begin{proof}
Suppose that $R^{f}$ is a generalized Esakia morphism. Let $x \in X$ and $y \in Y_0$. If $f(x)\leq y$, then $xR^fy$. Therefore, there exists $z \in X_0$ such that $x \leq z$ and $R^{f}[z] = {\uparrow}y$. It follows from the definition of $R^{f}$ that $f(z) = y$. Conversely, suppose that for each $x \in X$ and $y \in Y_0$, from $f(x) \leq y$ it follows that there exists $z \in X_0$ such that $x \leq z$ and $f(z) = y$. Let $x \in X$ and $y \in Y_0$. If $x R^{f} y$, then $f(x) \leq y$. Therefore, there exists $z \in X_0$ such that $x \leq z$ and $f(z) = y$. Thus, $R^{f}[z] = {\uparrow}y$, and so $R^{f}$ is a generalized Esakia morphism.
\end{proof}

\begin{definition} \label{functionEsakiamorphism}
{\rm Let $X$ and $Y$ be generalized Esakia spaces. We call a map $f: X \to Y$ a \textit{strong Esakia morphism} if it is a strong Priestley morphism such that for each $x \in X$ and $y \in Y_0$, from $f(x) \leq y$ it follows that there exists $z \in X_0$ such that $x \leq z$ and $f(z) = y$.
}
\end{definition}

As an immediate consequence of Lemmas \ref{lemma-c} and \ref{funtionmorphism-Esakiamorphisms}, we obtain:

\begin{lemma} \label{functionEsakiamorphism-2}
Let $X$ and $Y$ be generalized Esakia spaces and $R \subseteq X \times Y$ be a functional generalized Priestley morphism. Then $f^R$ is a strong Esakia morphism iff $R$ is a generalized Esakia morphism.
\end{lemma}

Let $\mathsf{GES}^{\mathsf{S}}$ denote the category of generalized Esakia spaces and strong Esakia morphisms. By Proposition \ref{sp-fr-isom} and Lemmas \ref{funtionmorphism-Esakiamorphisms} and \ref{functionEsakiamorphism-2}, we obtain:

\begin{proposition} \label{prop-func-str}
The categories $\mathsf{GES}^{\mathsf{S}}$ and $\mathsf{GES}^{\mathsf{F}}$ are isomorphic.
\end{proposition}

Let also $\mathsf{BIM}^{\mathsf{S}}$ denote the category of bounded implicative meet semi-lattices and implicative meet semi-lattice sup-homomorphisms. As an immediate consequence of Theorem \ref{theorem-duality-sup-hom} and Proposition \ref{prop-func-str}, we obtain:

\begin{theorem}
\label{BIM-sup-Esakia-strong}
The categories $\mathsf{BIM}^{\mathsf{S}}$ and $\mathsf{GES}^{\mathsf{S}}$ are dually equivalent.
\end{theorem} 
\section{Priestley and Esakia dualities as particular cases}

In this section we show how the well-known Priestley and Esakia
dualities are particular cases of our dualities for bounded
distributive meet semi-lattices and bounded implicative meet
semi-lattices, respectively. We also obtain an application to modal
logic by showing that descriptive frames, which are duals of modal
algebras, are exactly generalized Priestley morphisms of Stone
spaces into themselves. Finally, we give a new dual representation
of Heyting algebra homomorphisms by means of special partial
functions. This yields a new duality for Heyting algebras, which is
an alternative to the Esakia duality.

\subsection{Priestley duality as a particular case}

We start by showing how the Priestley duality between the category $\mathsf{BDL}$ of  bounded
distributive lattices and bounded lattice homomorphisms on the one
hand and the category $\mathsf{PS}$ of Priestley spaces and Priestley morphisms on the other follows from  Theorem \ref{bdmsh-gpsf}.

Let $L$ be a bounded distributive lattice. Then $L_*=L_+$, and so $\la L_*,\tau,\subseteq,L_+\ra = {\la L_+,\tau,\subseteq\ra}$ is a Priestley space.
Conversely, if $X=\la X,\tau,\leq \ra$ is a Priestley space, then $X^*=\mathfrak{CU}(X)$. Therefore, given two Priestley spaces $X$ and $Y$, a map $f:X\to Y$ is a strong Priestley morphism iff $f$ is order-preserving and $V\in\mathfrak{CU}(Y)$ implies $f^{-1}(V)\in\mathfrak{CU}(X)$. Because $\mathfrak{CU}(X)\cup\{U^c:U\in\mathfrak{CU}(X)$ and $\mathfrak{CU}(Y)\cup\{V^c:V\in\mathfrak{CU}(Y)$ are subbases for the Priestley topologies on $X$ and $Y$, respectively, the last condition is equivalent to $f$ being continuous. Thus, the notions of a strong Priestley morphism and of a Priestley morphism coincide.


\begin{lemma}\label{lemma-dist}
Let $L$ and $K$ be bounded distributive lattices and $h:L\to K$ be a bounded meet semi-lattice homomorphism. Then the following conditions are equivalent:
\begin{enumerate}
\item $h$ preserves $\vee$.
\item $h$ is a sup-homomorphism.
\item $h^{-1}(x)\in L_+$ for each $x\in K_+$.
\end{enumerate}
\end{lemma}

\begin{proof}
This is an immediate consequence of Propositions \ref{prop-sup} and \ref{sup-hom-pres-exist-joins} and the fact that $L_*=L_+$ and $K_*=K_+$.
\end{proof}

The well-known Priestely duality is now an immediate consequence of Theorem \ref{bdmsh-gpsf} and Lemmas \ref{obs-functional-1} and \ref{lemma-dist}.

\begin{corollary}
The category $\mathsf{BDL}$ is dually equivalent to the category
$\mathsf{PS}$.
\end{corollary}


Let $\mathsf{BDL}^{\wedge,\top}$ denote the
category of bounded distributive lattices and meet semi-lattice
homomorphisms preserving top and let
$\mathsf{BDL}^{\wedge,\top,\bot}$ denote the category of bounded
distributive lattices and bounded meet semi-lattice homomorphisms.
Clearly
$\mathsf{BDL}$ is a proper subcategory of
$\mathsf{BDL}^{\wedge,\top,\bot}$ and
$\mathsf{BDL}^{\wedge,\top,\bot}$ is a proper subcategory of
$\mathsf{BDL}^{\wedge,\top}$. Let $\mathsf{PS}^\mathsf{R}$ denote
the category of Priestley spaces and generalized Priestley
morphisms, $\mathsf{PS}^\mathsf{T}$ denote the category
of Priestley spaces and total generalized Priestley morphisms, and $\mathsf{PS}^\mathsf{F}$ denote the category of
Priestley spaces and functional generalized Priestley morphisms. Clearly
$\mathsf{PS}^\mathsf{F}$ is a proper subcategory of
$\mathsf{PS}^\mathsf{T}$ and $\mathsf{PS}^\mathsf{T}$ is a proper subcategory of
$\mathsf{PS}$. As an immediate consequence of Theorems
\ref{theorem-duality} and \ref{theorem-duality-sup-hom} and Proposition \ref{sp-fr-isom},
we obtain:

\begin{corollary}\label{cor-DL}
\begin{enumerate}
\item[]
\item The category $\mathsf{BDL}^{\wedge,\top}$ is dually equivalent to the
category $\mathsf{PS}^\mathsf{R}$.
\item The category $\mathsf{BDL}^{\wedge,\top,\bot}$ is dually equivalent to
the category $\mathsf{PS}^\mathsf{T}$.
\item The category $\mathsf{BDL}$ is dually equivalent to the category
$\mathsf{PS}^\mathsf{F}$, which is isomorphic to $\mathsf{PS}$.
\end{enumerate}
\end{corollary}

\subsubsection{Application to modal logic}

Our results have an application to modal logic. Let $(B,\Box)$ be
a \emph{modal algebra}; that is, $B$ is a Boolean algebra and
$\Box:B\to B$ satisfies $\Box(a\wedge b)=\Box(a) \wedge\Box(b)$
and $\Box \top=\top$. It follows from the well-known duality for
modal algebras that the dual objects corresponding to modal
algebras are \emph{descriptive frames} $(X,R)$, where $X$ is a
Stone space and $R$ is a binary relation on $X$ such that (i)
$R[x]$ is closed for each $x\in X$, and (ii) $R^{-1}[U]$ is clopen
whenever $U$ is clopen. Obviously $\Box:B\to B$ is a particular
case of a meet semi-lattice homomorphism (between Boolean
algebras) preserving top. We show that for Stone spaces $X$,
generalized Priestley morphisms $R\subseteq X\times X$ are exactly
those binary relations $R$ on $X$ for which $(X,R)$ is a
descriptive frame. Let $X$ be a Stone space and let $R\subseteq
X\times X$ be a generalized Priestley morphism. Since $\leq_X$ is simply equality, $X^*$ coincides with the
Boolean algebra of clopen subsets of $X$. Therefore, condition (2) of
Definition \ref{definition} is equivalent to condition (ii). We
show that condition (1) of Definition \ref{definition} is
equivalent to condition (i). Suppose that condition (1) is
satisfied and $y\notin R[x]$. Then $x\nR y$, and by condition (1),
there exists a clopen subset $U$ of $X$ such that $R[x]\subseteq
U$ and $y\notin U$. Thus, $R[x]$ is closed and condition (i) is
satisfied. Conversely suppose that condition (i) is satisfied and
$x\nR y$. Then $y\notin R[x]$. Since $R[x]$ is closed and $X$ is a
Stone space, there exists a clopen subset $U$ of $X$ such that
$R[x]\subseteq U$ and $y\notin U$. Thus, condition (1) is
satisfied.

\subsection{Esakia duality as a particular case}

Now we show that Esakia's duality between the category $\mathsf{HA}$ of Heyting algebras and Heyting algebra homomorphisms and the category $\mathsf{ES}$ of Esakia spaces and Esakisa morphisms follows from Theorem \ref{BIM-sup-Esakia-strong}. As with Priestely spaces, Esakia spaces are simply generalized Esakis spaces $X = \la X, \leq, \tau, X_0\ra$ in which $X_0 = X$. Consequently, the concepts of an Esakia morphism and of a strong Esakia morphism coincide. Thus, by Lemmas \ref{lemma-a}, \ref{lemma-b}, \ref{funtionmorphism-Esakiamorphisms}, and \ref{functionEsakiamorphism-2}, for Esakia spaces $X$ and $Y$, if $R\subseteq X\times Y$ is a functional generalized Esakia morphism, then $f^R:X\to Y$ is an Esakia morphism, and if $f:X\to Y$ is an Esakia morphism, then $R^{f}\subseteq X\times Y$ is a functional generalized Esakia morphism.
%
%
This together with Theorem \ref{BIM-sup-Esakia-strong} gives us the well-known Esakia duality:

\begin{corollary}
The category $\mathsf{HA}$ of Heyting algebras and Heyting algebra
homomorphisms is dually equivalent to the category $\mathsf{ES}$.
\end{corollary}

Let $\mathsf{HA}^{\wedge,\to}$ denote the category of Heyting
algebras and implicative meet semi-lattice homomorphisms, and let
$\mathsf{HA}^{\wedge,\to,\bot}$ denote the category of Heyting
algebras and bounded implicative meet semi-lattice homomorphisms.
Clearly $\mathsf{HA}$ is
a proper subcategory of $\mathsf{HA}^{\wedge,\to,\bot}$ and
$\mathsf{HA}^{\wedge,\to,\bot}$ is a proper subcategory of
$\mathsf{HA}^{\wedge,\to}$. Let $\mathsf{ES}^\mathsf{R}$ denote the
category of Esakia spaces and generalized Esakia morphisms,
$\mathsf{ES}^\mathsf{T}$ denote the category of Esakia spaces and
total generalized Esakia morphisms, and $\mathsf{ES}^\mathsf{F}$
denote the category of Esakia spaces and functional generalized
Esakia morphisms. Clearly $\mathsf{ES}^\mathsf{F}$ is a proper
subcategory of $\mathsf{ES}^\mathsf{T}$ and
$\mathsf{ES}^\mathsf{T}$ is a proper subcategory of
$\mathsf{ES}^\mathsf{R}$.  As an immediate
consequence of Theorem \ref{theorem-ges} and Corollary \ref{cor-DL},
we obtain:

\begin{corollary}\label{cor-HA}
\begin{enumerate}
\item[]
\item The category $\mathsf{HA}^{\wedge,\to}$ is dually equivalent
to the category $\mathsf{ES}^\mathsf{R}$.
\item The category $\mathsf{HA}^{\wedge,\to,\bot}$ is dually equivalent
to the category $\mathsf{ES}^\mathsf{T}$.
\item The category $\mathsf{HA}$ is dually equivalent to the category
$\mathsf{ES}^\mathsf{F}$, which is isomorphic to $\mathsf{ES}$.
\end{enumerate}
\end{corollary}

%

\subsection{Partial Esakia functions}

Now we give an alternative description of morphisms of
$\mathsf{ES}^\mathsf{R}$, $\mathsf{ES}^\mathsf{T}$, and
$\mathsf{ES}\simeq\mathsf{ES}^{\mathsf{F}}$ by means of special partial functions between Esakia
spaces, thus obtaining new dualities for
$\mathsf{HA}^{\wedge,\to}$, $\mathsf{HA}^{\wedge,\to,\bot}$, and
$\mathsf{HA}$. The new duality for $\mathsf{HA}$ is an alternative
to the Esakia duality.

Let $X$ and $Y$ be Esakia spaces and let $R\subseteq X\times Y$ be a
generalized Esakia morphism. We define an equivalence relation $E_R$
on $X$ by $$x E_R y \mbox{ iff } R[x] = R[y].$$

\begin{lemma}
Let $X$ and $Y$ be Esakia spaces and let $R\subseteq X\times Y$ be a
generalized Esakia morphism. For each $x\in X$ we have that the
equivalence class $E_R[x]$ is closed.
\end{lemma}

\begin{proof}
Let $y\notin E_R[x]$. Then $R[x]\not\subseteq R[y]$ or $R[y]
\not\subseteq R[x]$. Therefore, there exists $u\in Y$ such that
$xRu$ and $y\nR u$ or there exists $v\in Y$ such that $x\nR v$ and
$yRv$. Since $R$ is a generalized Esakia morphism, hence a
generalized Priestley morphism, there exists a clopen upset $U$ of
$Y$ such that $R[y]\subseteq U$ and $u\notin U$ or there exists a
clopen upset $V$ of $Y$ such that $R[x]\subseteq V$ and $v\notin V$.
Thus, $y\in \Box_R U$ and $E_R[x]\cap\Box_R U = \emptyset$ or
$E_R[x]\subseteq\Box_R V$ and $y\notin\Box_R V$, and so $E_R[x]\cap
(\Box_R V)^c=\emptyset$ and $y\in (\Box_R V)^c$. In either case,
there exists a clopen subset $W$ of $X$ ($W=\Box_R U$ or $W=(\Box_R
V)^c$) such that $y\in W$ and $E_R[x]\cap W=\emptyset$.
Consequently, $E_R[x]$ is a closed subset of $X$.
\end{proof}

Since $X$ is an Esakia space, hence a Priestley space, and $E_R[x]$
is closed, for each $y\in E_R[x]$ there exists $z\in
\mathrm{max}E_R[x]$ such that $y\leq z$. We define a partial
function $f_R: X\to Y$ as follows. Let $$\mathrm{dom}(f_R) = \{x\in
X: R[x] \mbox{ has a least element and } x \in
\mathrm{max}E_R[x]\},$$ and for $x \in \mathrm{dom}(f_R)$ let
$$f_R(x)=\mbox{the least element of } R[x].$$ For $x,y\in X$ we use the
standard notation $x<y$ whenever $x\leq y$ and $x\neq y$.

\begin{lemma}\label{lemma-a1}
Let $X$ and $Y$ be Esakia spaces and let $R\subseteq X\times Y$ be a
generalized Esakia morphism. Then:
\begin{enumerate}
\item For $x,y\in\mathrm{dom}(f_R)$, if $x<y$, then $f_R(x)<f_R(y)$.
\item For $x\in\mathrm{dom}(f_R)$ and $y\in Y$, from $f_R(x)<y$
it follows that there exists $z\in \mathrm{dom}(f_R)$ such that
$x<z$ and $f_R(z)=y$.
\item For $x\in X$ and $y\in Y$, we have $xRy$ iff there exists
$z\in\mathrm{dom}(f_R)$ such that $x\leq z$ and $f_R(z)=y$.
\item If $U\in\mathfrak{CU}(Y)$, then $\Box_R U = ({\downarrow}
f_R^{-1}(U^c))^c$.
\item For $x\in X$ and $y\in Y$, if $y\notin f_R[{\uparrow}x]$,
then there exists $U\in\mathfrak{CU}(Y)$ such that $x\in
({\downarrow}f_R^{-1}(U^c))^c$ and $y\notin U$.
\end{enumerate}
In addition, if $R$ is total, then $\mathrm{max}X\subseteq
\mathrm{dom}(f_R)$.
\end{lemma}

\begin{proof}
Let $X$ and $Y$ be Esakia spaces and let $R\subseteq X\times Y$ be
a generalized Esakia morphism. To see (1), let $x, y \in
\mathrm{dom}(f_R)$ with $x<y$. Then $x\leq y$, and so $R[y]
\subseteq R[x]$. Thus, $f_R(y) \in R[x]$, and so $f_R(x) \leq
f_R(y)$. If $f_R(x) = f_R(y)$, then $R[x] = R[y]$. Therefore, $x
E_R y$ and $x<y$, so $x\notin \mathrm{max}E_R[x]$, a
contradiction. Thus, $f_R(x) < f_R(y)$. To see (2), let
$x\in\mathrm{dom}(f_R)$, $y\in Y$, and $f_R(x)<y$. Then
$f_R(x)\leq y$, so $x R y$. Since $R$ is a generalized Esakia
morphism, there exists $z \in X$ such that $x \leq z$ and $R[z] =
{\uparrow} y$. If $x = z$, then $R[x] = R[z]$, and so $R[x] =
{\uparrow} y$, which implies that $f_R(x) = y$, a contradiction.
Thus, $x < z$. Let $u \in \mathrm{max}E_R[z]$ be such that $z \leq
u$. Then $x < u$, $u \in \mathrm{dom}(f_R)$, and $R[u] = R[z] =
{\uparrow} y$. Thus, $f_R(u) = y$. For (3), it is clear that if
there exists $z\in\mathrm{dom}(f_R)$ such that $x\leq z$ and
$f_R(z)=y$, then $x\leq zRy$, and so $xRy$. Conversely, if $xRy$,
then as $R$ is a generalized Esakia morphism, there exists $z\in
X$ such that $x\leq z$ and $R[z]={\uparrow}y$. Let $u \in
\mathrm{max}E_R[z]$ be such that $z \leq u$. Then $x\leq u$, $u
\in \mathrm{dom}(f_R)$, and $f_R(u)=y$. To see (4), let
$U\in\mathfrak{CU}(Y)$. We have that $x\in \Box_R U$ iff
$R[x]\subseteq U$, and that $x\in ({\downarrow}f_R^{-1}(U^c))^c$
iff $(\forall z\in \mathrm{dom}(f_R))(x\leq z$ implies $f_R(z)\in
U)$. First suppose that $x\in \Box_R U$, $z\in\mathrm{dom}(f_R)$,
and $x\leq z$. Then $f_R(z)\in R[z]\subseteq R[x]\subseteq U$, so
$f_R(z)\in U$, and so $x\in ({\downarrow}f_R^{-1}(U^c))^c$. Now
suppose that $x\in ({\downarrow}f_R^{-1}(U^c))^c$ and $xRy$. Since
$R$ is a generalized Esakia morphism, there exists $z\in X$ such
that $x\leq z$ and $R[z]={\uparrow}y$. Let $m\in\mathrm{max}
E_R[z]$ be such that $z\leq m$. Then $m\in \mathrm{dom}(f_R)$,
$x\leq m$, and $f_R(m)=y$. Thus, $f_R(m)\in U$, so $y\in U$, and
so $x\in\Box_R U$. To see (5), let $x\in X$, $y\in Y$, and
$y\notin f_R[{\uparrow}x]$. By (3), $f_R[{\uparrow}x]=R[x]$.
Therefore, $y\notin f_R[{\uparrow}x]$ implies $x\nR y$. Since $R$
is a generalized Esakia morphism, hence a generalized Priestley
morphism, there exists $U\in\mathfrak{CU}(Y)$ such that
$x\in\Box_R U$ and $y\notin U$. By (4), $\Box_R U =
({\downarrow}f_R^{-1}(U^c))^c$. Thus, $x\in
({\downarrow}f_R^{-1}(U^c))^c$ and $y\notin U$. Finally, let $R$
be total and let $x\in\mathrm{max}X$. Then there exists $y\in Y$
such that $xRy$. From $y\in Y$ it follows that there exists
$u\in\mathrm{max}Y$ such that $y\leq u$. So $xRy\leq u$, implying
that $xRu$. Since $R$ is a generalized Esakia morphism, there
exists $z\in X$ such that $x\leq z$ and $R[z]={\uparrow}u=\{u\}$.
But $x\in \mathrm{max}X$ and $x\leq z$ imply $x=z$. Therefore,
$x\in \mathrm{max}E_R[x]$ and $R[x]=\{u\}$. Thus,
$x\in\mathrm{dom}(f_R)$ and $f_R(x)=u$. Consequently,
$\mathrm{max}X\subseteq \mathrm{dom}(f_R)$.
\end{proof}

Lemma \ref{lemma-a1} motivates the following definition of a partial
Esakia function between Esakia spaces.

\begin{definition}\label{def-pEf}
{\rm Let $X$ and $Y$ be Esakia spaces and let $f:X\to Y$ be a partial
function. We call $f$ a \emph{partial Esakia function} if $f$
satisfies the following four conditions:
\begin{enumerate}
\item For $x,y\in\mathrm{dom}(f)$, if $x<y$, then $f(x)<f(y)$.
\item For $x\in\mathrm{dom}(f)$ and $y\in Y$, from $f(x)<y$
it follows that there exists $z\in \mathrm{dom}(f)$ such that $x<z$
and $f(z)=y$.
\item If $U\in\mathfrak{CU}(Y)$, then $({\downarrow}f^{-1}(U^c))^c\in
\mathfrak{CU}(X)$.
\item For $x\in X$ and $y\in Y$, if $y\notin f[{\uparrow}
x]$, then there exists $U\in\mathfrak{CU}(Y)$ such that
$f[{\uparrow} x]\subseteq U$ and $y\notin U$.
\end{enumerate}
If in addition $\mathrm{max}X\subseteq\mathrm{dom}(f)$, then we call
$f$ \emph{well}.
}
\end{definition}

Let $X$ and $Y$ be Esakia spaces and let $f:X\to Y$ be a partial
Esakia function. We define $R_f\subseteq X \times Y$ by
$$\text{$x R_f y\ $ iff $\ \exists z\in \mathrm{dom}(f)\ $ such that
$\ x\leq z\ $ and $\ f(z) = y$.}$$

\begin{lemma}\label{lemma-b1}
Let $X$ and $Y$ be Esakia spaces and let $f:X\to Y$ be a partial
Esakia function. Then $R_f$ is a generalized Esakia morphism. If in
addition $f$ is well, then $R_f$ is total.
\end{lemma}

\begin{proof}
Let $x,y\in X$, $u\in Y$, $x\leq y$, and $yR_fu$. Then there
exists $z\in\mathrm{dom}(f)$ such that $y\leq z$ and $f(z)=u$.
Therefore, $x\leq z$ and $f(z)=u$, implying that $xR_fu$. Thus,
condition (1) of Definition \ref{definition} is satisfied. Let
$x\in X$, $u,v\in Y$, $xR_fu$, and $u\leq v$. Then there exists
$y\in\mathrm{dom}(f)$ such that $x\leq y$ and $f(y)=u$. Therefore,
$f(y)\leq v$, and as $R$ is a generalized Esakia morphism, there
exists $z\in\mathrm{dom}(f)$ such that $y\leq z$ and $f(z)=v$.
Thus, $x\leq z$ and $f(z)=v$, so $xR_fv$, and so condition (2) of
Definition \ref{definition} is satisfied. Let $x\in X$, $y\in Y$,
and $x\nR_f y$. Therefore, $y\notin f[{\uparrow}x]$. Since $f$ is
a partial Esakia function, there exists $U\in\mathfrak{CU}(Y)$
such that $f[{\uparrow}x]\subseteq U$ and $y\notin U$. Thus, $x\in
({\downarrow}f^{-1}(U^c))^c$ and $y\notin U$. We show that
$\Box_{R_f} U = ({\downarrow}f^{-1}(U^c))^c$. For $x\in X$ we have
$x\in\Box_{R_f} U$ iff $R_f[x]\subseteq U$, and $x\in
({\downarrow}f^{-1}(U^c))^c$ iff $(\forall z\in\mathrm{dom}(f))
(x\leq z$ implies $f(z)\in U)$. First suppose that $x\in\Box_{R_f}
U$, $z\in \mathrm{dom}(f)$, and $x\leq z$. Then $xR_f f(z)$, so
$f(z)\in U$, and so $x\in ({\downarrow}f^{-1}(U^c))^c$. Now
suppose that $x\in ({\downarrow}f^{-1}(U^c))^c$ and $xR_f y$. Then
there exists $z\in \mathrm{dom}(f)$ such that $x\leq z$ and
$f(z)=y$. Therefore, $f(z)\in U$, so $y\in U$, and so
$x\in\Box_{R_f} U$. Consequently, $\Box_{R_f} U =
({\downarrow}f^{-1}(U^c))^c$. But then $x\in \Box_{R_f}U$ and
$y\notin U$, and so condition (3) of Definition \ref{definition}
is satisfied. Let $U\in\mathfrak{CU}(Y)$. Since $\Box_{R_f} U =
({\downarrow}f^{-1}(U^c))^c$ and $({\downarrow}f^{-1}(U^c))^c\in
\mathfrak{CU}(X)$, we have $\Box_{R_f} U\in\mathfrak{CU}(X)$, and
so condition (4) of Definition \ref{definition} is satisfied. It
follows that $R_f$ is a generalized Priestley morphism. To see
that $R_f$ is a generalized Esakia morphism, let $x\in X$, $y\in
Y$, and $xR_fy$. Then there exists $z\in\mathrm{dom}(f)$ such that
$x\leq z$ and $f(z)=y$. We show that $R_f[x] = {\uparrow} y$.
Clearly ${\uparrow} y \subseteq R_f[x]$. If $x R_f u$, then there
exists $v\in \mathrm{dom}(f)$ such that $x \leq v$ and $f(v)=u$.
Therefore, $f(x) \leq u$, so $y\leq u$, and so $u\in {\uparrow}
y$. Thus, $R_f[x]\subseteq {\uparrow} y$, and so $R_f[x] =
{\uparrow} y$. Consequently, there exists $z\in X$ such that
$x\leq z$ and $R_f[z]={\uparrow}y$, and so $R_f$ is a generalized
Esakia morphism. Lastly suppose that $f$ is a well partial Esakia
function and $x\in X$. Then there exists $z\in\mathrm{max}X$ such
that $x\leq z$. Since $f$ is well, $z\in\mathrm{dom}(f)$.
Therefore, $f(z)\in Y$ and $xR_f f(z)$. Thus, there exists
$y=f(z)$ in $Y$ such that $xR_f y$, so ${R_f}^{-1}[Y]=X$, and so
$R_f$ is a total generalized Esakia morphism.
\end{proof}

\begin{lemma}\label{lemma-c1}
Let $X$ and $Y$ be Esakia spaces, $R\subseteq X\times Y$ be a
generalized Esakia morphism, and $f:X\to Y$ be a partial Esakia
function. Then $R_{f_R}=R$ and $f_{R_f}=f$.
\end{lemma}

\begin{proof}
Let $x\in X$ and $y\in Y$. By Lemma \ref{lemma-a1}, $xRy$ iff
there exists $z\in\mathrm{dom}(f_R)$ such that $x\leq z$ and
$f_R(z)=y$, which by the definition of $R_{f_R}$ means that
$xR_{f_R}y$. Thus, $R_{f_R}=R$. Now let $x \in
\mathrm{dom}(f_{R_f})$ and let $f_{R_f}(x)=y$. Then $x\in
\mathrm{max}E_{R_f}[x]$ and $R_f[x] = {\uparrow}y$. Therefore,
$xR_fy$, and so there exists $z\in\mathrm{dom}(f)$ such that
$x\leq z$ and $f(z)=y$. Thus, $R_f[z] ={\uparrow} y$. It follows
that $z\in E_{R_f}[x]$ and since $x\in\mathrm{max}E_{R_f}[x]$ and
$x\leq z\in E_{R_f}[x]$, we obtain $x=z$. Thus,
$x\in\mathrm{dom}(f)$ and $f_{R_f}(x)=f(x)$, so $f_{R_f} = f$.
Conversely, let $x\in\mathrm{dom}(f)$ and let $f(x)=y$. Then
$R_f[x]={\uparrow}y$. We show that $x\in \mathrm{max}E_{R_f}[x]$.
Let $z\in E_{R_f}[x]$ and $x\leq z$. Then
$R_f[z]=R_f[x]={\uparrow}y$. Therefore, $zR_fy$, and so there
exists $z'\in\mathrm{dom}(f)$ such that $z\leq z'$ and $f(z')=y$.
Thus, $x\leq z'$ and $f(x)=f(z')$. This implies $x=z=z'$, and so
$x\in \mathrm{max}E_{R_f}[x]$. Therefore,
$x\in\mathrm{dom}(f_{R_f})$ and $f_{R_f}(x)=y=f(x)$. It follows
that $\mathrm{dom}(f_{R_f})=\mathrm{dom}(f)$ and $f_{R_f}(x)=f(x)$
for each $x\in\mathrm{dom}(f_{R_f})= \mathrm{dom}(f)$. Thus,
$f_{R_f}=f$.
\end{proof}

\begin{definition}\label{def-pHf}
{\rm Let $X$ and $Y$ be Esakia spaces and let $f:X\to Y$ be a partial
Esakia function. We call $f$ a \emph{partial Heyting function} if
for each $x\in X$ there exists $z\in\mathrm{dom}(f)$ such that
$x\leq z$ and $f[{\uparrow}x]={\uparrow}f(z)$.
}
\end{definition}

It is easy to verify that each partial Heyting function is well.

\begin{lemma}\label{lem-pHf}
Let $X$ and $Y$ be Esakia spaces. If $R\subseteq X\times Y$ is a
functional generalized Esakia morphism, then the partial Esakia
function $f_R$ is a partial Heyting function. Conversely, if $f:X\to
Y$ is a partial Heyting function, then the generalized Esakia
morphism $R_f$ is a functional generalized Esakia morphism.
\end{lemma}

\begin{proof}
Let $R\subseteq X\times Y$ be a functional generalized Esakia
morphism and let $x\in X$. Then $R[x]={\uparrow}y$ for some $y\in
Y$. Therefore, $xRy$, and so there exists $z\in \mathrm{dom}(f_R)$
such that $x\leq z$ and $f_R(z)=y$. Thus,
$f_R[{\uparrow}x]=R[x]={\uparrow}y= {\uparrow}f_R(z)$.
Consequently, $f_R$ is a partial Heyting function. Now let $f:X\to
Y$ be a partial Heyting function and let $x\in X$. Then there
exists $z\in\mathrm{dom}(f)$ such that $x\leq z$ and
$f[{\uparrow}x]={\uparrow}f(z)$. Thus, $R[x] =
f[{\uparrow}x]={\uparrow}f(z)$, and so $R$ is a functional
generalized Esakia morphism.
\end{proof}

Let $X,Y,$ and $Z$ be Esakia spaces and $f:X\to Y$ and $g:Y\to Z$ be partial Esakia functions. We define $g*f:X\to Y$ as the partial Esakia function $g*f=f_{(R_g*R_f)}$. By Lemma \ref{lemma-c1}, it is not difficult to see that Esakia spaces and partial Esakia functions form a category, we denote by $\mathsf{ES}^{\mathsf{P}}$. It is also clear that if both $f$ and $g$ are well (resp.\ Heyting), then so is $g*f$. Therefore, Esakia spaces and well partial Esakia functions as well as Esakia spaces and partial Heyting functions also form categories, we denote by $\mathsf{ES}^{\mathsf{W}}$ and $\mathsf{ES^{H}}$, respectively.
Obviously $\mathsf{ES^{H}}$ is a proper
subcategory of $\mathsf{ES}^{\mathsf{W}}$ and 
$\mathsf{ES}^{\mathsf{W}}$ is a proper subcategory of
$\mathsf{ES}^{\mathsf{P}}$. As an immediate consequence of
Corollary \ref{cor-HA} and Lemmas \ref{lemma-a1}, \ref{lemma-b1},
\ref{lemma-c1}, and \ref{lem-pHf}, we obtain:

\begin{corollary}\label{cor-a1}
\begin{enumerate}
\item[]
\item The category $\mathsf{ES}^\mathsf{R}$ is isomorphic to the
category $\mathsf{ES}^{\mathsf{P}}$.
\item The category $\mathsf{ES}^{\mathsf{T}}$ is isomorphic to the
category $\mathsf{ES}^{\mathsf{W}}$.
\item The categories $\mathsf{ES}$, $\mathsf{ES}^{\mathsf{F}}$, and
$\mathsf{ES}^{\mathsf{H}}$ are isomorphic.
\end{enumerate}
\end{corollary}

Now putting Corollaries \ref{cor-HA} and \ref{cor-a1} together, we
obtain:

\begin{corollary}\label{cor-parfun}
\begin{enumerate}
\item[]
\item The category $\mathsf{HA}^{\wedge,\to}$ is dually equivalent
to the category $\mathsf{ES}^{\mathsf{P}}$.
\item The category $\mathsf{HA}^{\wedge,\to,\bot}$ is dually
equivalent to the category $\mathsf{ES}^{\mathsf{W}}$.
\item The category $\mathsf{HA}$ is dually equivalent to the
category $\mathsf{ES}^{\mathsf{H}}$.
\end{enumerate}
\end{corollary}

In particular, this provides an alternative to the Esakia duality.
We give a direct proof of this by establishing that $\mathsf{ES}$
is isomorphic to $\mathsf{ES^{H}}$, thus providing an explicit
construction of an Esakia morphism from a partial Heyting function
and vice versa. Let $X$ and $Y$ be Esakia spaces and let $f:X\to
Y$ be a partial Heyting function. We define a function $g_f:X\to
Y$ as follows. Let $x\in X$. Since $f$ is a partial Heyting
function, there exists $z\in \mathrm{dom}(f)$ such that $x\leq z$
and $f[{\uparrow}x]={\uparrow}f(z)$. Such a $z$ may not be unique,
but all such $z$'s have the same $f$-image. Thus, we set
$g_f(x)=f(z)$.

\begin{lemma}\label{lem-pEf1}
If $X$ and $Y$ are Esakia spaces and $f:X\to Y$ is a partial Heyting
function, then $g_f:X\to Y$ is an Esakia morphism.
\end{lemma}

\begin{proof}
We need to show that $g_f$ is continuous, order-preserving, and
satisfies the following condition: For $x\in X$ and $y\in Y$, from
$g_f(x)\leq y$ it follows that there exists $z\in X$ with $x\leq
z$ and $g_f(z)=y$. Let $x,y\in X$ and $x\leq y$. Since $f$ is a
partial Heyting function, there exist $u,v\in \mathrm{dom}(f)$
such that $x\leq u$, $y\leq v$, $f[{\uparrow}x]={\uparrow}f(u)$,
and $f[{\uparrow}y]={\uparrow}f(v)$. Therefore, ${\uparrow}f(v)
\subseteq {\uparrow}f(u)$, so $f(u)\leq f(v)$, and so $g_f(x)\leq
g_f(y)$. Now let $x\in X$, $y\in Y$, and $g_f(x)\leq y$. Then
there exists $u\in\mathrm{dom}(f)$ such that $x\leq u$ and
$g_f(x)=f(u)$. Therefore, $f(u)\leq y$, and since $f$ is a partial
Heyting function, hence a partial Esakia function, there is $z\in
\mathrm{dom}(f)$ such that $u\leq z$ and $f(z)=y$. Thus, $x\leq z$
and $g_f(z)=f(z)=y$. Next suppose that $U$ is a clopen upset of
$Y$. Then $x\in g_f^{-1}(U)$ iff $g_f(x)\in U$ iff
$f[{\uparrow}x]\subseteq U$ iff $x\in
({\downarrow}f^{-1}(U^c))^c$. Thus, $g_f^{-1}(U) =
({\downarrow}f^{-1}(U^c))^c$, and so $g_f^{-1}(U)$ is a clopen
upset of $X$. Now suppose that $U$ is clopen in $Y$. Then
$U=\displaystyle{\bigcup_{i=1}^n}(U_i-V_i)$, where $U_i,V_i$ are
clopen upsets of $X$. Therefore,
$g_f^{-1}(U)=g_f^{-1}(\displaystyle{\bigcup_{i=1}^n}(U_i-V_i)) =
\displaystyle{\bigcup_{i=1}^n}(g_f^{-1}(U_i)-g_f^{-1}(V_i))=
\displaystyle{\bigcup_{i=1}^n}(({\downarrow}f^{-1}(U_i^c))^c -
({\downarrow}f^{-1}(V_i^c))^c)$, and so $g_f^{-1}(U)$ is clopen in
$X$. Thus, $g_f$ is continuous, and so $g_f$ is an Esakia
morphism.
\end{proof}

Now let $X$ and $Y$ be Esakia spaces and let $g:X\to Y$ be an Esakia
morphism. We define a partial function $f_g:X\to Y$ as follows. We
let $$\mathrm{dom}(f_g)=\{x\in X: x\in\mathrm{max}g^{-1}g(x)\},$$
and for $x\in \mathrm{dom}(f_g)$ we set $f_g(x)=g(x)$.

\begin{lemma}\label{lem-pEf2}
If $X$ and $Y$ are Esakia spaces and $g:X\to Y$ is an Esakia
morphism, then $f_g:X\to Y$ is a partial Heyting function.
\end{lemma}

\begin{proof}
Let $x,y\in\mathrm{dom}(f_g)$ and let $x<y$. Then
$x\in\mathrm{max}g^{-1}g(x)$ and $y\in \mathrm{max}g^{-1}g(y)$.
From $x<y$ it follows that $x\leq y$, so $g(x)\leq g(y)$. If
$g(x)=g(y)$, then $x,y\in \mathrm{max}g^{-1}g(x)$, which together
with $x<y$ leads to a contradiction. Thus, $g(x)<g(y)$, so
$f_g(x)<f_g(y)$, and so condition (1) of Definition \ref{def-pEf}
is satisfied. Now let $x\in\mathrm{dom}(f_g)$, $y\in Y$, and
$f_g(x)<y$. Then $g(x)<y$, so $g(x)\leq y$, and since $g$ is an
Esakia morphism, there exists $z\in X$ such that $x\leq z$ and
$g(z)=y$. Since $g$ is continuous, $g^{-1}g(z)$ is closed, so
there exists $u\in \mathrm{max}g^{-1}g(z)$ such that $z\leq u$.
Therefore, $x\leq u$, $u\in \mathrm{dom}(f_g)$, and
$f_g(u)=g(u)=y$. If $x=u$, then $g(x)=g(u)=y$, a contradiction.
Thus, $x<u$ and $f_g(u)=y$, and so condition (2) of Definition
\ref{def-pEf} is satisfied. Next let $U$ be a clopen upset of $Y$.
We show that $({\downarrow}f_g^{-1}(U^c))^c=g^{-1}(U)$. We have
$x\in ({\downarrow}f_g^{-1}(U^c))^c$ iff $(\forall
z\in\mathrm{dom}(f_g)) (x\leq z\Rightarrow f_g(z)\in U)$ and $x\in
g^{-1}(U)$ iff $g(x)\in U$. First let $x\in g^{-1}(U)$. Then for
each $z\in\mathrm{dom}(f_g)$ with $x\leq z$ we have $g(x)\leq
g(z)=f_g(z)$. Thus, $f_g(z)\in U$, and so $x\in
({\downarrow}f_g^{-1}(U^c))^c$. Conversely, let $x\notin
g^{-1}(U)$. Then $g(x)\notin U$. Let $z\in\mathrm{max}g^{-1}g(x)$
be such that $x\leq z$. Then $z\in\mathrm{dom}(f_g)$, $x\leq z$,
and $f_g(z)=g(z)=g(x)\notin U$. Thus, $x\notin
({\downarrow}f_g^{-1}(U^c))^c$. Consequently,
$({\downarrow}f_g^{-1}(U^c))^c=g^{-1}(U)$, so
$({\downarrow}f_g^{-1}(U^c))^c$ is clopen in $X$, and so condition
(3) of Definition \ref{def-pEf} is satisfied. Next we show that
$f_g$ satisfies the condition of Definition \ref{def-pHf}. Let
$x\in X$ and let $z\in\mathrm{max}g^{-1}g(x)$ be such that $x\leq
z$. Then $z\in\mathrm{dom}(f_g)$, $x\leq z$, and $g(x)=g(z)$.
Thus, $f_g[{\uparrow}x]={\uparrow}g(x)={\uparrow}g(z)=
{\uparrow}f_g(z)$, and so $f_g$ satisfies the condition of
Definition \ref{def-pHf}. Lastly, let $x\in X$, $y\in Y$, and
$y\notin f_g[{\uparrow}x]$. By the above, there exists
$z\in\mathrm{dom}(f_g)$ such that $x\leq z$ and
$f_g[{\uparrow}x]={\uparrow}f_g(z)$. Therefore, $y\notin
{\uparrow}f_g(z)$. So $f_g(z)\not\leq y$, and by the Priestley
separation axiom, there exists a clopen upset $U$ of $Y$ such that
$f_g(z)\in U$ and $y\notin U$. Thus, $f_g[{\uparrow}x] =
{\uparrow}f_g(z) \subseteq U$ and $y\notin U$. Consequently, $f_g$
satisfies condition (4) of Definition \ref{def-pEf}, so $f_g$ is a
partial Esakia function satisfying the condition of Definition
\ref{def-pHf}, and so $f_g$ is a partial Heyting function.
\end{proof}

\begin{lemma}\label{lem-pEf3}
Let $X$ and $Y$ be Esakia spaces, let $g:X\to Y$ be an Esakia
morphism, and let $f:X\to Y$ be a partial Heyting function. Then
$g_{f_g}=g$ and $f_{g_f}=f$.
\end{lemma}

\begin{proof}
Let $g:X\to Y$ be an Esakia morphism and let $x\in X$. Since $f_g$
is a partial Heyting function, there exists $z\in
\mathrm{dom}(f_g)$ such that $x\leq z$ and $f_g[{\uparrow}x] =
{\uparrow}f_g(z)$. Therefore, $z\in \mathrm{max}g^{-1}g(z)$ and
$f_g(z)=g(z)$. But $f_g[{\uparrow}x]={\uparrow}g(x)$. Thus,
${\uparrow}g(x)={\uparrow}g(z)$, and so $g(x)=g(z)$. It follows
that $g_{f_g}(x)=f_g(z)=g(z)=g(x)$, and so $g_{f_g}=g$. Now let
$f:X\to Y$ be a partial Heyting function. If $x\in
\mathrm{dom}(f_{g_f})$, then $x\in\mathrm{max}g_f^{-1}g_f(x)$ and
$f_{g_f}(x)=g_f(x)$. Since $f$ is a partial Heyting function,
there exists $z\in\mathrm{dom}(f)$ such that $x\leq z$ and
$f[{\uparrow}x] = {\uparrow}f(z)$. But then $g_f(x)=f(z)=g_f(z)$,
so $x=z$ as $x\in\mathrm{max}g_f^{-1}g_f(x)$. Thus,
$x\in\mathrm{dom}(f)$ and $f_{g_f}(x)=f(z)=f(x)$. If $x\in
\mathrm{dom}(f)$, then $f(x)=g_f(x)$. So $x\in\mathrm{max}g_f^{-1}
g_f(x)$ and $f_{g_f}(x)=g_f(x)=f(x)$. Thus, $\mathrm{dom}(f_{g_f})
= \mathrm{dom}(f)$ and for each $x\in \mathrm{dom}(f_{g_f}) =
\mathrm{dom}(f)$ we have $f_{g_f}(x)=f(x)$. Consequently,
$f_{g_f}=f$.
\end{proof}

As an immediate consequence of Lemmas \ref{lem-pEf1},
\ref{lem-pEf2}, and \ref{lem-pEf3}, we obtain a direct proof of the
fact that $\mathsf{ES}$ is isomorphic to $\mathsf{ES^{H}}$. For
the reader's convenience we give a table that gathers together the
dual equivalences of different categories that we obtained in the
last two sections. For two categories $C$ and $D$, we use $C \dueq
D$ to denote that $C$ is dually equivalent to $D$, and $C \cong D$
to denote that $C$ is isomorphic to $D$.

\newpage

\

\begin{center}

{\large\bf Categories of algebras:}

\

\

\begin{tabular}{p{5em}|p{10em}|p{15em}}
Category & Objects & Morphisms\\
\hline \hline
 {\sf BDM} & Bounded distributive meet semi-lattices &
Top-preserving
meet semi-lattice homomorphisms\\
\hline {\sf BDM$^\bot$} & ``\_\_\_\_\_\_\_\_'' & Bounded meet
semi-lattice
homomorphisms\\
\hline {\sf BDM$^{\mathsf{S}}$} & ``\_\_\_\_\_\_\_\_'' & Sup-homomorphisms\\
\hline \vspace{.01mm} {\sf BDL$^{\land,\top}$} & Bounded distributive lattices &
Top-preserving meet semi-lattice homomorphisms\\
\hline \vspace{.01mm} {\sf BDL$^{\land,\top,\bot}$} & ``\_\_\_\_\_\_\_\_'' &
Bounded meet
semi-lattice homomorphisms\\
\hline
{\sf BDL} & ``\_\_\_\_\_\_\_\_'' & Bounded lattice homomorphisms\\
\hline {\sf BIM} & Bounded implicative meet semi-lattices &
Implicative
meet semi-lattice homomorphisms\\
\hline {\sf BIM$^\bot$} & ``\_\_\_\_\_\_\_\_'' & Bounded implicative
meet
semi-lattice homomorphisms\\
\hline {\sf BIM$^{\mathsf{S}}$} & ``\_\_\_\_\_\_\_\_'' & Implicative
meet
semi-lattice sup-homomorphisms\\
\hline {\sf HA$^{\land,\to}$} & Heyting algebras & Implicative meet
semi-lattice homomorphisms\\
\hline {\sf HA$^{\land,\to,\bot}$} & ``\_\_\_\_\_\_\_\_'' & Bounded
implicative meet
semi-lattice homomorphisms\\
\hline {\sf HA} & ``\_\_\_\_\_\_\_\_'' & Heyting algebra
homomorphisms
\end{tabular}

\end{center}

\newpage

\

\begin{center}

{\large\bf Categories of spaces:}

\

\

\begin{tabular}{p{5em}|p{10em}|p{15em}}
Category & Objects & Morphisms\\
\hline \hline%
{\sf GPS} & Generalized Priestley spaces & Generalized Priestley
morphisms\\
\hline {\sf GPS$^{\mathsf T}$} & ``\_\_\_\_\_\_\_\_'' & Total
generalized Priestley morphisms\\
\hline {\sf GPS$^{\mathsf F}$} & ``\_\_\_\_\_\_\_\_'' & Functional
generalized Priestley morphisms\\
\hline {\sf GPS$^{\mathsf{S}}$} & ``\_\_\_\_\_\_\_\_'' & strong Priestley
 morphisms\\
\hline {\sf PS$^{\mathsf R}$} & Priestley spaces & Generalized
Priestley morphisms\\
\hline {\sf PS$^{\mathsf{T}}$} & ``\_\_\_\_\_\_\_\_'' & Total
generalized Priestley morphisms\\
\hline {\sf PS$^{\mathsf{F}}$} & ``\_\_\_\_\_\_\_\_'' & Functional
generalized Priestley morphisms\\
\hline
{\sf PS} & ``\_\_\_\_\_\_\_\_'' & Priestley morphisms\\
\hline {\sf GES} & Generalized Esakia spaces & Generalized Esakia
morphisms\\
\hline {\sf GES$^{\mathsf T}$} & ``\_\_\_\_\_\_\_\_'' & Total
generalized Esakia morphisms\\
\hline {\sf GES$^{\mathsf F}$} & ``\_\_\_\_\_\_\_\_'' & Functional
generalized Esakia morphisms\\
\hline {\sf GES$^{\mathsf{S}}$} & ``\_\_\_\_\_\_\_\_'' & strong
 Esakia morphisms\\
\hline {\sf ES$^{\mathsf R}$} & Esakia spaces & Generalized Esakia
morphisms\\
\hline {\sf ES$^{\mathsf{T}}$} & ``\_\_\_\_\_\_\_\_'' & Total
generalized Esakia morphisms\\
\hline {\sf ES$^{\mathsf{F}}$} & ``\_\_\_\_\_\_\_\_'' & Functional
generalized Esakia morphisms\\
\hline {\sf ES$^{\mathsf{P}}$} & ``\_\_\_\_\_\_\_\_'' & Partial
Esakia functions\\
\hline {\sf ES$^{\mathsf{W}}$} & ``\_\_\_\_\_\_\_\_'' & Well
partial Esakia functions\\
\hline {\sf ES$^{\mathsf{H}}$} & ``\_\_\_\_\_\_\_\_'' & Partial
Heyting functions\\
\hline {\sf ES} & ``\_\_\_\_\_\_\_\_'' & Esakia morphisms
\end{tabular}

\end{center}

\newpage

\

\begin{center}

{\large\bf Dualities:}

\

\

\begin{tabular}{rclclll}
{\sf BDM}&$\dueq$&{\sf GPS}\\
{\sf BDM}$^\bot$&$\dueq$&{\sf GPS}$^{\mathsf T}$\\
{\sf BDM}$^{\mathsf{S}}$&$\dueq$&{\sf GPS}$^{\mathsf F}$&$\cong$&{\sf GPS}$^{\mathsf S}$\\
{\sf BDL}$^{\land,\top}$&$\dueq$&{\sf PS}$^{\mathsf R}$\\
{\sf BDL}$^{\land,\top,\bot}$&$\dueq$&{\sf PS}$^{\mathsf{T}}$\\
{\sf BDL}&$\dueq$&{\sf PS}$^{\mathsf{F}}$&$\cong$&{\sf PS}\\
\\
{\sf BIM}&$\dueq$&{\sf GES}\\
{\sf BIM}$^\bot$&$\dueq$&{\sf GES}$^{\mathsf T}$\\
{\sf BIM}$^{\mathsf{S}}$&$\dueq$&{\sf GES}$^{\mathsf F}$&$\cong$&{\sf GES}$^{\mathsf S}$\\
{\sf HA}$^{\land,\to}$&$\dueq$&{\sf ES}$^{\mathsf R}$&$\cong$&{\sf ES}$^{\mathsf{P}}$\\
{\sf HA}$^{\land,\to,\bot}$&$\dueq$&{\sf ES}$^{\mathsf{T}}$&$\cong$&{\sf ES}$^{\mathsf{W}}$\\
{\sf HA}&$\dueq$&{\sf ES}$^{\mathsf{F}}$&$\cong$&{\sf ES}$^{\mathsf{H}}$&$\cong$&{\sf ES}
\end{tabular}

\end{center}

\

\

We conclude this section by giving two counterexamples. The first
one shows that it is impossible to characterize generalized
Priestley morphisms between Priestley spaces in terms of partial
functions, and the second one shows that it is impossible to
characterize generalized Esakia morphisms between generalized Esakia
spaces in terms of partial functions. Thus, there is no
generalization of Corollary \ref{cor-a1} to neither Priestley spaces
nor generalized Esakia spaces.

\begin{example}

Consider the Priestley spaces $X$ and $Y$ and the generalized
Priestley morphism $R\subseteq X\times Y$ shown in Fig.5. The
corresponding top preserving meet semi-lattice homomorphism
$h_R:\mathfrak{CU}(Y)\to\mathfrak{CU}(X)$ is also shown in Fig.5.
There are only three partial functions from $X$ to $Y$: the empty
function, the total function sending $x$ to $y$, and the total
function sending $x$ to $z$. It is easy to see that their
corresponding meet semi-lattice homomorphisms from
$\mathfrak{CU}(Y)$ to $\mathfrak{CU}(X)$ are different from $h_R$.
Thus, it is impossible to characterize $R\subseteq X\times Y$ in
terms of partial functions from $X$ to $Y$.


\begin{center}

\

\begin{tikzpicture}[scale=.5,inner sep=.5mm]
\node[bull] (y) at (-3,0) [label=left:$y$] {};%
\node[bull] (z) at (-2,0) [label=right:$z$] {};%
\node[bull] (x) at (1,0) [label=right:$x$] {};%
\node at (1,-3) {$X$};%
\node at (-2.5,-3) {$Y$};%
\draw[color=brown,->] (x) to [out=-135,in=-45] (y);
\draw[color=brown,->] (x) to [out=-135,in=-45] (z);%
\node[bull] (sy) at (6,0) [label=left:$\{y\}\ $] {};%
\node[bull] (sz) at (8,0) [label=right:$\ \{z\}$] {};%
\node[bull] (Y) at (7,1) [label=left:$Y\ $] {};%
\node[bull] (eY) at (7,-1) [label=left:$\emptyset\ \ $] {};%
\node[bull] (X) at (11,1) [label=right:$\ X$] {};%
\node[bull] (eX) at (11,-1) [label=right:$\ \emptyset$] {};%
\draw (Y) -- (sy) -- (eY) -- (sz) -- (Y);%
\draw (X) -- (eX);
\node at (7,-3) {$\mathfrak{CU}(Y)$};%
\node at (11,-3) {$\mathfrak{CU}(X)$};
\draw[color=brown,rounded corners=1.5ex] (5.5,.25) -- (8.5,.25) -- (7,-1.5) -- cycle;%
\draw[color=brown,->] (Y) -- (X);%
\draw[color=brown,->] (7.5,-1) -- (eX);
\end{tikzpicture}

\

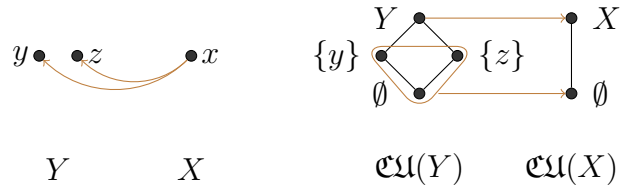
Fig.5\ \ \ \ \ \ \ \
\end{center}

\end{example}

\begin{example}

Consider the generalized Esakia spaces $X$ and $Y$ and the
generalized Esakia morphism $R\subseteq X\times Y$ shown in Fig.6,
where $Y_0=Y-\{z_1,z_2\}$, the elements of $Y_0$ are isolated points
of $Y$, $z_1$ is the limit point of $\{x_1,x_2,\ldots\}$, $z_2$ is
the limit point of $\{y_1,y_2, \ldots\}$,
$R[r]={\uparrow}z_1\cup{\uparrow}z_2$, and $R[w_i]=\{u_i\}$ for
$i=1,2,3$. Then $\mathrm{dom}(f_R) =\{w_1,w_2,w_3\}$ and
$R_{f_R}[r]=\{u_1,u_2,u_3\}$. Thus, $R[r]\neq R_{f_R}[r]$, so $R\neq
R_{f_R}$, and so Corollary \ref{cor-a1} does not extend to
generalized Esakia spaces.

\begin{center}

$$
\begin{array}{ccc}
\begin{tikzpicture}[scale=.5,inner sep=.5mm]
\node[bull] (x1) at (-1,0) [label=left:$x_1$] {};%
\node[bull] (x2) at (-1,1) [label=left:$x_2$] {};%
\node[bull] (x3) at (-1,2) [label=left:$x_3$] {};%
\node (empx) at (-1,3) {};
\node (dots) at (-1,4) {$\vdots$};%
\node[holl] (z1) at (-1,5) [label=left:$z_1$] {};
\node[bull] (y1) at (1,0) [label=right:$y_1$] {};%
\node[bull] (y2) at (1,1) [label=right:$y_2$] {};%
\node[bull] (y3) at (1,2) [label=right:$y_3$] {};%
\node (empy) at (1,3) {};
\node (dots) at (1,4) {$\vdots$};%
\node[holl] (z2) at (1,5) [label=right:$z_2$] {};
\node[bull] (u1) at (-2,7) [label=right:$u_1$] {};%
\node[bull] (u2) at (0,7) [label=right:$u_2$] {};%
\node[bull] (u3) at (2,7) [label=right:$u_3$] {};%
\node[bull] (w1) at (-8,7) [label=left:$w_1$] {};%
\node[bull] (w2) at (-6,7) [label=left:$w_2$] {};%
\node[bull] (w3) at (-4,7) [label=left:$w_3$] {};%
\node[bull] (r) at (-6,5) [label=below:$r$] {};%
\node (R) at (-3,9) {$R$};%
\node (X) at (-6,-3) {$X$};%
\node (Y) at (0,-3) {$Y$};%
\draw (x1) -- (x2) -- (x3) -- (empx);%
\draw (y1) -- (y2) -- (y3) -- (empy);%
\draw (u1) -- (z1) -- (u2) -- (z2) -- (u3);%
\draw (w1) -- (r) -- (w3);%
\draw (w2) -- (r);%
\draw [brown,->] (w1.north east).. controls +(1,1) and +(-1,1) .. (u1.north west);%
\draw [brown,->] (w2.north east).. controls +(1,1) and +(-1,1) .. (u2.north west);%
\draw [brown,->] (w3.north east).. controls +(1,1) and +(-1,1) .. (u3.north west);%
\draw [brown,->] (r.south east).. controls +(1,-.7) and +(-1,-.7) .. (z1.south west);%
\draw [brown,->] (r.south east).. controls +(1,-1) and +(-1,-1) .. (z2.south west);%
\end{tikzpicture}&\ &
\begin{tikzpicture}[scale=.5,inner sep=.5mm]
\node[bull] (x1) at (-1,0) [label=left:$x_1$] {};%
\node[bull] (x2) at (-1,1) [label=left:$x_2$] {};%
\node[bull] (x3) at (-1,2) [label=left:$x_3$] {};%
\node (empx) at (-1,3) {};
\node (dots) at (-1,4) {$\vdots$};%
\node[holl] (z1) at (-1,5) [label=left:$z_1$] {};
\node[bull] (y1) at (1,0) [label=right:$y_1$] {};%
\node[bull] (y2) at (1,1) [label=right:$y_2$] {};%
\node[bull] (y3) at (1,2) [label=right:$y_3$] {};%
\node (empy) at (1,3) {};
\node (dots) at (1,4) {$\vdots$};%
\node[holl] (z2) at (1,5) [label=right:$z_2$] {};
\node[bull] (u1) at (-2,7) [label=right:$u_1$] {};%
\node[bull] (u2) at (0,7) [label=right:$u_2$] {};%
\node[bull] (u3) at (2,7) [label=right:$u_3$] {};%
\node[bull] (w1) at (-8,7) [label=left:$w_1$] {};%
\node[bull] (w2) at (-6,7) [label=left:$w_2$] {};%
\node[bull] (w3) at (-4,7) [label=left:$w_3$] {};%
\node[bull] (r) at (-6,5) [label=below:$r$] {};%
\node (R) at (-3,9) {$f_R$};%
\node (X) at (-6,-3) {$X$};%
\node (Y) at (0,-3) {$Y$};%
\draw (x1) -- (x2) -- (x3) -- (empx);%
\draw (y1) -- (y2) -- (y3) -- (empy);%
\draw (u1) -- (z1) -- (u2) -- (z2) -- (u3);%
\draw (w1) -- (r) -- (w3);%
\draw (w2) -- (r);%
\draw [brown,->] (w1.north east).. controls +(1,1) and +(-1,1) .. (u1.north west);%
\draw [brown,->] (w2.north east).. controls +(1,1) and +(-1,1) .. (u2.north west);%
\draw [brown,->] (w3.north east).. controls +(1,1) and +(-1,1) .. (u3.north west);%
\end{tikzpicture}
\end{array}
$$

\

Fig.6
\end{center}

\end{example}

\section{Duality at work}

In this section we show how the duality developed in the previous
sections works by establishing dual descriptions of a number of
algebraic concepts that play an important role in the theory of
distributive meet semi-lattices and implicative meet semi-lattices.

\subsection{Dual description of Frink ideals, ideals, and filters}

We start by recalling that for a bounded distributive lattice
(resp.\ Heyting algebra) $L$ and its dual Priestley space (resp.\
Esakia space) $X$, there is a lattice isomorphism between the
lattice of ideals of $L$ and the lattice of open upsets of $X$,
and the lattice of filters of $L$ (ordered by $\supseteq$) and the
lattice of closed upsets of $X$.
These isomorphisms are obtained as follows. If $I$ is an ideal of
$L$, then $U(I)=\displaystyle{\bigcup}\{\phi(a):a\in I\}$ is the
open upset of $X$ corresponding to $I$, and if $U$ is an open upset
of $X$, then $I(U)=\{a\in L:\phi(a)\subseteq U\}$ is the ideal of
$L$ corresponding to $U$; if $F$ is a filter of $L$, then
$C(F)=\displaystyle{\bigcap}\{\phi(a):a\in F\}$ is the closed upset
of $X$ corresponding to $F$, and if $C$ is a closed upset of $X$,
then $F(C)=\{a\in L:C\subseteq\phi(a)\}$ is the filter of $L$
corresponding $C$. Then we have $I\subseteq J$ iff $U(I)\subseteq
U(J)$, $I=I(U(I))$, and $U(I(U))=U$; and $F\supseteq G$ iff
$C(F)\subseteq C(G)$, $F=F(C(F))$, and $C(F(C))=C$. Now we show how
these correspondences work for Frink ideals, ideals, and filters of
bounded distributive meet semi-lattices and bounded implicative meet
semi-lattices.

Let $L$ be a bounded distributive meet semi-lattice and let $D(L)$
be its distributive envelope. Let also $X=\la X,\tau,\leq,X_0\ra$ be
the generalized Priestley space of $L$. We know that $\la
X,\tau,\leq\ra$ is order-isomorphic and homeomorphic to the
Priestley space of $D(L)$. Since the lattice of Frink ideals of $L$
is isomorphic to the lattice of ideals of $D(L)$, we immediately
obtain from the above:

\begin{proposition}\label{prop9.1}
Let $L$ be a bounded distributive meet semi-lattice and let $X$ be
its generalized Priestley space. Then the maps $I\mapsto U(I)$ and
$U\mapsto I(U)$ set an isomorphism of the lattice of Frink ideals
of $L$ with the lattice of open upsets of $X$.
\end{proposition}

In particular, prime F-ideals of $L$ correspond to the open upsets
of $X$ of the form $({\downarrow}x)^c$ for $x\in X$. Now we give a
dual description of ideals of $L$. Since each ideal of $L$ is an
F-ideal, ideals correspond to special open upsets of $X$.

\begin{lemma}\label{lem9.2}
Let $L$ be a bounded distributive meet semi-lattice and let $X$ be
its generalized Priestley space. If $I$ is an ideal of $L$, then $X
- U(I) = {\downarrow}(X_0 - U(I))$.
\end{lemma}

\begin{proof}
The inclusion ${\downarrow}(X_0 - U(I)) \subseteq X - U(I)$ is
trivial. To prove the other inclusion, let $x \in X - U(I)$. Then $x
\cap I = \emptyset$. By the prime filter lemma, there is a prime
filter $y$ of $L$ such that $x \subseteq y$ and $y \cap I =
\emptyset$. Thus, $y \in X_0 - U(I)$, and so $x \in {\downarrow}(X_0
- U(I))$.
\end{proof}

\begin{lemma}\label{lem9.3}
Let $L$ be a bounded distributive meet semi-lattice and let $X$ be
its generalized Priestley space. If $U$ is an open upset of $X$ such
that $X - U = {\downarrow}(X_0 - U)$, then $I(U)$ is an ideal of
$L$.
\end{lemma}

\begin{proof}
Since $U$ is an open upset of $X$, it follows from Proposition
\ref{prop9.1} that $I(U)$ is an F-ideal of $L$. Let $a, b \in I(U)$
with ${\uparrow}a\cap{\uparrow}b \cap I(U) = \emptyset$. By the
optimal filter lemma, there exists $x\in X$ such that
${\uparrow}a\cap{\uparrow}b\subseteq x$ and $x\cap I(U)=\emptyset$.
Therefore, $x\notin U$, so $x\in X-U$, and so there exists $y\in X_0
- U$ such that $x\leq y$. It follows that ${\uparrow}a \cap
{\uparrow} b \subseteq y$, and as $y$ is a prime filter, we have
${\uparrow}a \subseteq y$ or ${\uparrow}b \subseteq y$. Thus, $a\in
y$ or $b\in y$, which is a contradiction because $a,b\in I(U)$.
Consequently, ${\uparrow}a\cap{\uparrow}b \cap I(U) \neq \emptyset$,
and so $I(U)$ is an ideal of $L$.
\end{proof}

Putting Proposition \ref{prop9.1} and Lemmas \ref{lem9.2} and
\ref{lem9.3} together, we obtain:

\begin{theorem}\label{prop9.4}
Let $L$ be a bounded distributive meet semi-lattice and let $X$ be
its generalized Priestley space. Then the maps $I\mapsto U(I)$ and
$U\mapsto I(U)$ set an isomorphism of the ordered set of ideals of
$L$ with the ordered set of open upsets $U$ of $X$ such that $X - U
= {\downarrow}(X_0 - U)$.
\end{theorem}

\begin{remark}
Since for an open upset $U$ of $X$, $X-U$ is closed in $X$, and so
for each $x\in X-U$ there is $y\in\mathrm{max}(X-U)$ with $x\leq y$,
the condition $X - U = {\downarrow}(X_0 - U)$ is obviously
equivalent to the condition $\mathrm{max}(X-U)\subseteq X_0$.
\end{remark}

Our next task is to give a dual description of prime ideals of $L$.

\begin{lemma}\label{lem9.6}
Let $L$ be a bounded distributive meet semi-lattice and let $X$ be
its generalized Priestley space. If $I$ is a prime ideal of $L$,
then $U(I)=({\downarrow}x)^c$ for some $x\in X_0$.
\end{lemma}

\begin{proof}
Let $I$ be a prime ideal of $L$ and let $x= L - I$. By Proposition
\ref{filterideal}, $x\in X_0$. Moreover, we have $y\in U(I)$ iff
$y\cap I\neq\emptyset$ iff $y\not\subseteq x$ iff $y\in
({\downarrow}x)^c$. Thus, $U(I)=({\downarrow}x)^c$.
\end{proof}

\begin{lemma}\label{lem9.7}
Let $L$ be a bounded distributive meet semi-lattice and let $X$ be
its generalized Priestley space. If $U=({\downarrow}x)^c$ for $x\in
X_0$, then $I(U)$ is a prime ideal of $L$.
\end{lemma}

\begin{proof}
Clearly $I(U)$ is an ideal of $L$ because
$\mathrm{max}(U^c)=\mathrm{max}({\downarrow}x)=\{x\}\subseteq X_0$.
We show that it is prime. Let $a\wedge b\in I(U)$. Then $\phi(a)\cap
\phi(b)=\phi(a\wedge b)\subseteq U$. So $\phi(a)\cap
\phi(b)\subseteq({\downarrow}x)^c$, and so ${\downarrow}x\subseteq
\phi(a)^c\cup\phi(b)^c$. Therefore, $x\in\phi(a)^c\cup\phi(b)^c$,
which implies that $x\in\phi(a)^c$ or $x\in\phi(b)^c$. Thus,
${\downarrow}x\subseteq \phi(a)^c$ or ${\downarrow}x\subseteq
\phi(b)^c$, so $\phi(a)\subseteq({\downarrow}x)^c$ or $\phi(b)
\subseteq({\downarrow}x)^c$. It follows that $\phi(a)\subseteq U$ or
$\phi(b)\subseteq U$, so $a\in I(U)$ or $b\in I(U)$, and so $I(U)$
is a prime ideal.
\end{proof}

Putting Theorem \ref{prop9.4} and Lemmas \ref{lem9.6} and
\ref{lem9.7} together, we obtain:

\begin{proposition}
Let $L$ be a bounded distributive meet semi-lattice and let $X$ be
its generalized Priestley space. Then the maps $I\mapsto U(I)$ and
$U\mapsto I(U)$ set an isomorphism of the ordered set of prime
ideals of $L$ with the ordered set of open upsets of $X$ of the form
$({\downarrow}x)^c$ for some $x\in X_0$.
\end{proposition}

Now we give a dual description of filters of $L$. Since there are
less filters in $L$ than in $D(L)$, not every closed upset of $X$
corresponds to a filter of $L$.

\begin{lemma}\label{lem9.9}
Let $L$ be a bounded distributive meet semi-lattice and let $X$ be
its generalized Priestley space. If $F$ is a filter of $L$, then $X
- C(F) = {\downarrow}(X_0 - C(F))$.
\end{lemma}

\begin{proof}
The inclusion ${\downarrow}(X_0 - C(F)) \subseteq X - C(F)$ is
trivial. For the other inclusion, let $x \in X - C(F)$. Then
$x\notin C(F)$, and so there exists $a \in F$ such that $a \not \in
x$. By the prime filter lemma, there is $y \in X_0$ such that $x
\subseteq y$ and $a \not \in y$. Thus, $y \not \in C(F)$, so $y \in
X_0 - C(F)$, and so $x \in {\downarrow}(X_0 - C(F))$.
\end{proof}

\begin{lemma}\label{lem9.10}
Let $L$ be a bounded distributive meet semi-lattice and let $X$ be
its generalized Priestley space. If $C$ is a closed upset of $X$,
then $F(C)$ is a filter of $L$, and $C=C(F(C))$ iff $X - C =
{\downarrow}(X_0 - C)$.
\end{lemma}

\begin{proof}
If $C$ is a closed upset, it is routine to see that $F(C)$ is a
filter of $L$. Suppose that $C=C(F(C))$. Since $F(C)$ is a filter of
$L$, Lemma \ref{lem9.9} implies that $X - C(F(C)) = {\downarrow}(X_0
- C(F(C)))$. From $C=C(F(C))$ and the last equality we get $X - C =
{\downarrow}(X_0 - C)$. Conversely, suppose that $X - C =
{\downarrow}(X_0 - C)$. We show that $C=C(F(C))$. Since
$C(F(C))=\displaystyle{\bigcap}\{\phi(a):C \subseteq\phi(a)\}$, it
is obvious that $C\subseteq C(F(C))$. For the converse, suppose that
$x\notin C$. Then there exists $y \in X_0-C$ such that $x \leq y$.
Since $C$ is a closed upset of $X$, $C$ is the intersection of
clopen upsets of $X$ containing $C$. Therefore, from $y\notin C$ it
follows that there is a clopen upset $U$ of $X$ such that
$C\subseteq U$ and $y\notin U$. As each clopen upset of $X$ is a
finite union of elements of $\phi[L]$, there exist $a_1, \ldots, a_n
\in L$ such that $U = \phi(a_{1}) \cup \ldots \cup \phi(a_{n})$.
Thus, $y \not \in \phi(a_{1}) \cup \ldots \cup \phi(a_{n})$, and so
$a_1, \ldots, a_n \not \in y$. Since $y$ is a prime filter of $L$,
we have $\displaystyle{\bigcap_{i=1}^n}{\uparrow}a_i \not \subseteq
y$. Therefore, there exists $a \in \displaystyle{\bigcap_{i=1}^n}
{\uparrow}a_i$ such that $a \not \in y$. But then
$\phi(a)\supseteq\phi(a_1)\cup\ldots\cup\phi(a_n) = U \supseteq C$
and $y\notin \phi(a)$. Consequently, $C\subseteq \phi(a)$ and
$x\notin\phi(a)$, implying that $x \not \in C(F(C))$.
\end{proof}

Putting Lemmas \ref{lem9.9} and \ref{lem9.10} together, we obtain:

\begin{theorem}
Let $L$ be a bounded distributive meet semi-lattice and let $X$ be
its generalized Priestley space. Then the maps $F\mapsto C(F)$ and
$C\mapsto F(C)$ set an isomorphism of the lattice of filters of $L$
(ordered by reverse inclusion) and the lattice of closed upsets $C$
of $X$ satisfying the condition $X - C = {\downarrow}(X_0 - C)$.
\end{theorem}

In particular, since there is a 1-1 correspondence between prime
filters and prime ideals of $L$, we obtain that prime filters of $L$
correspond to closed upsets of $X$ of the form ${\uparrow}x$ for
$x\in X_0$. Also, since there is a 1-1 correspondence between
optimal filters and prime F-ideals of $L$, optimal filters of $L$
correspond to closed upsets of $X$ of the form ${\uparrow}x$ for
$x\in X$. In the next table we gather together the dual descriptions
of different notions of filters and ideals of a bounded distributive
meet semi-lattice $L$ and its distributive envelope $D(L)$.

\

\begin{center}

{\large\bf Dual description of filters of $L$ and $D(L)$}

\

\begin{tabular}{p{15em}|p{15em}}
Filters of $D(L)$ &Closed upsets of $L_*$\\
\hline%
Filters of $L$& Closed upsets $C$ of $L_*$ such that
$L_*-C={\downarrow}(L_+-C)$\\
\hline%
Optimal filters of $L$ = Prime filters of $D(L)$&${\uparrow} x$,
$x\in L_*$\\
\hline%
Prime filters of $L$&
${\uparrow} x$, $x\in L_+$
\end{tabular}


\

\

{\large\bf Dual description of ideals of $L$ and $D(L)$}

\

\begin{tabular}{p{15em}|p{15em}}
F-ideals of $L$ = Ideals of $D(L)$&Open upsets of $L_*$\\
\hline%
Prime F-ideals of $L$ = Prime ideals of $D(L)$&$({\downarrow}x)^c$,
$x\in L_*$\\
\hline%
Ideals of $L$&Open upsets $U$ of
$L_*$ such that $L_*-U={\downarrow}(L_+-U)$\\
\hline%
Prime ideals of $L$&$({\downarrow} x)^c$, $x\in L_+$
\end{tabular}
\end{center}

\

\

The picture remains unchanged for bounded implicative meet
semi-lattices.

\subsection{Dual description of 1-1 and onto homomorphisms}

Our next task is to give a dual description of 1-1 and onto
homomorphisms.

\begin{lemma}\label{theorem1}
Let $X$ and $Y$ be generalized Priestley spaces and let $R\subseteq
X\times Y$ be a generalized Priestley morphism.
\begin{enumerate}
\item If $F$ is a closed subset of $X$, then $R[F]$ is a closed
upset of $Y$.
\item If $G$ is a closed subset of $Y$, then $R^{-1}[G]$ is a
closed downset of $X$.
\end{enumerate}
\end{lemma}

\begin{proof}
(1) Suppose that $F$ is a closed subset of $X$. It follows from
condition (2) of Definition \ref{definition} that $R[F]$ is an
upset of $Y$. We show that $R[F]$ is closed in $Y$. Let $y \not
\in R[F]$. Then for each $x \in F$ we have $x \nR y$. By condition
(3) of Definition \ref{definition}, there is $U_x \in Y^{*}$ such
that $R[x] \subseteq U_x$ and $y \notin U_x$. Thus, $x \in \Box_R
U_x$ and by condition (4) of Definition \ref{definition}, $\Box_R
U_x \in X^{*}$, so $\Box_R U_x$ is clopen. Then we have
$F\subseteq \displaystyle{\bigcup} \{\Box_R U_x: x\in F\}$. Since
$F$ is a closed subset of a compact space, $F$ is compact.
Therefore, there are $x_1, \ldots, x_n \in F$ such that $F
\subseteq \displaystyle{\bigcup_{i=1}^n} \Box_R U_{x_i}$. We claim
that $U_{x_1}^{c} \cap \ldots \cap U_{x_n}^{c} \cap R[F] =
\emptyset$. If not, then there exists $z \in U_{x_1}^{c} \cap
\ldots \cap U_{x_n}^{c} \cap R[F]$. Thus, there is $u \in F$ such
that $u R z$. But then $u \in \Box_R U_{x_i}$, so $z \in U_{x_i}$
for some $i \leq n$, which is a contradiction. It follows that
there is an open neighborhood $U_{x_1}^c \cap \ldots \cap
U_{x_n}^c$ of $y$ missing $R[F]$, so $R[F]$ is closed in $Y$.

(2) Suppose that $G$ is a closed subset of $Y$. It follows from
condition (1) of Definition \ref{definition} that $R^{-1}[G]$ is a
downset of $X$. We show that $R^{-1}[G]$ is closed in $X$. Let $x
\notin R^{-1}[G]$. Then for each $y \in G$ we have $x \nR y$. So,
by condition (3) of Definition \ref{definition}, there is $U_y \in
Y^{*}$ such that $x \in \Box_{R} U_y$ and $y \notin U_y$.
Therefore, $G\subseteq \displaystyle{\bigcup} \{U_y^c: y\in G\}$,
and as $G$ is compact, there are $y_1, \ldots, y_n \in G$ such
that $G \subseteq \displaystyle{\bigcup_{i=1}^n} U_{y_i}^{c}$. We
claim that $\Box_{R} U_{y_1} \cap \ldots \cap \Box_{R} U_{y_n}
\cap R^{-1}[G]= \emptyset$. If not, then there is $z \in \Box_{R}
U_{y_1} \cap \ldots \cap \Box_{R} U_{y_n} \cap R^{-1}[G]$. So
$R[z] \subseteq U_{y_1} \cap \ldots \cap U_{y_n}$ and $z \in
R^{-1}[G]$. Thus, there is $u \in G$ such that $z R u$. But then
$u \in U_{y_1} \cap \ldots \cap U_{y_n} \cap G$, which is a
contradiction. Consequently, there is an open neighborhood
$\Box_{R} U_{y_1} \cap \ldots \cap \Box_{R} U_{y_n}$ of $x$
missing $R^{-1}[G]$, so $R^{-1}[G]$ is closed in $X$.
\end{proof}

\begin{definition}\label{definition1}
{\rm Let $X$ and $Y$ be generalized Priestley spaces and let
$R\subseteq X\times Y$ be a generalized Priestley morphism.
\begin{enumerate}
\item We call $R$ \textit{onto} if for each $y \in Y$ there is $x
\in X$ such that $R[x] = {\uparrow} y$.
\item We call $R$ \textit{1-1} if for each $x \in X$ and
$U\in X^{*}$ with $x \notin U$, there is $V \in Y^{*}$ such that
$R[U] \subseteq V$ and $R[x] \not \subseteq V$.
\end{enumerate}
}
\end{definition}

Let $X$ and $Y$ be generalized Priestley spaces and let $R\subseteq
X\times Y$ be a generalized Priestley morphism. We observe that if
$R$ is 1-1, then $R$ is total. Indeed, if $R$ is not total, then
there exists $x\in X$ such that $R[x]=\emptyset$. Therefore, for
each $V\in Y^*$ we have $R[x]\subseteq V$. Thus, $R$ can not be 1-1.
We also observe that using condition (5) of Definition
\ref{g-Priestley}, it is easy to verify that if a generalized
Priestley morphism $R$ is 1-1, then $x \not \leq y$ implies $R[y]
\not \subseteq R[x]$, and $x \notin U$ implies $R[x] \not \subseteq
R[U]$ for each $x,y\in X$ and $U\in X^*$. However, these two
conditions do not imply that $R$ is 1-1 as the following example
shows.

\begin{example}

Let $X$ and $Y$ be the generalized Esakia spaces and $R\subseteq
X\times Y$ be the generalized Esakia morphism shown in Fig.7. We
have that $Y_0=Y-\{x_1\}$, the elements of $Y_0$ are isolated
points of $Y$, $x_1$ is the only limit point of $Y$,
$R[x]={\uparrow}x_1$, $R[y]=\{y_1\}$, and $R[z]=\{z_1\}$. Then it
is easy to verify that for each $x,y\in X$ and $U\in X^*$, we have
$x \not \leq y$ implies $R[y] \not \subseteq R[x]$, and $x \notin
U$ implies $R[x] \not \subseteq R[U]$. On the other hand, for
$U=\{y,z\}\in X^*$ we have $x\notin U$, $R[x] = {\uparrow}x_1$,
and $R[U]=\{y_1,z_1\}$. Now $R[U]\notin Y^*$ and for each $V\in
Y^*$ with $R[U]\subseteq V$, we have $x_1\in V$, and so
$R[x]\subseteq V$. Thus, there is no $V \in Y^{*}$ such that $R[U]
\subseteq V$ and $R[x] \not \subseteq V$, and so $R$ is not 1-1.
The bounded implicative meet semi-lattices $X^*$ and $Y^*$
corresponding to $X$ and $Y$ together with the bounded implicative
meet semi-lattice homomorphism $h_R:Y^*\to X^*$ corresponding to
$R$ are shown in Fig.7. Clearly $h_R$ is not onto as
$h_R^{-1}(\{y,z\})=\emptyset$.
\end{example}


\

\begin{center}

\

\begin{tikzpicture}[scale=.5,inner sep=.5mm,label distance=.5mm]
\node[bull] (y) at (-7,7) [label=left:$y$] {};%
\node[bull] (z) at (-5,7) [label=left:$z$] {};%
\node[bull] (x) at (-6,6) [label=below left:$x$] {};%
\node[bull] (u1) at (0,0)  [label=right:$u_1$] {};%
\node[bull] (u2) at (0,1) [label=right:$u_2$] {};%
\node[bull] (u3) at (0,2) [label=right:$u_3$] {};%
\node (empdown) at (0,3) {};
\node (dots) at (0,4) {$\vdots$};%
\node[holl] (x1) at (0,6) [label=below right:$x_1$] {};
\node[bull] (y1) at (-1,7) [label=right:$y_1$] {};%
\node[bull] (z1) at (1,7) [label=right:$z_1$] {};%
\node at (-6,-3) {$X$};%
\node at (0,-3) {$Y$};%
\draw (u1) -- (u2) -- (u3) -- (empdown);%
\draw (y1) -- (x1) -- (z1);%
\draw (y) -- (x) -- (z);%
\draw[color=brown,->] (y) to [out=45,in=135] (y1);
\draw[color=brown,->] (z) to [out=45,in=135] (z1);
\draw[color=brown,->] (x) to [out=-45,in=-135] (x1);
\end{tikzpicture}\qquad\qquad
\begin{tikzpicture}[scale=.5,inner sep=.5mm,label distance=.5mm]
\node[bull] (bot) at (0,0) [label=below:$\emptyset$] {};%
\node[bull] (sy1) at (-1,1) [label=left:$\{y_1\}$] {};%
\node[bull] (sz1) at (1,1) [label=right:$\{z_1\}$] {};%
\node (dots) at (0,3.5) {$\vdots$};%
\node (emp) at (0,4) {};
\node[bull] (uu3) at (0,5) [label=right:${\uparrow}u_3$] {};%
\node[bull] (uu2) at (0,6) [label=right:${\uparrow}u_2$] {};%
\node[bull] (Y) at (0,7) [label=above:$Y$] {};%
\draw[color=brown,rounded corners=.4ex] (-.2,2.6) rectangle
(.2,7.2);
\draw (sy1) -- (bot) -- (sz1);%
\draw (emp) -- (uu3) -- (uu2) -- (Y);%
\node[bull] (rttop) at (5,4) [label=above:$X$] {};%
\node[bull] (rtop) at (5,3) [label=right:$\{y{,}z\}$] {};%
\node[bull] (sy) at (4,2) [label=above left:$\{y\}$] {};%
\node[bull] (sz) at (6,2) [label=right:$\{z\}$] {};%
\node[bull] (er) at (5,1) [label=below:$\emptyset$] {};%
\node at (0,-3) {$Y^*$};%
\node at (5,-3) {$X^*$};%
\draw (rtop) -- (sy) -- (er) -- (sz) -- (rtop) -- (rttop);%
\draw[color=brown,->] (bot) -- (er);%
\draw[color=brown,->] (sy1) -- (sy);%
\draw[color=brown,->] (sz1) -- (sz);%
\draw[color=brown,->] (.2,3) -- (rttop);%
\end{tikzpicture}

\

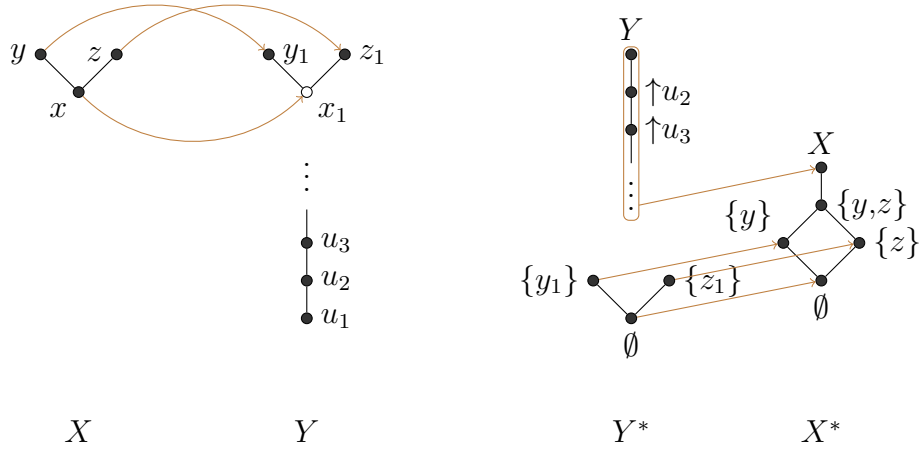
Fig.7\ {}

\end{center}

\begin{lemma}\label{lemma6}
Let $L$ and $K$ be bounded distributive meet semi-lattices and let
$h:L\to K$ be a meet semi-lattice homomorphism preserving top. Then
for $x \in K_{*}$ and $y\in L_*$ we have $R_{h}[x] = {\uparrow}y$
iff $h^{-1}(x) = y$.
\end{lemma}

\begin{proof}
First suppose that $R_{h}[x] = {\uparrow}y$. Then $x R_h y$, so
$h^{-1}(x)\subseteq y$, and so $h^{-1}(x)$ is a proper filter of
$L$. Thus, by the optimal filter lemma, $h^{-1}(x) =
\displaystyle{\bigcap}\{z\in L_*:h^{-1}(x)\subseteq z\} =
\displaystyle{\bigcap} \{z\in L_*: xR_h z\} = \displaystyle{\bigcap}
R_{h}[x] = \displaystyle{\bigcap} {\uparrow}y = y$. Now suppose that
$h^{-1}(x) = y$. Then $R_{h}[x] = \{z\in L_*: x R_h z\} = \{z\in
L_*: h^{-1}(x) \subseteq z\} = \{z\in L_*: y \subseteq z\} =
{\uparrow} y$.
\end{proof}

\begin{theorem}\label{onto-oneone}
Let $L$ and $K$ be bounded distributive meet semi-lattices and let
$h:L\to K$ be a meet semi-lattice homomorphism preserving top.
\begin{enumerate}
\item $h$ is 1-1 iff $R_h$ is onto.
\item $h$ is onto iff $R_{h}$ is 1-1.
\end{enumerate}
\end{theorem}

\begin{proof}
(1)  Suppose that $h$ is 1-1. We show that $R_h$ is onto. Let $y \in
L_{*}$. Since $h$ preserves $\wedge$, we have ${\uparrow}h[y]$ is a
filter of $K$. Let $J$ be the F-ideal generated by $h[L - y]$. We
claim that ${\uparrow}h[y] \cap J = \emptyset$. If ${\uparrow}h[y]
\cap J\neq\emptyset$, then there exist $a \in y$, $a_1, \ldots, a_n
\in L - y$, and $b \in K$ such that $h(a) \leq b$ and
$\displaystyle{\bigcap_{i=1}^n} {\uparrow}h(a_i) \subseteq
{\uparrow}b$. Therefore, $\displaystyle{\bigcap_{i=1}^n}
{\uparrow}h(a_i) \subseteq {\uparrow}h(a)$. Since $h$ is 1-1, we
have $\displaystyle{\bigcap_{i=1}^n} {\uparrow}a_i \subseteq
{\uparrow} a$. As $y$ is an optimal filter of $L$, we have $L - y$
is an F-ideal of $L$, so $a \in L - y$, a contradiction. Thus,
${\uparrow}h[y] \cap J = \emptyset$, and by the optimal filter
lemma, there is $x \in K_{*}$ such that ${\uparrow}h[y] \subseteq x$
and $x \cap J = \emptyset$. It follows that $h^{-1}(x) = y$, and so
$R_{h}[x] = {\uparrow}y$ by Lemma \ref{lemma6}. Now suppose that
$R_{h}$ is onto. For $a, b \in L$ with $a \not = b$, we may assume
without loss of generality that $a \not \leq b$. Then
${\uparrow}a\cap {\downarrow}b=\emptyset$, and so by the prime
filter lemma, there is $y\in L_+\subseteq L_{*}$ such that $a \in y$
and $b \notin y$. Since $R_h$ is onto, there is $x \in K_{*}$ such
that $R_{h}[x] = {\uparrow}y$. This, by Lemma \ref{lemma6}, implies
that $h^{-1}(x) = y$. Thus, $h(a) \in x$ and $h(b) \notin x$, and so
$h(a) \not \leq h(b)$. It follows that $h$ is 1-1.

(2) Suppose that $h$ is onto. We show that $R_h$ is 1-1. Let $x\in
K_*$, $b \in K$, and $x\notin\phi(b)$. Since $h$ is onto, there is
$a \in L$ such that $h(a) = b$. By Proposition \ref{h-Rh},
$\Box_{R_h}(\phi(a))=\phi(b)$. So $R_h[\phi(b)] \subseteq \phi(a)$
and $R_{h}[x] \not \subseteq \phi(a)$. Thus, $R_h$ is 1-1. Now
suppose that $R_h$ is 1-1. Let $b\in K$. For each $x\in K_*$ such
that $b\notin x$, we have $x\notin\phi(b)$. Since $R_h$ is 1-1,
there exists $a_x \in L$ such that $R_{h}[\phi(b)] \subseteq
\phi(a_x)$ and $R_{h}[x] \not \subseteq \phi(a_x)$. Then $\phi(b)
\subseteq \Box_{R_{h}}\phi(a_x)$ and $x \not \in
\Box_{R_{h}}\phi(a_x)$. Therefore, $\displaystyle{\bigcap}
\{\Box_{R_h}\phi(a_x): x \notin \phi(b)\} \cap \phi(b)^{c} =
\emptyset$. Since $X$ is compact, there exist $x_1, \ldots x_n
\notin \phi(b)$ such that
$\displaystyle{\bigcap_{i=1}^n}\Box_{R_h}\phi(a_{x_i}) \cap
\phi(b)^{c} = \emptyset$. Thus, $ \Box_{R_h}\phi(a_{x_1} \wedge
\ldots \wedge a_{x_n}) \cap \phi(b)^{c} = \emptyset$, so $\phi(b)
= \Box_{R_h}\phi(a_{x_1} \wedge \ldots \wedge
a_{x_n})=\phi(h(a_{x_1} \wedge \ldots \wedge a_{x_n}))$, and so
$b=h(a_{x_1} \wedge \ldots \wedge a_{x_n})$. It follows that $h$
is onto.
\end{proof}

\begin{proposition}\label{6.4}
Let $X$ and $Y$ be generalized Priestley spaces and let $R\subseteq
X\times Y$ be a generalized Priestley morphism.
\begin{enumerate}
\item $R\subseteq X\times Y$ is onto iff $R_{h_{R}}\subseteq
{X^{*}}_{*}\times {Y^{*}}_{*}$ is onto.
\item $R\subseteq X\times Y$ is 1-1 iff $R_{h_{R}}\subseteq
{X^{*}}_{*}\times {Y^{*}}_{*}$ is 1-1.
\end{enumerate}
\end{proposition}

\begin{proof}
Apply Theorem \ref{homeom} and Propositions \ref{h-Rh} and
\ref{R-hR}.
\end{proof}

\begin{theorem}
Let $X$ and $Y$ be generalized Priestley spaces and let $R\subseteq
X\times Y$ be a generalized Priestley morphism.
\begin{enumerate}
\item $R$ is 1-1 iff $h_R$ is onto.
\item $R$ is onto iff $h_{R}$ is 1-1.
\end{enumerate}
\end{theorem}

\begin{proof}
It follows from Theorem \ref{onto-oneone} and Proposition \ref{6.4}.
\end{proof}

Thus, we obtain that 1-1 homomorphisms of bounded distributive meet
semi-lattices preserving top correspond to onto generalized
Priestley morphisms, and that bounded 1-1 homomorphisms correspond
to total onto generalized Priestley morphisms. Moreover, onto
homomorphisms of bounded distributive meet semi-lattices preserving
top coincide with bounded onto homomorphisms (which is easy to see
either algebraically or by recalling that each 1-1 generalized
Priestley morphism is total) and correspond to 1-1 generalized
Priestley morphisms.

As an immediate consequence of the bounded distributive meet
semi-lattice case, we obtain that 1-1 homomorphisms of bounded
implicative meet semi-lattices correspond to onto generalized
Esakia morphisms, that bounded 1-1 homomorphisms correspond to
total onto generalized Esakia morphisms, and that onto
homomorphisms are the same as bounded onto homomorphisms and
correspond to 1-1 generalized Esakia morphisms.

\subsubsection{Dual description of 1-1 and onto homomorphisms for
bounded distributive lattices and Heyting algebras}

Now we show how our results above imply the well-known dual
description of 1-1 and onto homomorphisms of bounded distributive
lattices and Heyting algebras.

\begin{lemma}\label{lemma-P}
Let $X$ and $Y$ be Priestley spaces and let $R\subseteq X\times Y$
be a functional generalized Priestley morphism.
\begin{enumerate}
\item $R$ is 1-1 iff $f^R$ is an embedding.
\item $R$ is onto iff $f^R$ is onto.
\end{enumerate}
\end{lemma}

\begin{proof}
(1) Suppose that $R$ is 1-1. We show that $f^R$ is an embedding. Let
$x\leq y$. Then $R[y]\subseteq R[x]$. Since $R$ is a functional
generalized Priestley morphism, both $R[x]$ and $R[y]$ have least
elements. Let $l_x$ be the least element of $R[x]$ and $l_y$ be the
least element of $R[y]$. Then $R[y]\subseteq R[x]$ implies $l_x\leq
l_y$. Thus, $f^R(x)\leq f^R(y)$. Now let $x\not\leq y$. Since $R$ is
1-1, we have $R[y]\not\subseteq R[x]$. Therefore, $l_x\not\leq l_y$,
so $f^R(x)\not\leq f^R(y)$, and so $f^R$ is an embedding.
Conversely, suppose that $f^R$ is an embedding, $x\in X$,
$U\in\mathfrak{CU}(X)$, and $x\notin U$. Since $f^R$ is an
embedding, we have $f^R(x)\notin{\uparrow}f^R(U)$. So there exists
$V\in\mathfrak{CU}(Y)$ such that $f^R(x)\notin V$ and
${\uparrow}f^R(U)\subseteq V$. Thus, $R[x]={\uparrow}f^R(x)\not
\subseteq U$ and $R[U]={\uparrow}f^R(U)\subseteq V$, and so $R$ is
1-1.

(2) Suppose that $R$ is onto and $y\in Y$. Then there exists $x\in
X$ such that $R[x]={\uparrow}y$. Thus, $f^R(x)=y$, and so $f^R$ is
onto. Now suppose that $f^R$ is onto and $y\in Y$. Then there exists
$x\in X$ such that $f^R(x)=y$. Thus, $R[x]={\uparrow}f^R(x)=
{\uparrow}y$, and so $R$ is onto.
\end{proof}

\begin{corollary}\label{cor-fR}
Let $X$ and $Y$ be Priestley spaces and let $f:X\to Y$ be a
Priestley morphism.
\begin{enumerate}
\item $f$ is an embedding iff $R^f$ is 1-1.
\item $f$ is onto iff $R^f$ is onto.
\end{enumerate}
\end{corollary}

\begin{proof}
Apply Lemmas \ref{lemma-c} and \ref{lemma-P}.
\end{proof}

Let $L$ and $K$ be bounded distributive lattices and let $h:L\to
K$ be a bounded lattice homomorphism. Define $f_h:K_+\to L_+$ by
$f_h(x)=h^{-1}(x)$ for each $x\in K_+$. It is well-known (and
follows from Lemmas \ref{lemma-dist} and
\ref{lemma-a}) that $f_h$ is a well-defined Priestley morphism.
The next proposition is a well-known consequence of the Priestley
duality. We show that it as an easy consequence of Theorem
\ref{onto-oneone} and Lemma \ref{lemma-P}.

\begin{lemma}
Let $L$ and $K$ be bounded distributive lattices and let $h:L\to K$
be a bounded lattice homomorphism.
\begin{enumerate}
\item $h$ is 1-1 iff $f_h$ is onto.
\item $h$ is onto iff $f_h$ is an embedding.
\end{enumerate}
\end{lemma}

\begin{proof}
For $x\in K_+$ and $y\in L_+$ we have $x R_h y$ iff
$h^{-1}(x)\subseteq y$ iff $f_h(x)\subseteq y$. Therefore,
$R_h[x]={\uparrow}f_h(x)$, so $f_h=f_{R_h}$ (and $R_h=R_{f_h}$).
Thus, by Theorem \ref{onto-oneone} and Lemma \ref{lemma-P}, $h$ is
1-1 iff $R_h$ is onto iff $f_h$ is onto, and $h$ is onto iff $R_h$
is 1-1 iff $f_h$ is an embedding.
\end{proof}

The above results immediately imply that for a Heyting algebra
homomorphism $h$, we have that $h$ is 1-1 iff $f_h$ is onto, and
that $h$ is onto iff $f_h$ is 1-1. Another way to see this is
through partial Heyting functions. But first we characterize 1-1 and
onto $(\wedge,\to)$ and $(\wedge,\to,\bot)$-homomorphisms of Heyting
algebras in terms of onto and 1-1 partial Esakia functions.

\begin{definition}
{\rm Let $X$ and $Y$ be Esakia spaces and let $f:X\to Y$ be a partial
Esakia function.
\begin{enumerate}
\item We call $f$ \emph{onto} if the restriction of $f$ to
$\mathrm{dom}(f)$ is an onto function.
\item We call $f$ \emph{1-1} if for each $x\in X$ and $U\in X^*$,
from $x\notin U$ it follows that there exists $V\in Y^*$ such that
$U\subseteq ({\downarrow}f^{-1}(V^c))^c$ and $f[{\uparrow}x] \not
\subseteq V$.
\end{enumerate}
}
\end{definition}

Let $X$ and $Y$ be Esakia spaces and let $f:X\to Y$ be a partial
Esakia function. We observe that if $f$ is 1-1, then $f$ is well.
Indeed, if $f$ is not well, then there exists $x\in \mathrm{max}X$
such that $x\notin\mathrm{dom}(f)$. Therefore,
$f[{\uparrow}x]=\emptyset$, and so $f$ can not be 1-1.

\begin{lemma}
Let $X$ and $Y$ be Esakia spaces and let $f:X\to Y$ be a partial
Esakia function. Then $f$ is onto iff $R^f$ is onto, and $f$ is 1-1
iff $R^f$ is 1-1.
\end{lemma}

\begin{proof}
First suppose that $f$ is onto and $y\in Y$. Then there exists
$x\in\mathrm{dom}(f)$ such that $f(x)=y$. Therefore, ${\uparrow}f(x)
= {\uparrow}y$. As $x\in\mathrm{dom}(f)$, we have $R^f[x] =
{\uparrow}f(x) = {\uparrow}y$, and so $R^f$ is onto. Conversely,
suppose that $R^f$ is onto and $y\in Y$. Then there is $x\in X$ such
that $R^f[x]={\uparrow}y$. Therefore, $xR^fy$, and so there exists
$z\in\mathrm{dom}(f)$ such that $x\leq z$ and $f(z)=y$. Thus, $f$ is
onto.

Now suppose that $f$ is 1-1, $x\in X$, $U\in X^*$, and $x\notin
U$. Then there exists $V\in Y^*$ such that $U\subseteq
({\downarrow}f^{-1}(V^c))^c$ and $x\notin
({\downarrow}f^{-1}(V^c))^c$. Therefore, $U\subseteq\Box_{R^f}(V)$
and $x\notin\Box_{R^f}(V)$. Thus, $R^f[U]\subseteq V$ and
$R^f[x]\not\subseteq V$, and so $R^f$ is 1-1. Conversely, suppose
that $R^f$ is 1-1, $x\in X$, $U\in X^*$, and $x\notin U$. Then
there exists $V\in Y^*$ such that $R^f[U]\subseteq V$ and
$R^f[x]\not\subseteq V$. Therefore, $U\subseteq\Box_{R^f}(V)$ and
$R^f[x]\not\subseteq V$. Since $\Box_{R^f}(V)=
({\downarrow}f^{-1}(V^c))^c$ and $R^f[x]=f[{\uparrow}x]$, we
obtain $U\subseteq ({\downarrow}f^{-1}(V^c))^c$ and
$f[{\uparrow}x] \not \subseteq V$. Thus, $f$ is 1-1.
\end{proof}

As a consequence, we obtain that 1-1 $(\wedge,\to)$-homomorphisms of
Heyting algebras correspond to onto partial Esakia functions, that
1-1 $(\wedge,\to,\bot)$-homomorphisms correspond to well onto
partial Esakia functions, that onto $(\wedge,\to)$-homomorphisms are
the same as onto $(\wedge,\to,\bot)$-homomorphisms and correspond to
1-1 partial Esakia functions. Moreover, 1-1 Heyting algebra
homomorphisms correspond to onto partial Heyting functions and onto
Heyting algebra homomorphisms correspond to 1-1 partial Heyting
functions. We show that a partial Heyting function $f$ is onto iff
$g_f$ is an onto function and that $f$ is 1-1 iff $g_f$ is a 1-1
function.

Let $X$ and $Y$ be Esakia spaces and let $f:X\to Y$ be a partial
Heyting function. First we show that $f$ is onto iff $g_f$ is an
onto function. Let $f$ be onto and $y\in Y$. Then there exists $x\in
\mathrm{dom}(f)$ such that $f(x)=y$. Since $x\in \mathrm{dom}(f)$,
we have $g_f(x)=f(x)$. Thus, $g_f(x)=y$, and so $g_f$ is an onto
function. Conversely, suppose that $g_f$ is an onto function and
$y\in Y$. Then there exists $x\in X$ such that $g_f(x)=y$.
Therefore, there exists $z\in\mathrm{dom}(f)$ such that $x\leq z$
and $g_f(x)=f(z)$. Thus, $f(z)=y$ and so $f$ is onto.

Next we show that $f$ is 1-1 iff $g_f$ is a 1-1 function. Let $f$
be 1-1. We show that $g_f$ is a 1-1 function. Let $x,y\in X$ with
$x\neq y$. Without loss of generality we may assume that
$x\not\leq y$. Then there exists a clopen upset $U$ of $X$ such
that $x\in U$ and $y\notin U$. Since $f$ is 1-1, there exists a
clopen upset $V$ of $Y$ such that $U\subseteq
({\downarrow}f^{-1}(V^c))^c$ and $f[{\uparrow}y]\not\subseteq V$.
But $({\downarrow}f^{-1}(V^c))^c=g_f^{-1}(V)$ and $f[{\uparrow}y]=
{\uparrow}g_f(y)$. Therefore, $U\subseteq g_f^{-1}(V)$ and
$g_f(y)\notin V$. Thus, $x\in U\subseteq g_f^{-1}(V)$ and $y\notin
g_f^{-1}(V)$, implying that $g_f(x)\not\leq g_f(y)$. Consequently,
$g_f$ is a 1-1 function. Conversely, suppose that $g_f$ is a 1-1
function. Let $x\in X$, $U$ be a clopen upset of $X$, and $x\notin
U$. Since $g_f$ is a 1-1 function, we have $g_f(x)\not\in g_f(U)$.
As $g_f(U)$ is a closed upset of $Y$, there exists a clopen upset
$V$ of $Y$ such that $g_f(U)\subseteq V$ and $g_f(x)\notin V$.
Therefore, $U\subseteq g_f^{-1}(V)=(f^{-1}(V^c))^c$ and
$f[{\uparrow}x]={\uparrow}g_f(x)\not\subseteq V$. Thus,
$U\subseteq ({\downarrow}f^{-1}(V^c))^c$ and
$f[{\uparrow}x]\not\subseteq V$, implying that $f$ is 1-1.

Consequently, we obtain the well-known result that for a Heyting
algebra homomorphism $h$ we have $h$ is 1-1 iff $f_h$ is onto, and
that $h$ is onto iff $f_h$ is 1-1.

\section{Non-bounded case}

The duality we have developed for bounded distributive meet
semi-lattices can be modified accordingly to obtain duality for
non-bounded distributive meet semi-lattices. In this section we
discuss briefly the main ideas of the modification.

First we deal with the case of distributive meet semi-lattices
with top but possibly without bottom. Let $L\in\mathsf{DM}$. If
$L$ does not have bottom, then we have to add $L$ to the set of
optimal filters of $L$. As a result, $L$ is the greatest element
of $L_*$, and so $\mathrm{max}(L_*)$ is not contained in $L_+$.
Thus, we have to drop condition (3) of Definition
\ref{g-Priestley}. Moreover, for $x=L$ we have
$\mathcal{I}_x=\emptyset$, so $\mathcal{I}_x$ is trivially
directed, although $L\notin L_+$. Thus, we have to modify
condition (4) of Definition \ref{g-Priestley} as follows: $x \in
L_{+}$ iff $\mathcal{I}_{x}$ is nonempty and directed. This
suggests the following modification of the definition of a
generalized Priestley space.

\begin{definition}
{\rm A quadruple $X = \la X, \tau, \leq, X_{0}\ra$ is called a
\emph{$*$-generalized Priestley space} if (i) $\la X, \tau, \leq\ra$
is a Priestley space, (ii) $X_{0}$ is a dense subset of $X$, (iii)
$x \in X_ {0}$ iff $\mathcal{I}_{x}$ is nonempty and directed, and
(iv) for all $x, y \in X $, we have $x \leq y$ iff $(\forall U \in
X^{*})(x \in U \Rightarrow y \in U)$.
}
\end{definition}

Clearly each generalized Priestley space is a $*$-generalized
Priestely space, so the concept of $*$-generalized Priestely space
generalizes that of generalized Priestely space. Moreover, a
$*$-generalized Priestely space is a generalized Priestley space iff
$\mathrm{max}(X) \subseteq X_{0}$ iff $X^{*}$ has a bottom element.
Therefore, a $*$-generalized Priestely space is a generalized
Priestely space iff it satisfies condition (4) of Definition
\ref{g-Priestley}. The same generalization of generalized Esakia
spaces yields $*$-generalized Esakia spaces. Let $\mathsf{GPS}^*$
denote the category of $*$-generalized Priestley spaces and
generalized Priestely morphisms, and let $\mathsf{GES}^*$ denote the
category of $*$-generalized Esakia spaces and generalized Esakia
morphisms. Then we immediately obtain the following theorem, which
generalizes the duality we obtained for $\mathsf{BDM}$ and
$\mathsf{BIM}$ to $\mathsf{DM}$ and $\mathsf{IM}$, respectively.

\begin{theorem}
\begin{enumerate}
\item[]
\item $\mathsf{DM}$ is dually equivalent to $\mathsf{GPS}^*$.
\item $\mathsf{IM}$ is dually equivalent to $\mathsf{GES}^*$.
\end{enumerate}
\end{theorem}


If $L$ is a distributive meet-semilattice without top but with
bottom, then two cases are possible: either $D(L)$ has top or $D(L)$
does not have top. If $D(L)$ has top, then we obtain the dual of $L$
in exactly the same way as in the bounded case. But in this case $L$
will be realized as ${L_*}^*-\{L_*\}$. If $D(L)$ does not have top,
then again we construct the dual of $L$ as before, however in this
case the space we obtain is locally compact but not compact. We can
handle this as the case for distributive lattices \cite[Section
10]{Pri72} by adding a new top to $L$. If $L^{\top}$ is the
resulting meet semi-lattice, then the dual space of $L^{\top}$ is
the one-point compactication of the dual of $L$. Moreover, the new
point of $(L^\top)_*$ is the smallest optimal filter $\{\top\}$ of
$L^\top$, which is below every point of $L_*$.

This way we can handle all possible situations; that is, when $L$
has $\top$, but lacks $\bot$; when $L$ has $\bot$, but lacks $\top$;
or the most general case, when $L$ lacks both $\top$ and $\bot$.

\section{Comparison with the relevant work}

In this final section we compare our duality with the existing
dualities for distributive meet semi-lattices and implicative meet
semi-lattices.

\subsection{Comparison with the work of Celani and Hansoul}

The first representation of distributive meet semi-lattices is
already present in the pioneering work of Stone \cite{Sto37}. It was
made more explicit in Gr\"atzer \cite{Gra98}, where with each join
semi-lattice $L$ with bottom is associated the space $S(L)$ of prime
ideals of $L$. The topology on $S(L)$ is generated by the basis
consisting of the sets $r(a) = \{I\in S(L): a \not \in I\}$. The
space $S(L)$ is not in general Hausdorff, and it is compact iff $L$
has a top. In \cite{Gra98} a purely topological characterization of
such spectral-like spaces is given. We call a topological space $\la
X, \tau\ra$ \emph{spectral-like} if (i) $X$ is $T_0$, (ii) the
compact open subsets of $X$ form a basis for $\tau$, and (iii) for
each closed subset $F$ of $X$ and each nonempty down-directed family
of compact open subsets $\{U_i: i \in I\}$ of $X$, from $F \cap U_i
\neq \emptyset$ for each $i\in I$, it follows that $F \cap
\displaystyle{\bigcap}\{U_i: i \in I\} \neq \emptyset$. It was shown
in \cite{Gra98} that $S(L)$ is spectral-like for each $L$, that the
basis of compact open subsets of a spectral-like space forms a
distributive join semi-lattice with bottom, and that this
correspondence between distributive join semi-lattices with bottom
and spectral-like spaces is 1-1. But there was no attempt made in
\cite{Gra98} to obtain categorical duality between the category of
distributive join semi-lattices with bottom and the category of
spectral-like spaces.

In \cite{Cel03b} Celani filled in this gap. He chose to work with
meet semi-lattices instead of join semi-lattices, hence his building
blocks for the dual space were prime filters instead of prime
ideals. To be more specific, let us recall that $\mathsf{DM}$
denotes the category of distributive meet semi-lattices with top as
objects and meet semi-lattice homomorphisms preserving top as
morphisms. Celani's dual category has (in the terminology of
\cite{Cel03b}) DS-spaces as objects and meet-relations between two
DS-spaces as morphisms. We recall from \cite[Definition 14]{Cel03b}
that a \textit{DS-space} is an ordered topological space $X=\la X,
\tau, \leq\ra$ such that (i) $\la X, \tau \ra$ is a spectral-like
space, (ii) each closed subset of $X$ is an upset, and (iii) if $x,
y \in X$ with $x \not \leq y$, then there is a compact open subset
$U$ of $X$ such that $x \not \in U$ and $y \in U$.

\begin{remark}
For a topological space $X$, we recall that the \emph{specialization
order} of $X$ is defined by $x\leq y$ if $x\in\overline{\{y\}}$,
that $\leq$ is reflexive and transitive, and that it is a partial
order iff $X$ is $T_0$. Let $X$ be a spectral-like space. Then it is
easy to see that Celani's order is nothing more than the dual of the
specialization order of $X$.
\end{remark}

For a DS-space  $X$, let $\mathcal{E}(X)$ denote the set of compact
open subsets of $X$, and let $D_X = \{U \subseteq X: U^{c} \in
\mathcal{E}(X)\}$. Let $X$ and $Y$ be two DS-spaces and let $R
\subseteq X \times Y$ be a binary relation. Following
\cite[Definition 19]{Cel03b}, we call $R$ a \textit{meet-relation}
if (i) $U \in D_Y$ implies $\Box_{R}U \in D_X$ and (ii) $R[x] =
\displaystyle{\bigcap}\{U \in D_Y: R[x] \subseteq U\}$ for each
$x\in X$. Let $\mathsf{DS}$ denote the category of DS-spaces as
objects and meet-relations as morphisms.

Celani defines two functors $(-)_{+}:\mathsf{DM}\to\mathsf{DS}$ and
$(-)^{+}:\mathsf{DS}\to\mathsf{DM}$ as follows. If $L \in
\mathsf{DM}$, then $L_+=\la L_{+}, \tau, \subseteq \ra$, where
$L_{+}$ is the set of prime filters of $L$ and $\tau$ is the
topology generated by the basis $\{\sigma(a)^c: a \in L\}$; if $h: L
\to K$ is a meet semi-lattice homomorphism preserving $\top$, then
$h_+=R_h \subseteq K_{+} \times L_{+}$ is defined by $x R_h y$ iff
$h^{-1}(x) \subseteq y$. If $X$ is a DS-space, then $X^+ = \la D_X,
\cap, X\ra$; if $X$ and $Y$ are DS-spaces and $R\subseteq X\times Y$
is a meet-relation, then $R^+=h_R: D_Y \to D_X$ is defined by
$h_R(U) = \Box_{R}U$. Then it follows from \cite{Cel03b} that the
functors $(-)_+$ and $(-)^+$ are well-defined, and that they
establish dual equivalence of the categories $\mathsf{DM}$ and
$\mathsf{DS}$.

\begin{remark}
Although not addressed in \cite{Cel03b}, the composition of two meet-relations is not the usual set-theoretic composition. Rather, similar to the case of $\mathsf{GPS}$, we have that for DS-spaces $X,Y,$ and $Z$ and meet-relations $R\subseteq X\times Y$ and $S\subseteq Y\times Z$, the composition $S*R\subseteq X\times Z$ is given by $$x(S*R)z \mbox{ iff } (\forall U\in Z^+)(x\in (h_R\circ h_S)(U)\Rightarrow z\in U)$$ for each $x\in X$ and $z\in Z$.
\end{remark}

The bounded distributive meet semi-lattices are exactly the objects
of $\mathsf{DM}$ whose dual spaces are compact. Indeed, if $L$ is a
bounded distributive meet semi-lattice, then $\sigma(\bot) =
\emptyset$, so $\sigma(a)^c=L_+$, and so $L_{+}$ is compact as
$\{\sigma(a)^c:a\in L\}=\mathcal{E}(L_+)$. Conversely, if $L_{+}$ is
compact, then $L_{+} = \sigma(a)^c$ for some $a \in L$, so $a$ is
the bottom of $L$, and so $L$ is bounded. It follows that the full
subcategory $\mathsf{BDM}$ of $\mathsf{DM}$ whose objects are
bounded distributive meet semi-lattices is dually equivalent to the
full subcategory $\mathsf{CDS}$ of $\mathsf{DS}$ whose objects are
compact DS-spaces. Now putting Celani's duality together with ours,
we obtain that $\mathsf{CDS}$ is equivalent to the category of
generalized Priestley spaces and generalized Priestley morphisms
introduced in this paper. In order to make the connection easier to
understand, we show how to construct compact DS-spaces from
generalized Priestely spaces, and how to construct meet-relations
from generalized Preistely morphisms.

\begin{proposition}\label{prop-celani}
Let $X = \la X, \tau, \leq, X_0\ra$ be a generalized Priestely
space. Then the space $X_0 = \la X_0, \tau_0, \leq_0 \ra$ is a
compact DS-space, where $\leq_0$ is the restriction of $\leq$ to
$X_0$ and $\tau_0$ is the topology generated by the basis $\{X_0-U :
U \in X^{*}\}$.
\end{proposition}

\begin{proof}
First we show that $\{X_0-U : U \in X^{*}\}$ is indeed a basis for
$\tau_0$. Let $U,V\in X^*$ and $x\in (X_0-U) \cap (X_0-V)$. Then
$x\in X_0$ and $x\notin U,V$. By condition (4) of Definition
\ref{g-Priestley}, there exists $W\in X^*$ such that $U,V\subseteq
W$ and $x\notin W$. Thus, $x\in X_0-W\subseteq (X_0-U) \cap
(X_0-V)$, and so $\{X_0-U : U \in X^{*}\}$ is a basis for $\tau_0$.
Next we show that compact open subsets of $X_0$ are exactly of the
form $X_0 - U$ for $U \in X^{*}$. Let $U \in X^{*}$ and let $X_0 - U
\subseteq \displaystyle{\bigcup_{i \in I}} (X_0 - U_i)$, where
$\{U_i: i \in I\} \subseteq X^{*}$. Then ${\downarrow}(X_0 - U)
\subseteq \displaystyle{\bigcup_{i \in I}} {\downarrow}(X_0 - U_i)$.
Since for each $V \in X^{*}$ we have $X - V = {\downarrow}(X_0 -
V)$, then $X - U \subseteq \displaystyle{\bigcup_{i \in I}} (X -
U_i)$. As $X-U$ is closed, hence compact in $X$, there is a finite
$J \subseteq I$ such that $X - U \subseteq \displaystyle{\bigcup_{i
\in J}} (X-U_i)$. Hence, $X_0 - U \subseteq \displaystyle{\bigcup_{i
\in J}} (X_0 - U_i)$, and so $X_0 - U$ is compact. Conversely, if
$V$ is a compact open subset of $X_0$, then $V$ is a finite union of
the sets of the form $X_0 - V_i$. Since $X^*$ is closed under finite
intersections, $\{X_0-U:U\in X^*\}$ is closed under finite unions.
Thus, each compact open subset of $X_0$ has the form $X_0-U$ for
$U\in X^*$, and so $\{X_0 - U: U \in X^{*}\}$ is exactly the set of
compact open subsets of $X_0$. Moreover, as $\emptyset\in X^*$, we
have that $X_0 = X_0-\emptyset$ is compact open, and so $X_0$ is
compact. Now suppose that $x, y \in X_0$ are such that $x \not \leq
y$. By condition (5) of Definition \ref{g-Priestley}, there exists
$U \in X^{*}$ such that $x \in U$ and $y \not \in U$. Thus, $x \not
\in X_0 - U$ and $y \in X_0 - U$, and so condition (iii) of the
definition of a DS-space is satisfied. It also follows that $X_0$ is
a $T_0$-space. Since each $U \in X^{*}$ is an upset of $X$, each
$X_0 - U$ is a downset of $X_0$. Thus, open sets of $X_0$ are
downsets, and so closed sets of $X_0$ are upsets. Consequently,
$X_0$ satisfies condition (ii) of the definition of a DS-space.
Finally, let  $F$ be a closed subset of $X_0$ and let $\{U_i: i \in
I\}$ be a down-directed family of compact open subsets of $X_0$ such
that for each $i \in I$ we have $F \cap U_i \neq \emptyset$. Then $F
= \displaystyle{\bigcap} \{W_k: k \in K\}$ for some family $\{W_k: k
\in K\} \subseteq X^{*}$ and $U_i=X_0-V_i$ for some $V_i\in X^*$.
Since $\{X_0 - V_i: i \in I\}$ is down-directed, so is
$\{{\downarrow}(X_0 - V_i): i \in I\}$. But ${\downarrow}(X_0 -
V_i)=X-V_i$, so $\{V_i: i \in I\}\subseteq X^*$ is directed.
Therefore, $\Delta=\{V_i:i\in I\}$ is an ideal of $X^{*}$. Let
$\nabla$ be the filter of $X^{*}$ generated by $\{W_k: k \in K\}$.
Clearly $\nabla \cap \Delta= \emptyset$. Thus, there exists a prime
filter $P$ of $X^{*}$ such that $\nabla \subseteq P$ and $P \cap
\Delta = \emptyset$. Let $x \in X_0$ be such that $\varepsilon(x) =
P$. Then $x \in F$ and $x \notin V_i$ for each $i\in I$.
Consequently, $x\in F\cap \displaystyle{\bigcap_{i\in I}}
(X_0-V_i)$, so $X_0$ is a spectral-like space, so condition (i) of
the definition of a DS-space is satisfied, and so $X_0$ is a compact
DS-space.
\end{proof}

\begin{proposition}\label{prop-celani1}
If $X$ and $Y$ are generalized Priestley spaces and $R \subseteq X
\times Y$ is a generalized Priestley morphism, then $R_0 = R \cap
(X_0 \times Y_{0})$ is a meet-relation between the compact DS-spaces
$X_0$ and $Y_0$.
\end{proposition}

\begin{proof}
Suppose $R \subseteq X \times Y$ is a generalized Priestley morphism
between two generalized Priestey spaces $X$ and $Y$. We prove that
$R_0$ is a meet-relation between $X_0$ and $Y_0$. Let  $U \in
D_{Y_0}$. Then $Y_0 - U$ is compact open in $Y_0$. Therefore, there
is $V \in Y^{*}$ such that $Y_0 - U = Y_0 - V$. Thus, $U = V \cap
Y_0$ and ${\downarrow}(Y_0 - U) = Y - V$. We show that $\Box_{R_0} U
= \Box_R V\cap X_0$. First suppose that $x \in \Box_R V\cap X_0$.
Then $x\in X_0$ and $R[x] \subseteq V$. Thus, $R_0[x] \subseteq
V\cap Y_0 = U$, and so $x \in \Box_{R_0} U$. Next suppose that $x
\in \Box_{R_0} U$. If $R[x] \not\subseteq V$, then there exists
$y\in Y$ such that $x R y$ and $y \not \in V$. Therefore, $y \in Y -
V = {\downarrow}(Y_0 - U)$. So there exists $z \in Y_0 - U$ such
that $y \leq z$. Since $xRy$ and $y\leq z$, we have $xRz$, so $xR_0
z$. Thus, $z \in R_0[x]$ and $z \not \in U$, a contradiction. We
conclude that $R[x]\subseteq V$, and so $x \in \Box_R V\cap X_0$. It
follows that $\Box_{R_0} U = \Box_R V\cap X_0$, so $X_0-\Box_{R_0}U
= X_0-\Box_R V$, and so $\Box_{R_0}U \in D_{X_0}$. Thus, condition
(i) of the definition of a meet-relation is satisfied. Now we show
that for each $x \in X_0$ we have $R_0[x] = \displaystyle{\bigcap}
\{U \in D_{Y_0}: R_{0}[x] \subseteq U\}$. Clearly $R_0[x] \subseteq
\displaystyle{\bigcap}\{U \in D_{Y_0}: R_{0}[x] \subseteq U\}$.
Suppose that $y\in Y_0$ and $x \nR_0 y$. Then $x\nR y$. By condition
(3) of Definition \ref{definition}, there exists $V \in Y^{*}$ such
that $y \not \in V$ and $R[x] \subseteq V$. Therefore, $V \cap X_0
\in D_{Y_0}$, $R_0[x] \subseteq V \cap X_0$, and $y \notin V\cap
X_0$. Thus, $y\notin \displaystyle{\bigcap}\{U \in D_{Y_0}: R_{0}[x]
\subseteq U\}$, so $R_0[x] = \displaystyle{\bigcap}\{U \in D_{Y_0}:
R_{0}[x] \subseteq U\}$, and so condition (ii) of the definition of
a meet-relation is satisfied. Consequently, $R_0$ is a meet-relation
between the compact DS-spaces $X_0$ and $Y_0$.
\end{proof}

Now we compare our work with that of Hansoul \cite{Han03}. Like
Gr\"atzer, Hansoul prefers to work with distributive join
semi-lattices. But unlike both Gr\"atzer and Celani, he tries to
build a Priestley-like dual of a bounded distributive join
semi-lattice. Thus, his work is the closest to ours. We recall the
main definition of \cite{Han03}. For convenience, we call the spaces
Hansoul calls Priestley structures simply Hansoul spaces.

\begin{definition}
\label{Priestley structure} {\rm A \textit{Hansoul space} is a tuple
$X=\la X, \tau, \leq, X_{0} \ra$, where:
\begin{enumerate}
\item $\la X, \tau, \leq \ra$ is a Priestley space.
\item $X_{0}$ is a dense subset of $X$.
\item If $x, y \in X$ with $x \not \leq y$, then there is
$z \in X_{0}$ such that $x \not \leq z$ and $y\leq z$.
\item $X_{0}$ is the set of elements of $X$ for
which the family of clopen downsets $U$ that contain $x$ and have
the property that $U \cap X_{0}$ is cofinal in $U$ is a basis of
clopen downset neighborhoods of $x$.
\item For each $x\in X$ there exists $y\in X_0$ such that $x\leq y$.
\end{enumerate}
}
\end{definition}

Hansoul constructs the dual $X$ of a bounded distributive join
semi-lattice $L$ by taking what he calls \emph{weakly prime ideals}
of $L$ as points of $X$ and by taking prime ideals of $L$ as points
of the dense subset $X_0$ of $X$. The weakly prime ideals of $L$ are
exactly the optimal filters of the dual $L^d$ of $L$ (see Remark
\ref{Hansoul}) and the prime ideals of $L$ are exactly the prime
filters of $L^d$. Thus, Hansoul's construction is dual to ours. In
fact, we show that Hansoul spaces are the same as our generalized
Priestley spaces.

\begin{proposition}
A tuple $X = \la X, \tau, \leq, X_{0}\ra$ is a Hansoul space iff it
is a generalized Priestely space.
\end{proposition}

\begin{proof}
Since conditions (1) and (2) of Definition \ref{g-Priestley} are the
same as conditions (1) and (2) of Definition \ref{Priestley
structure} and condition (3) of Definition \ref{g-Priestley} is the
same as condition (5) of Definition \ref{Priestley structure}, it is
sufficient to show that if $X$ is a Hansoul space, then $X$
satisfies conditions (4) and (5) of Definition \ref{g-Priestley},
and that if $X$ is a generalized Priestley space, then $X$ satisfies
conditions (3) and (4) of Definition \ref{Priestley structure}.

Let $X$ be a Hansoul space. First we show that $X$ satisfies
condition (5) of Definition \ref{g-Priestley}. If $x, y \in X$ with
$x \leq y$ then it is clear that for each $U \in X^{*}$ we have $x
\in U$ implies $y \in U$ (because $U$ is an upset). Suppose that $x
\not \leq y$. By condition (3) of Definition \ref{Priestley
structure}, there exists $z \in X_{0}$ such that $x \not \leq z$ and
$y \leq z$. Since $x \not \leq z$ and $X$ is a Priestley space,
there is a clopen downset $V$ of $X$ such that $z \in V$ and $x
\notin V$. By condition (4) of Definition \ref{Priestley structure},
we can take $V$ such that $V \cap X_{0}$ is cofinal in $V$. Thus,
there exists $U=V^c$ in $X^*$ such that $x \in U$ and $y \notin U$,
and so $X$ satisfies condition (5) of Definition \ref{g-Priestley}.
Now we show that $X$ satisfies condition (4) of Definition
\ref{g-Priestley}. Suppose that $x \in X_{0}$. We show that
$\mathcal{I}_{x}$ is directed. Let $U, V \in \mathcal{I}_{x}$. Then
$x \in (U \cup V)^{c}$, and so $(U \cup V)^{c}$ is a clopen downset
neighborhood of $x$. By condition (4) of Definition \ref{Priestley
structure}, there is a clopen downset $W$ such that $W \cap X_{0}$
is cofinal in $W$ and $x \in W \subseteq (U \cup V)^{c}$. Therefore,
$x \not \in W^{c} \in X^{*}$, so $W^{c} \in \mathcal{I}_{x}$, and
$U\cup V\subseteq W^c$. Thus, $\mathcal{I}_{x}$ is directed. Now
suppose that $\mathcal{I}_{x}$ is directed. In order to show that $x
\in X_{0}$, let $V$ be a clopen downset neighborhood of $x$. Then
$V^{c}$ is a clopen upset and $x \not \in V^{c}$. Since in the proof
of Proposition \ref{CU-finun} condition (4) of Definition
\ref{g-Priestley} is not used, we can use Proposition \ref{CU-finun}
here, by which there exist $U_{1}, \ldots, U_{n} \in X^{*}$ such
that $V^{c} = U_{1}\cup \ldots \cup U_{n}$. Therefore, $U_{1},
\ldots, U_{n} \in \mathcal{I}_{x}$. Since $\mathcal{I}_{x}$ is
directed, there exists $U \in \mathcal{I}_{x}$ such that $U_{1}\cup
\ldots \cup U_{n} \subseteq U$. Thus, $x \in U^{c} \subseteq V$ and
$U^{c}$ is a clopen downset of $X$ with $U^{c} \cap X_{0}$ cofinal
in $U^{c}$. This, by condition (4) of Definition \ref{Priestley
structure}, implies that $x\in X_0$. Consequently, $x\in X_0$ iff
$\mathcal{I}_x$ is directed, and so $X$ satisfies condition (4) of
Definition \ref{g-Priestley}.

Now let $X$ be a generalized Priestley space. First we show that $X$
satisfies condition (3) of Definition \ref{Priestley structure}.
Suppose $x, y \in X$ are such that $x \not \leq y$. By condition (5)
of definition \ref{g-Priestley}, there is $U \in X^{*}$ such that $x
\in U$ and $y \not \in U$. Since $U^c={\downarrow}(X_{0} - U)$,
there is $z \in X_{0}$ such that $y \leq z$ and $z \not \in U$. As
$U$ is an upset, it follows that $x \not \leq z$. Thus, there is
$z\in X_0$ such that $x\not\leq z$ and $y\leq z$, and so $X$
satisfies condition (3) of Definition \ref{Priestley structure}.
Finally, we show that $X$ satisfies condition (4) of Definition
\ref{Priestley structure}. First suppose that $x \in X_{0}$. By
condition (4) of Definition \ref{g-Priestley}, $\mathcal{I}_{x}$ is
directed. Let $V$ be a clopen downset neighborhood of $x$. Then $x
\not \in V^{c}$ and $V^{c}$ is a clopen upset of $X$. By Proposition
\ref{CU-finun}, $V^c$ is a finite union of elements of
$\mathcal{I}_{x}$. Since $\mathcal{I}_{x}$ is directed, there is $U
\in \mathcal{I}_{x}$ such that $V^{c} \subseteq U$. Therefore, $x
\in U^{c} \subseteq V$ and $U^{c} \cap X_{0}$ is cofinal in $U^{c}$.
Thus, the family of clopen downsets $U$ that contain $x$ and have
the property that $U \cap X_{0}$ is cofinal in $U$ form a basis of
clopen downset neighborhoods of $x$. Conversely, suppose that for
each clopen downset neighborhood $V$ of $x$ there is a clopen
downset $W$ of $X$ such that $W \cap X_{0}$ is cofinal in $W$ and $x
\in W \subseteq V$. We show that $\mathcal{I}_{x}$ is directed.
Suppose that $U,W \in \mathcal{I}_{x}$. Then $x \in (U \cup W)^{c}$
and $(U \cup W)^{c}$ is a clopen downset of $X$. Therefore, there is
a clopen downset $V$ of $X$ such that $V \cap X_{0}$ is cofinal in
$V$ and $x \in V \subseteq (U \cup W)^{c}$. Then $x \not \in V^{c}
\in X^{*}$. Thus, $V^{c} \in \mathcal{I}_{x}$ and  $U \cup W
\subseteq V^{c}$, so $\mathcal{I}_x$ is directed. This, by condition
(4) of Definition \ref{g-Priestley}, implies that $x \in X_{0}$. It
follows that $X$ satisfies condition (4) of Definition
\ref{Priestley structure}.
\end{proof}

Consequently, Hansoul spaces are exactly our generalized Priestley
spaces. In \cite{Han03} Hansoul provided duality for the category whose objects are bounded join semi-lattices and whose morphisms correspond to our sup-homomorphisms. However, he was unable to extend his duality to bounded join semi-lattice homomorphisms. Thus, our work can be viewed as a completion of Hansoul's work.

\subsection{Comparison with the work of K\"ohler}

The first duality between finite implicative meet semi-lattices and
finite partially ordered sets was developed by K\"ohler
\cite{Koh81}.\footnote{K\"ohler called implicative meet
semi-lattices \emph{Brouwerian semi-lattices}.} Let
$\mathsf{IM_{fin}}=\mathsf{BIM_{fin}}$ denote the category of finite
implicative meet semi-lattices (which are always bounded) and
implicative meet semi-lattice homomorphisms. K\"ohler established
that $\mathsf{IM_{fin}}$ is dually equivalent to the category of
finite partially ordered sets and partial functions $f:X\to Y$
satisfying the following two conditions:
\begin{enumerate}
\item If $x, y \in \mathrm{dom}(f)$ and $x < y$, then $f(x) < f(y)$.
\item If $x \in \mathrm{dom}(f)$, $y \in Y$, and $f(x) < y$, then
there is $z \in \mathrm{dom}(f)$ such that $x < z$ and $f(z) =
y$.\footnote{K\"oler developed his duality by working with downsets
of posets. Since we prefer to work with upsets instead of downsets,
we rephrased K\"ohler's conditions using the dual of his order.}
\end{enumerate}
We call such functions \textit{partial K\"ohler functions}, and
denote the category of finite posets and partial K\"ohler functions
by $\mathsf{POS_{fin}^K}$.

Now we show how the dual equivalence of $\mathsf{HA}^{\wedge,\to}$
and $\mathsf{ES}^{\mathsf{pf}}$ restricted to the finite case
implies the K\"ohler duality. Let $\mathsf{ES_{fin}^{pf}}$ denote
the category of finite Esakia spaces and partial Esakia functions.
Since each finite implicative meet-semilattice is in fact a Heyting
algebra, it follows from Corollary \ref{cor-parfun} that
$\mathsf{IM_{fin}}$ is dually equivalent to
$\mathsf{ES_{fin}^{pf}}$. But the objects of
$\mathsf{ES_{fin}^{pf}}$ are simply finite partially ordered sets.
For finite partially ordered sets $X$ and $Y$, a partial function
$f:X\to Y$ is a partial Esakia function iff $f$ is a partial
K\"ohler function. Therefore, the categories
$\mathsf{ES_{fin}^{pf}}$ and $\mathsf{POS_{fin}^K}$ are isomorphic.
Thus, K\"ohler's theorem that $\mathsf{IM_{fin}}$ is dually
equivalent to $\mathsf{POS_{fin}^K}$ is an easy consequence of the
fact that $\mathsf{HA}^{\wedge,\to}$ is dually equivalent to
$\mathsf{ES}^{\mathsf{pf}}$. K\"ohler was unable to extend his
duality to the infinite case. Thus, our work can be viewed as a
completion of K\"ohler's work.

\subsection{Comparison with the work of Celani and Vrancken-Mawet}

In \cite{Cel03a} Celani showed that a restricted version of his
duality of \cite{Cel03b} between $\mathsf{DM}$ and $\mathsf{DS}$
yields duality between the subcategory $\mathsf{IM}$ of
$\mathsf{DM}$ of implicative meet semi-lattices and implicative meet
semi-lattice homomorphisms and the subcategory $\mathsf{IS}$ of
$\mathsf{DS}$ of IS-spaces and functional meet-relations. We recall
from \cite{Cel03a} that an \emph{IS-space} is a DS-space $X$ such
that for each $U,V\in D_X$ we have $[{\downarrow}(U - V)]^{c}\in
D_X$, and that a \emph{functional meet-relation} is a meet-relation
$R$ between two IS-spaces $X$ and $Y$ such that for each  $x \in X$
and $y \in Y$, from $xRy$ it follows that there exists $z \in X$
with $x \leq z$ and $R[z] = {\uparrow}y$.\footnote{Warning: Do not
confuse Celani's notion of functional meet-relation with our notion
of functional generalized Priestley morphism!}

We consider the subcategory $\mathsf{BIM}$ of $\mathsf{IM}$ of
bounded implicative meet semi-lattices, and the subcategory
$\mathsf{CIS}$ of $\mathsf{IS}$ of compact IS-spaces. Then
$\mathsf{BIM}$ is a subcategory of $\mathsf{BDM}$, $\mathsf{CIS}$ is
a subcategory of $\mathsf{CDS}$, and $\mathsf{BIM}$ is dually
equivalent to $\mathsf{CIS}$. From this and our duality for
$\mathsf{BIM}$ it follows that $\mathsf{CIS}$ is equivalent to
$\mathsf{GES}$. We give an explicit construction of a compact
IS-space from a generalized Esakia space and of a functional
meet-relation from a generalized Esakia morphism.

\begin{proposition}
Let $X = \la X, \tau, \leq, X_0\ra$ be a generalized Esakia space
and let $X_0 = \la X_0, \tau_0, \leq_0 \ra$ be the space constructed
in Proposition \ref{prop-celani}. Then $X_0$ is a compact IS-space.
\end{proposition}

\begin{proof}
It follows from Proposition \ref{prop-celani} that $X_0$ is a
compact DS-space. It is left to be shown that if $U,V\in D_{X_0}$,
then $X_0-{\downarrow}_0(U - V)\in D_{X_0}$. Since $U,V\in D_X$,
there exist $U',V'\in X^*$ such that $U=U'\cap X_0$ and $V=V'\cap
X_0$. We show that $X_0-{\downarrow}_0(U - V) = X_0-{\downarrow}(U'
- V')$. We have $x\in X_0-{\downarrow}_0(U - V)$ iff $(\forall y\in
X_0)(x\leq_0 y$ and $y\in U\Rightarrow y\in V)$, and $x\in
X_0-{\downarrow}(U' - V')$ iff $(\forall y\in X)(x\leq y$ and $y\in
U'\Rightarrow y\in V')$. Therefore, it is clear that $x\in
X_0-{\downarrow}(U' - V')$ implies $x\in X_0-{\downarrow}_0(U - V)$.
Conversely, suppose that $x\in X_0-{\downarrow}_0(U - V)$. If there
exists $y\in X$ such that $x\leq y$, $y\in U'$ and $y\notin V'$,
then there exists $z\in \mathrm{max}(X-V')$ such that $y\leq z$. But
$\mathrm{max}(X-V') \subseteq X_0$. Thus, there exists $z\in X_0$
such that $x\leq_0 z$, $z\in U$, and $z\notin V$, a contradiction.
Consequently, $x\in X_0-{\downarrow}(U' - V')$, and so
$X_0-{\downarrow}_0(U - V) = X_0-{\downarrow}(U' - V')$. Since $X$
is a generalized Esakia space, $X-{\downarrow}(U' - V')\in X^*$.
Moreover, $X_0-{\downarrow}_0(U - V) = X_0-{\downarrow}(U' -
V')=(X-{\downarrow}(U' - V'))\cap X_0$. So $X_0-{\downarrow}_0(U -
V)\in D_{X_0}$, and so $X_0$ is a compact IS-space.
\end{proof}

\begin{proposition}
If $X$ and $Y$ are generalized Esakia spaces and $R \subseteq X
\times Y$ is a generalized Esakia morphism, then $R_0 = R \cap (X_0
\times Y_{0})$ is a functional meet-relation between the compact
IS-spaces $X_0$ and $Y_0$.
\end{proposition}

\begin{proof}
It follows from Proposition \ref{prop-celani1} that $R_0$ is a
meet-relation. We show that $R_0$ is functional. Let $x\in X_0$,
$y\in Y_0$, and $xR_0y$. Then $xRy$ and since $y\in Y_0$ and $R$ is
a generalized Esakia morphism, there exists $z\in X_0$ such that
$x\leq z$ and $R[z]={\uparrow}y$. Thus, $x\leq_0 z$ and $R_0[z] =
{\uparrow}_0 y$, and so $R_0$ is a functional meet-relation.
\end{proof}

Celani \cite[Theorem 4.11]{Cel03a} also claimed erroneously that
functional meet-relations between IS-spaces $X$ and $Y$ are in a 1-1
correspondence with partial maps $f:X\to Y$, he calls
\emph{IS-morphisms}, that satisfy the following properties: (i) If
$x, y \in \mathrm{dom}(f)$ and $x \leq y$, then $f(x) \leq f(y)$,
(ii) if $x \in \mathrm{dom}(f)$, $y\in Y$, and $f(x) \leq y$, then
there exists $z \in \mathrm{dom}(f)$ such that $x \leq z$ and
$f(z)=y$, and (iii) if $U\in D_Y$, then
$[{\downarrow}f^{-1}(U^c)]^{c}\in D_X$. For a functional
meet-relation $R$ between two IS-spaces $X$ and $Y$, he constructed
the IS-morphism $f_R$ as follows. He set $\mathrm{dom}(f_R)=\{z\in
X: (\exists x\in X)(\exists y\in Y)(x\leq z$ and
$R[z]={\uparrow}y\}$, and for $z\in\mathrm{dom}(f_R)$ he set
$f_R(z)=y$. For an IS-morphism $f$ between two IS-spaces $X$ and
$Y$, he constructed the functional meet-relation $R_f\subseteq
X\times Y$ by $xR_fy$ iff $(\exists z\in\mathrm{dom}(f))(x\leq z$
and $f(z)=y)$. However, his claim that the maps $R\mapsto f_R$ and
$f\mapsto R_f$ establish a 1-1 correspondence between functional
meet-relations and IS-morphisms between IS-spaces is false as the
following simple example shows.

\begin{example}\label{Celani}

Consider the IS-spaces $X$ and $Y$ and the IS-morphism $f:X\to Y$
shown in Fig.8. The functional meet-relation $R_f\subseteq X\times
Y$ corresponding to $f$, together with the IS-morphism $f_{R_f}:X\to
Y$ corresponding to $R_f$ are also shown in Fig.8. 

\

\begin{center}
\begin{tikzpicture}[scale=.8,inner sep=.5mm]
\node[bull] (x) at (0,0) {};%
\node[bull] (b) at (0,1) {};%
\node[bull] (y) at (2,1) {};%
\node at (1,1.5) {$f$};%
\node at (0,-1) {$X$};%
\node at (2,-1) {$Y$};%
\draw (x) -- (b);%
\draw[color=brown,->] (b) -- (y);
\end{tikzpicture}\qquad\qquad\qquad
\begin{tikzpicture}[scale=.8,inner sep=.5mm]
\node[bull] (x) at (0,0) {};%
\node[bull] (b) at (0,1) {};%
\node[bull] (y) at (2,1) {};%
\node at (1,1.5) {$R_f$};%
\node at (0,-1) {$X$};%
\node at (2,-1) {$Y$};%
\draw (x) -- (b);%
\draw[color=brown,->] (x) -- (y);%
\draw[color=brown,->] (b) -- (y);
\end{tikzpicture}\qquad\qquad\qquad
\begin{tikzpicture}[scale=.8,inner sep=.5mm]
\node[bull] (x) at (0,0) {};%
\node[bull] (b) at (0,1) {};%
\node[bull] (y) at (2,1) {};%
\node at (1,1.5) {$f_{R_f}$};%
\node at (0,-1) {$X$};%
\node at (2,-1) {$Y$};%
\draw (x) -- (b);%
\draw[color=brown,->] (x) -- (y);%
\draw[color=brown,->] (b) -- (y);
\end{tikzpicture}

\

Fig.8

\end{center}

Clearly $f\neq
f_{R_f}$, thus the maps $R\mapsto f_R$ and $f\mapsto R_f$ do not
establish a 1-1 correspondence between functional meet-relations and
IS-morphisms between IS-spaces.
\end{example}

In order to obtain a characterization of Celani's functional
meet-relations by means of partial functions, we need to
strengthen the notion of an IS-morphism. For two IS-spaces $X$ and
$Y$, we call a map $f:X \to Y$ a \emph{strong IS-morphism}
(\emph{SIS-morphism} for short) if it is an IS-morphism and
satisfies the two additional conditions:
\begin{enumerate}
\item[$(*)$] For each $x \in X$, $x \in \mathrm{dom}(f)$ iff there exists
$y \in Y$ such that $f[{\uparrow}x]={\uparrow}y $.
\item[$(**)$] For each $x \in X$ and each $y \in Y$, if $y \not \in f[{\uparrow}x]$,
then there is $U \in D_{Y}$ such that $f[{\uparrow}x] \subseteq U$
and $y \notin U$.
\end{enumerate}

We recall that $\mathsf{IS}$ is the category of IS-spaces and
functional meet-relations. Let $\mathsf{SIS}$ denote the category
of IS-spaces and SIS-morphisms. We show that $\mathsf{IS}$ is
isomorphic to $\mathsf{SIS}$.

\begin{lemma}\label{lem:11.9}
Let $X$ and $Y$ be two DS-spaces and let $R \subseteq  X \times Y$
be a meet-relation. Then $x\leq y$ and $yRu$ imply $xRu$, and
$xRu$ and $u\leq v$ imply $xRv$.
\end{lemma}

\begin{proof}
Let $x\leq yRu$. If $x\nR u$, then by condition (ii) of the
definition of a meet-relation, there exists $U\in D_Y$ such that
$R[x]\subseteq U$ and $u\notin U$. Therefore, $x\in\Box_R U$. By
condition (i) of the definition of a meet-relation, $\Box_R U$ is
an upset. Thus, $x\leq y$ implies $y\in\Box_R U$. But then
$R[y]\subseteq U$, and since $yRu$, it follows that $u\in U$. The
obtained contradiction proves that $xRu$.

Let $xRu\leq v$. If $x\nR v$, then by condition (ii) of the
definition of a meet-relation, there exists $U\in D_Y$ such that
$R[x]\subseteq U$ and $v\notin U$. From $xRu$ it follows that
$u\in U$, and from $u\leq v$, it follows that $u\notin U$. The
obtained contradiction proves that $xRv$.
\end{proof}

Let $X$ and $Y$ be two IS-spaces and let $R \subseteq  X \times Y$
be a functional meet-relation. Similar to Celani \cite[Theorem
4.11]{Cel03a}, we define a partial function $f_R:X\to Y$ as
follows. We set $\mathrm{dom}(f_R)=\{x\in X:R[x]={\uparrow}y\}$
and for $x\in \mathrm{dom}(f_R)$ we set $f_R(x)=y$.

\begin{proposition}\label{prop:11.10}
Let $X$ and $Y$ be two IS-spaces and let $R \subseteq  X \times Y$
be a functional meet-relation. Then $f_{R}$ is a SIS-morphism.
\end{proposition}

\begin{proof}
First we show that $f_{R}$ is an IS-morphism. Let $x, y \in
\mathrm{dom}(f_{R})$ and $x \leq y$. Then $R[x] =
{\uparrow}f_{R}(x)$ and $R[y] = {\uparrow}f_{R}(y)$. By Lemma
\ref{lem:11.9}, $x\leq y$ implies $R[y] \subseteq R[x]$.
Therefore, $f_{R}(x) \leq f_{R}(y)$, and so condition (i) of the
definition of an IS-morphism is satisfied. To see that condition
(ii) is also satisfied, let $x \in \mathrm{dom}({f_{R}})$, $y \in
Y$, and $f_{R}(x) \leq y$. Then $xRy$. Since $R$ is a functional
meet-relation, there is $z \in X$ such that $x \leq z$ and $R[z] =
{\uparrow} y$. Therefore, $z \in \mathrm{dom}(f_{R})$, $x \leq z$,
and $f_{R}(z) = y$. Finally, to see that condition (iii) of the
definition of an IS-morphism is satisfied as well, let $U \in
D_{Y}$. We show that $\Box_{R}U =
[{\downarrow}f^{-1}(U^{c})]^{c}$. We have $x\in\Box_R U$ iff $R[x]
\subseteq U$ and $x\in[{\downarrow}f_R^{-1}(U^c)]^{c}$ iff
$(\forall z\in\mathrm{dom}(f_R))(x\leq z\Rightarrow f_R(z)\in U)$.
Let $x\in\Box_R U$, $z\in\mathrm{dom}(f_R)$, and $x\leq z$. Then
$R[x] \subseteq U$, $R[z]={\uparrow}f_R(z)$, and $R[z]\subseteq
R[x]$. Therefore, ${\uparrow}f_R(z)\subseteq U$, so $f_R(z)\in U$,
and so $x\in[{\downarrow}f_R^{-1}(U^c)]^{c}$. Now let $x\in
[{\downarrow}f_R^{-1}(U^c)]^{c}$ and $xRy$. Since $R$ is a
functional meet-relation, there exists $z\in X$ such that $x\leq
z$ and $R[z]= {\uparrow}y$. Then $z\in\mathrm{dom}(f_R)$, $x\leq
z$, and $f_R(z)=y$. Therefore, $f_R(z)\in U$, so $y\in U$, and so
$x\in \Box_R U$. Thus, $\Box_R U=[{\downarrow}f_R^{-1}(U^c)]^{c}$.
Now since $R$ is a meet-relation, $\Box_{R}U \in D_{X}$, so
$[{\downarrow}f_R^{-1}(U^c)]^{c}\in D_X$, and so $f_R$ is an
IS-morphism.

Next we show that $f_{R}$ satisfies condition $(*)$. Suppose that
$x \in \mathrm{dom}(f_{R})$. Then $f_{R}[{\uparrow}x] = R[x] =
{\uparrow}f_{R}(x)$. Therefore, there exists $y\in Y$ ($y=f_R(x)$)
such that $f_{R}[{\uparrow}x]={\uparrow}y$. Now suppose that
$f_{R}[{\uparrow}x] = {\uparrow}y$ for some $y\in Y$. Since
$f_{R}[{\uparrow}x] = R[x]$, we have $R[x] = {\uparrow} y$.
Therefore, $x \in \mathrm{dom}(f_{R})$ by the definition of $f_R$.
Finally, we show that $f_R$ satisfies condition $(**)$. Let $x \in
X$, $y \in Y$, and $y \notin f_R[{\uparrow}x]$. Since
$f_{R}[{\uparrow}x] = R[x]$ and $R$ is a meet-relation, there is
$U \in D_{Y}$ such that $R[x] \subseteq U$ and $y \notin U$. This
implies that $f_{R}[{\uparrow}x] \subseteq U$ and $y \notin U$.
Consequently, $f_R$ is a SIS-morphism.
\end{proof}

Let $X$ and $Y$ be IS-spaces and $f:X\to Y$ be a SIS-morphism.
Following Celani \cite[Theorem 4.11]{Cel03a}, we define
$R_f\subseteq X\times Y$ by
$$xR_fy \mbox{ iff there exists } z\in\mathrm{dom}(f) \mbox{ such
that } x\leq z \mbox{ and } f(z)=y.$$

\begin{proposition}\label{prop:11.11}
Let $X$ and $Y$ be two IS-spaces and let $f: X \to Y$ be a
SIS-morphism. Then $R_{f}$ is a functional meet-relation.
\end{proposition}

\begin{proof}
First we show that $R_{f}$ is a meet-relation. Let $U\in D_Y$. We
show that $\Box_{R_f}U= [{\downarrow}f^{-1}(U^c)]^{c}$. We have
$x\in\Box_{R_f} U$ iff $R_f[x] \subseteq U$ and
$x\in[{\downarrow}f^{-1}(U^c)]^{c}$ iff $(\forall
z\in\mathrm{dom}(f))(x\leq z\Rightarrow f(z)\in U)$. Let
$x\in\Box_{R_f} U$, $z\in\mathrm{dom}(f)$, and $x\leq z$. Then
$xR_f f(z)$, so $f(z)\in U$, and so
$x\in[{\downarrow}f^{-1}(U^c)]^{c}$. Now let $x\in
[{\downarrow}f^{-1}(U^c)]^{c}$ and $xR_fy$. Then there exists
$z\in\mathrm{dom}(f)$ such that $x\leq z$ and $f(z)=y$. Therefore,
$f(z)\in U$, so $y\in U$, and so $x\in \Box_{R_f} U$. Thus,
$\Box_{R_f} U=[{\downarrow}f^{-1}(U^c)]^{c}$. By (iii) of the
definition of an IS-morphism, $[{\downarrow}f^{-1}(U^c)]^{c}\in
D_X$. So $\Box_{R_f} U\in D_X$, and so $R_f$ satisfies condition
(i) of the definition of a meet-relation. By the definition of
$R_{f}$, we have that $R_{f}[x] = f[{\uparrow} x]$ for each $x\in
X$. Therefore, by condition $(**)$, $R_{f}$ satisfies condition
(ii) of the definition of a meet-relation, and so $R_f$ is a
meet-relation.

In order to show that $R_f$ is functional, we observe that if $z
\in \mathrm{dom}(f)$, then $R_f[z]={\uparrow}f(z)$. To see this,
let $zR_fu$. Then there exists $v\in\mathrm{dom}(f)$ such that
$z\leq v$ and $f(v)=u$. By condition (i) of the definition of an
IS-morphism, $f(z)\leq f(v)=u$. Conversely, if $f(z)\leq u$, then
by condition (ii) of the definition of an IS-morphism, there
exists $v\in\mathrm{dom}(f)$ such that $z\leq v$ and $f(v)=u$.
Thus, $zR_fu$, and so $R_f[z]={\uparrow}f(z)$. Now let $xR_fy$.
Then there exists $z\in\mathrm{dom}(f)$ such that $x\leq z$ and
$f(z)=y$. Therefore, $R_f[z]={\uparrow}f(z)$. Thus, $xR_fy$
implies that there exists $z\in X$ such that $x\leq z$ and
$R_f[z]={\uparrow}y$, and so $R_f$ is a functional meet-relation.
\end{proof}

Now we are in a position to prove that $\mathsf{IS}$ and
$\mathsf{SIS}$ are isomorphic.

\begin{theorem}
The categories $\mathsf{IS}$ and $\mathsf{SIS}$ are isomorphic.
\end{theorem}

\begin{proof}
Define two functors $\Phi:\mathsf{IS}\to\mathsf{SIS}$ and
$\Psi:\mathsf{SIS}\to\mathsf{IS}$ as follows. For an IS-space $X$
set $\Phi(X)=X=\Psi(X)$, for a functional meet relation $R$ set
$\Phi(R)=f_R$, and for a SIS-morphism $f$ set $\Psi(f)=f_R$. Then
it follows from Propositions \ref{prop:11.10} and \ref{prop:11.11}
that both $\Phi$ and $\Psi$ are well-defined. It is left to be
shown that $R = R_{f_{R}}$ and $f = f_{R_{f}}$.

First we show that $R = R_{f_{R}}$. For $x\in X$ and $y\in Y$, we
have $xRy$ iff there exists $z\in X$ such that $x\leq z$ and
$R[z]= {\uparrow}y$. On the other hand, $xR_{f_R}y$ iff there
exists $z\in \mathrm{dom}(f_R)$ such that $x\leq z$ and
$f_R(z)=y$. But $R[z]= {\uparrow}f_R(z)$, so $R[z]= {\uparrow}y$
iff $f_R(z)=y$. It follows that $xR_{f_R}y$ iff there exists $z\in
X$ such that $x\leq z$ and $R[z]= {\uparrow}y$. Thus, $xRy$ iff
$xR_{f_R}y$, and so $R=R_{f_R}$. Next we show that $f =
f_{R_{f}}$. Suppose that $x \in \mathrm{dom}(f)$. Then $R_{f}[x] =
{\uparrow}f(x)$. Therefore, $f(x) \in \mathrm{dom}(f_{R_{f}})$ and
$f_{R_{f}}(x) = f(x)$. Now suppose that $x \in
\mathrm{dom}(f_{R_{f}})$. Let $f_{R_f}(x) = y$. Then $R_{f}[x] =
{\uparrow}y$. Since $R_{f}[x] = f[{\uparrow} x]$, we have
$f[{\uparrow} x] = {\uparrow}y$. By condition $(*)$, $x \in
\mathrm{dom}(f)$. Thus, $f(x) = f_{R_{f}}(x)$, so
$\mathrm{dom}(f)=\mathrm{dom}(f_{R_f})$ and whenever
$x\in\mathrm{dom}(f)=\mathrm{dom}(f_{R_f})$, then
$f(x)=f_{R_f}(x)$. Consequently, $f=f_{R_f}$, and so $\mathsf{IS}$
is isomorphic to $\mathsf{SIS}$.
\end{proof}

As a result, we obtain that $\mathsf{IM}$ is also dually
equivalent to $\mathsf{SIS}$. Thus, we can represent dually
implicative meet semi-lattice homomorphisms by either functional
meet-relations or by strong IS-morphisms between IS-spaces.

We conclude our comparison section by looking at the work of
Vrancken-Mawet \cite{VM86}. We recall that Vrancken-Mawet
\cite{VM86} constructed a dual category of $\mathsf{BIM}$, which
is similar to that of Celani \cite{Cel03a}. The objects of the
Vrancken-Mawet category are essentially Celani's compact IS-spaces
with the only difference that Vrancken-Mawet works with the
specialization order, while Celani prefers to work with the dual
of the specialization order. Because of this, Celani works with
the bounded implicative meet semi-lattice of complements of
compact open subsets of a compact IS-space $X$, while
Vrancken-Mawet works with the bounded implicative meet
semi-lattice of compact open subsets of $X$. Both these meet
semi-lattices are isomorphic though because Celani's order is dual
to Vrancken-Mawet's order. In order to simplify our comparison, we
work with Celani's compact IS-spaces instead of Vrancken-Mawet's
spaces and adjust the definition of Vrancken-Mawet morphisms
accordingly.
Note that unlike the case with objects, Vrancken-Mawet morphisms
satisfy an extra condition that Celani's morphisms do not satisfy.
Let $X$ and $Y$ be two compact IS-spaces and let $f:X\to Y$ be an
IS-morphism. We call $f$ a \emph{Vrancken-Mawet morphism} if it
satisfies the following additional condition:
\begin{enumerate}
\item[(VM)] $x\in \mathrm{dom}(f)$ iff $(\exists y\in Y)(\forall U\in
D_Y)(y\notin U \Leftrightarrow (\exists z\in\mathrm{dom}(f))(x\leq
z$ and $f(z)\notin U)$.
\end{enumerate}
Observe that in Example \ref{Celani}, the $f$ we start with does
not satisfy (VM), and so it is not a Vrancken-Mawet morphism. In
fact, given two compact IS-spaces $X$ and $Y$, it is relatively
easy to verify that our conditions $(*)$ and $(**)$ imply (VM),
and that (VM) implies $(*)$. It requires more work to show that
(VM) also implies $(**)$. We skip the details and only mention
that the upshot of these observations is that the category
$\mathsf{CSIS}$ of compact IS-spaces and SIS-morphisms is
isomorphic to the category $\mathsf{VM}$ of compact IS-spaces and
Vrancken-Mawet morphisms. We feel that our conditions $(*)$ and
$(**)$ are more natural and easier to work with than the (VM)
condition.
%
%
%

\section*{Acknowledgements}

Many thanks to Mamuka Jibladze for drawing all the diagrams in the
paper. We are also thankful to Jonathan Farley for drawing our
attention to the Vrancken-Mawet paper \cite{VM86},
and to Sergio Celani for simplifying our original definition of a generalized Priestley morphism, as well as for pointing out that the usual set-theoretic composition of generalized Priestley morphisms may not be a generalized Priestley morphism. 

\bibliographystyle{amsplain}
\bibliography{MyBib}


\bigskip

\noindent Guram Bezhanishvili: Department of Mathematical Sciences, New Mexico State
University, Las Cruces NM 88003, USA. Email:
\email{guram@nmsu.edu}

\bigskip

\noindent Ramon Jansana: Dept. Filosofia,  
Universitat de Barcelona, Montalegre 6, 08001 Barcelona, Spain.
Email: \email{jansana@ub.edu}

\end{document}